\documentclass[11pt]{article}
\usepackage{a4wide}
\usepackage{amssymb}
\usepackage{hyperref}
\usepackage{amsbsy}
\usepackage{amscd}
\usepackage{amsmath}
\usepackage{amstext}
\usepackage{amsthm}
\usepackage{color,bm}
\usepackage{graphicx}

\usepackage{tabu}
\usepackage{listings}
\usepackage{subcaption}

\newcommand{\SIE}{\text{\rm SIE}}
\newcommand{\DIE}{\text{\rm DIE}}

\newcommand{\wcI}{\widehat\cI}
\newcommand{\qqand}{\qquad\text{and}\qquad}

\newcommand{\cKI}{\cK((\cI_q)_{q\ge0})}

\newcommand{\cR}{{\cal R}}
\newcommand{\cI}{{\cal I}}
\newcommand{\ka}{\kappa}
\newcommand{\ra}{R_{\alpha}}
\newcommand{\De}{\Delta}
\newcommand{\ma}{\beta_\alpha}

\renewcommand{\r}{\rho}
\newcommand{\La}{\Lambda}

\newcommand{\g}{\gamma}

\newcommand{\cM}{\mathcal{M}}
\newcommand{\cC}{\mathcal{C}}
\newcommand{\cV}{\mathcal{V}}

\newcommand{\vol}{\operatorname{vol}}
\newcommand{\vv}[1]{\mathbf{#1}}
\newcommand{\Z}{\mathbb{Z}}
\newcommand{\Q}{\mathbb{Q}}
\newcommand{\R}{\mathbb{R}}

\newcommand{\N}{\mathbb{N}}
\newcommand{\ZZ}{\mathbb{Z}}
\newcommand{\QQ}{\mathbb{Q}}
\newcommand{\RR}{\mathbb{R}}

\newcommand{\NN}{\mathbb{N}}
\newcommand{\I}{{\rm I}}
\newcommand{\Rp}{\R^+}

\newcommand{\vx}{\mathbf{x}}

\newcommand{\vy}{\mathbf{y}}

\newcommand{\cA}{\mathbf{A}}
\newcommand{\A}{\mathcal{A}}
\newcommand{\cB}{\mathbf{B}}

\newcommand{\p}{\psi}

\newcommand{\cH}{{\cal H}}

\theoremstyle{plain}
\newtheorem{thm}{Theorem}[section]

\newtheorem{cor}{Corollary}[section]
\newtheorem{lem}{Lemma}[section]
\newtheorem{prop}{Proposition}[section]
\newtheorem{thlittle}{Littlewood's Conjecture \!\!\!}

\newtheorem{thschmidt}{Schmidt's Conjecture \!\!\!}

\newtheorem{conjecture}{Conjecture}[section]

\theoremstyle{remark}
\newtheorem{rem}{Remark}[section]

\theoremstyle{definition}
\newtheorem{eg}{Example}[section]
\newtheorem{df}{Definition}[section]

\newcommand{\ds}{\displaystyle}

\def\a{\alpha}
\def\b{\beta}
\def\q{\mathbf{q}}

\renewcommand{\[}{ \left[ }

\newcommand{\ve}{\varepsilon}
\newcommand{\Bad}{\mathbf{Bad}}
\newcommand{\bad}{\mathbf{Bad}}
\newcommand{\mad}{\mathbf{Mad}}

\def\Jarnik{Jarn\'\i k}

 \numberwithin{equation}{section}

\newlength{\aaa}

\begin{document}


%
\centerline{\textbf{\huge\sf Metric  Diophantine Approximation:  }}
\centerline{\textbf{\huge\sf aspects of recent work 
}}

\vspace*{10ex}

\centerline{{\large\sf Victor Beresnevich, Felipe Ram\'{\i}rez and Sanju Velani   }}

\hfill\ {University of York}\hfill\

\vspace*{5ex}

\begin{abstract}
In these notes, we begin by recalling aspects of the classical theory of metric Diophantine approximation; such as theorems of Khintchine, \Jarnik, Duffin-Schaeffer  and Gallagher.  We then describe  recent strengthening of  various classical statements as well as recent developments in the area of Diophantine approximation on manifolds.  The latter includes the well approximable, the  badly approximable and the inhomogeneous aspects.
\end{abstract}

\vspace*{5ex}

{\footnotesize\sf
\noindent This Chapter is to be published by Cambridge University Press as part of a multi-volume work edited by Badziahin, D., Gorodnik, A., Peyerimhoff, N.

\noindent$\copyright$ in the Chapter, Victor Beresnevich, Felipe Ram\'{\i}rez, Sanju Velani, 2016

\noindent$\copyright$ in the Volume, Cambridge University Press, 201x

\noindent Cambridge University Press's catalogue entry for the Volume can be found at
{\tt www.cambridge.org}

\noindent NB: The copy of the Chapter, as displayed on this website, is a draft, pre-publication copy only. The final, published version of the Chapter shall be available for purchase from Cambridge University Press and other standard distribution channels as part of the wider, edited Volume. This draft copy is made available for personal use only.}

\

\noindent{\footnotesize FR is supported by EPSRC Programme Grant: EP/J018260/1.} \\
{\footnotesize VB and SV are supported in part by EPSRC Programme Grant: EP/J018260/1.}

\date{}

\vspace*{5ex}

\newpage
{\footnotesize
\tableofcontents
}

\thispagestyle{empty}\setcounter{page}{0}\newpage



\section{Background: Dirichlet and $\Bad$ \label{intro}}

\subsection{Dirichlet's Theorem and two important consequences}
\label{sec:dirichl-theor-two}

Diophantine approximation is a branch of number theory that can
loosely be described as a quantitative analysis of the density of the
rationals $\QQ $ in the reals $\RR$. Recall that to say that $\QQ$ is
dense in $\RR$ is to say that
\begin{quote}
  for any real number $x$ and $\epsilon > 0$ there exists a rational
  number $p/q$ ($q>0$) such that $\left| x - p/q \right| < \epsilon$.
\end{quote}
In other words, any real number can be approximated by a rational
number with any assigned degree of accuracy.  But how ``rapidly'' can
we approximate a given $x\in\RR$?
\begin{quote}
  Given $x \in \RR$ and $q \in \NN$, how small can we make $\epsilon$?
  Trivially we can take any $\epsilon > 1/2q$.
  Can we do better than $1/2q$?
\end{quote}
The following rational numbers all lie within
$1/({\rm denominator})^2$ of the circle constant $\pi=3.141\dots$:
\begin{equation}\label{eq:piapprox}
  \frac{3}{1}, \frac{22}{7}, \frac{333}{106}, \frac{355}{113}, \frac{103993}{33102}.
\end{equation}
This shows that, at least sometimes, the answer to the last question
is ``yes.'' A more complete answer is given by \emph{Dirichlet's theorem}, which is itself a simple consequence of the following
powerful fact.

\theoremstyle{plain} \newtheorem*{php}{Pigeonhole Principle\index{Pigeonhole Principle}}
\begin{php}
  If $n$ objects are placed in $m$ boxes and $ n > m $, then some box
  will contain at least two objects.
\end{php}

\begin{thm}[Dirichlet\index{Dirichlet's theorem}, 1842]\label{Dir}
  For any $x\in\RR$ and $N \in \mathbb{N}$, there exist
  $p,q \in \mathbb{Z}$ such that
  \begin{equation}\label{7.0}
    \left| x - \frac{p}{q} \right| < \frac{1}{qN}  \qquad \textrm{ and }  \qquad  1 \leq  q \leq N \, .
  \end{equation}
\end{thm}

The proof can be found in most elementary number theory books.
However, given the important consequences of the theorem and its
various hybrids, we have decided to include the proof.

\begin{proof}
  As usual, let $[x]:= \max\{n \in \Z : n \le x \}$ denote the integer
  part of the real number $x$ and let $\{x\} = x - [x]$ denote the
  fractional part of $x$. Note that for any $x\in\R$ we have that
  $0\le \{x\}<1$.

  Consider the $N+1$ numbers
  \begin{equation}\label{7.1}
    \{0x\}, \{x\}, \{2x\}, \dots, \{Nx\}
  \end{equation}
  in the unit interval $[0,1)$. Divide $[0,1)$ into $N$ equal
  semi-open subintervals as follows:
  \begin{equation}\label{7.1+}
    [0,1) = \bigcup_{u=0}^{N-1} I_u\quad\textrm{where}\quad I_u:=\left[\frac{u}{N},\frac{u+1}{N}\right), \quad u=0,1, \dots,
    N-1.
  \end{equation}
  Since the $N+1$ points \eqref{7.1} are situated in the $N$
  subintervals \eqref{7.1+}, the Pigeonhole principle guarantees that
  some subinterval contains at least two points, say
  $\{q_2x\},\{q_1x\}\in I_u$, where $0\le u\le N-1$ and $q_1,q_2\in\Z$
  with $0\le q_1 < q_2 \le N$. Since the length of $I_u$ is $N^{-1}$
  and $I_u$ is semi-open we have that
  \begin{equation}\label{7.2}
    |\{q_2x\}-\{q_1x\}|<\frac{1}{N}.
  \end{equation}
  We have that $q_ix=p_i + \{q_ix\}$ where $p_i=[q_ix]\in\mathbb{Z}$
  for $i=1,2$. Returning to \eqref{7.2} we get
  \begin{equation}\label{7vb4}
    |\{q_2x\}-\{q_1x\}| = |q_2x - p_2 - (q_1x - p_1)|   = |(q_2-q_1)x -
    (p_2-p_1)|.
  \end{equation}
  Now define $q=q_2-q_1\in\Z$ and $p=p_2-p_1\in\Z$. Since
  $0\le q_1,q_2\le N$ and $q_1<q_2$ we have that $1\le q\le N$. By
  (\ref{7.2}) and (\ref{7vb4}), we get
  \begin{gather*}
    |qx - p| < \frac{1}{N}
  \end{gather*}
  whence \eqref{7.0} readily follows.
\end{proof}

The following statement is an important consequence of Dirichlet's
Theorem.

\begin{thm}[Dirichlet, 1842]\label{thm:7.4}
  Let $x \in \mathbb{R}\setminus\Q$. Then there exist infinitely many
  integers $q,p$ such that $\gcd(p,q)=1$, $q>0$ and
  \begin{equation}\label{7.3}
    \left|x-\frac{p}{q}\right| < \frac{1}{q^2}.
  \end{equation}
\end{thm}

\medskip

\begin{rem}
  Theorem~\ref{thm:7.4} is true for all $x\in\RR$ if we remove the
  condition that $p$ and $q$ are coprime, that is, if we allow
  approximations by non-reduced rational fractions.
\end{rem}

\begin{proof}
  Observe that Theorem \ref{Dir} is valid with
  $\gcd(p,q)=1$. Otherwise $ p/q=p'/q'$ with $\gcd(p',q')=1$ and
  $0<q'<q\le N$ and $|x-p/q|=|x-p'/q'|<1/(qN)<1/(q'N)$.

  Suppose $x$ is irrational and that there are only finitely many
  rationals
  \begin{equation*}
    \frac{p_1}{q_1}, \frac{p_2}{q_2}, \dots, \frac{p_n}{q_n},
  \end{equation*}
  where $\gcd(p_i,q_i)=1$, $q_i>0$ and
  \begin{equation*}
    \left|x-\frac{p_i}{q_i}\right|<\frac{1}{q_i^2}
  \end{equation*}
  for all $i=1,2,\dots,n$. Since $x$ is irrational,
  $x-\frac{p_i}{q_i}\not=0$ for $i=1,\dots,n$. Then there exists
  $N \in \mathbb{N}$ such that
  \begin{equation*}
    \left| x - \frac{p_i}{q_i} \right| > \frac{1}{N} \qquad \text{for
      all } 1 \le i \le n.
  \end{equation*}
  By Theorem~\ref{Dir}, there exists a reduced fraction $\dfrac{p}{q}$
  such that
  \begin{equation*}
    \left| x - \frac{p}{q} \right| < \frac{1}{qN} \le \frac{1}{N}
    \qquad (1\le q \le N).
  \end{equation*}
  Therefore, $\frac{p}{q}\neq\frac{p_i}{q_i}$ for any $i$ but
  satisfies \eqref{7.3}. A contradiction.
\end{proof}

Theorem~\ref{thm:7.4} tells us in particular that the
list~(\ref{eq:piapprox}) of good rational approximations to $\pi$ is
not just a fluke. This list can be extended to an infinite sequence,
and furthermore, such a sequence of good approximations exists for
\emph{every} irrational number. (See \S\ref{sec:basics-cont-fract}.)

Another important consequence of Theorem~\ref{Dir} is
Theorem~\ref{thm:slv1}, below.  Unlike Theorem~\ref{thm:7.4}, the
significance of it is not so immediately clear.  However, it will
become apparent during the course of these notes that it is the key to
the two fundamental theorems of classical metric Diophantine
approximation; namely the theorems of Khintchine and \Jarnik.

First, some notational matters. Unless stated otherwise, given a set
$X \subset \R$, we will denote by $m(X)$ the $1$-dimensional Lebesgue
measure of $X$. And we will use $B(x,r)$ to denote
$(x-r,x+r)\subset\RR$, the ball around $x\in\RR$ of radius $r>0$.

\begin{thm}\label{thm:slv1} Let $[a,b]\subset \R$ be an interval and
  $k \geq 6 $ be an integer. Then
  \begin{equation*}
    m\left( [a,b] \cap  \!\! \bigcup_{k^{n-1} < q \le k^n } \bigcup_{p\in\ZZ}
      \ \textstyle{B\left(\frac{p}{q}, \frac{k}{k^{2n}} \right) } \right)
    \ \geq \ \mbox{\large $\frac{1}{2} $ } (b-a).
  \end{equation*}
  for all sufficiently large $n\in\NN$.
\end{thm}

\begin{proof}
  By Dirichlet's theorem, for any $x\in I:=[a,b]$ there are coprime
  integers $p,q$ with $1\le q\le k^n$ satisfying
  $|x-p/q|<(qk^n)^{-1}$.  We therefore have that
  \begin{multline*}
    m(I) = m\left(I \cap \bigcup_{q\leq k^n}\bigcup_{p\in\ZZ}B \Big(
      \frac{p}{q},
      \frac{1}{qk^n} \Big) \right) \\
    \leq m\left(I \cap \bigcup_{q\leq
        k^{n-1}}\bigcup_{p\in\ZZ}B\Big(\frac{p}{q},
      \frac{1}{qk^n}\Big) \right) + m\left(I \cap \bigcup_{k^{n-1}<
        q\leq k^n}\bigcup_{p\in\ZZ}B \Big( \frac{p}{q},
      \frac{k}{k^{2n}}\Big) \right).
  \end{multline*}
  Also, notice that
  \begin{multline*}
    m\left(I \cap \bigcup_{q\leq
        k^{n-1}}\bigcup_{p\in\ZZ}B\Big(\frac{p}{q},
      \frac{1}{qk^n}\Big)\right) = m\left(I \cap \bigcup_{q \le
        k^{n-1} } \bigcup_{p=aq-1}^{bq+1}
      B\Big(\frac{p}{q},\frac{1}{q k^n} \Big) \right)\\
    \leq 2 \sum_{q \leq k^{n-1}} \frac{1}{qk^{n}} \Big(m(I)q+3\Big)
    \leq \frac{3}{k}m(I)
  \end{multline*}
  for large $n$. It follows that for $k\geq 6$,
  \begin{equation*}
    m\left( I \cap  \!\! \bigcup_{k^{n-1} < q \le k^n } \bigcup_{p\in\ZZ}
      \ \textstyle{B\left(\frac{p}{q}, \frac{k}{k^{2n}} \right) } \right)
    \ \geq \ m(I) - \mbox{\large $\frac{3}{k} $ }m(I) \ \geq \
    \mbox{\large $\frac{1}{2} $ } m(I)
  \end{equation*}
  for large $n$.
\end{proof}

\subsection{Basics of continued fractions}
\label{sec:basics-cont-fract}

From Dirichlet's theorem we know that for any real number $x$ there
are infinitely many `good' rational approximates $p/q$, but how can we
find these? The theory of continued fraction provides a simple
mechanism for generating them. We collect some basic facts about
continued fractions in this section. For proofs and a more
comprehensive account see for example \cite{HW, Khintchine, RS}.

Let $x$ be an irrational number and let $ [a_0;a_1,a_2,a_3, \ldots] $
denote its continued fraction expansion. Denote its $n$-th convergent
by
\begin{equation*}
  \frac{p_n}{q_n}  :=   [a_0;a_1,a_2,a_3, \ldots,a_n].
\end{equation*}
Recall that the convergents can be obtained by the following recursion
\begin{align*}
  p_0 &= a_0,  &q_0 &= 1, \\
  p_1 &= a_1a_0 + 1, &q_1 &= a_1,\\
  p_k &= a_kp_{k-1} + p_{k-2}, &q_k &= a_kq_{k-1} + q_{k-2}
                                      \qquad\text{for $k\geq 2$,}
\end{align*}
and that they satisfy the inequalities
\begin{equation}\label{vbc1}
  \dfrac{1}{q_n(q_{n+1}+q_n)} \ \leq  \    \Big|  x- \frac{p_n}{q_n}  \Big|    \ < \ \dfrac{1}{q_nq_{n+1}}  \, .
\end{equation}
From this it is clear that the convergents provide explicit solutions
to the inequality in Theorem~\ref{thm:7.4} (Dirichlet); that is,
\begin{equation*}
  \left|x- \frac{p_n}{q_n} \right| \leq \frac{1}{q_n^2}
  \qquad \forall n \in \N.
\end{equation*}
In fact, it turns out that for irrational $x$ the convergents are
\emph{best approximates} in the sense that if $1 \le q < q_n$ then any
rational $\frac{p}{q}$ satisfies
\begin{equation*}
  \left|  x - \frac{p_n}{q_n} \right| <  \left| x - \frac{p}{q} \right|.
\end{equation*}
Regarding $\pi= 3.141\dots$, the rationals~(\ref{eq:piapprox}) are the
first 5 convergents.

\subsection{Competing with Dirichlet and losing badly}
\label{sec:beating-dirichlet}

We have presented Dirichlet's theorem as an answer to whether the
trivial inequality $|x-p/q|\leq 1/2q$ can be beaten. Naturally, one
may also ask if we can do any better than Dirichlet's theorem.  Let us
formulate this a little more precisely.  For $ x \in \RR$, let
\begin{equation*}
  \|x\| := \min \{ | x-m | : m \in \Z \}
\end{equation*}
denote the distance from $x$ to the nearest integer.  Dirichlet's
theorem (Theorem \ref{thm:7.4}) can be restated as follows: {\em for
  any $ x \in \R $, there exist infinitely many integers $q >0$ such
  that
  \begin{equation} \label{dir} q \, \| q x \| \leq 1 \, .
  \end{equation}  }
Can we replace right-hand side of~(\ref{dir}) by
arbitrary $\epsilon > 0$? In other words, is it true that
$\liminf_{q\to\infty}q\|qx\| = 0 $ for every $x$? One might notice
that~(\ref{vbc1}) implies that there certainly do exist $x$ for which
this is true. (One can write down a continued fraction whose
partial quotients grow as fast as one pleases.) Still, the answer to
the question is \textbf{No.} It was proved by Hurwitz (1891) that for
every $x\in\RR$, we have $q \, \|q x\| <\epsilon= 1/\sqrt5$ for
infinitely many $q>0$, and that this is best
possible in the sense that the statement becomes false if $\epsilon <
1/\sqrt{5}$.

The fact that $1/\sqrt{5}$ is best possible is relatively easy to see.
Assume that it can be replaced by
\begin{equation*}
  \frac{1}{\sqrt{5} + \epsilon} \qquad (\epsilon > 0,\text{ arbitrary}).
\end{equation*}
Consider the Golden Ratio $x_1 = \frac{\sqrt{5}+1}{2}$, root of the
polynomial
\begin{equation*}
  f(t) = t^2 - t - 1 = (t-x_1)(t-x_2)
\end{equation*}
where $x_2 = \frac{1-\sqrt{5}}{2}$. Assume there exists a sequence of
rationals $\frac{p_i}{q_i}$ satisfying
\begin{equation*}
  \left| x_1 - \frac{p_i}{q_i} \right| < \frac{1}{(\sqrt{5}+\epsilon)q_i^2}.
\end{equation*}
Then, for $i$ sufficiently large, the right-hand side of the above
inequality is less than $\epsilon$ and so
\begin{equation*}
  \left| x_2 - \frac{p_i}{q_i} \right| \leq |x_2 - x_1| + \left| x_1 - \frac{p_i}{q_i} \right| < \sqrt{5} + \epsilon \, .
\end{equation*}
It follows that
\begin{align*}
  0 \ \neq \  \left| f\left( \frac{p_i}{q_i} \right) \right| &< \frac{1}{(\sqrt{5}+\epsilon)q_i^2} \cdot (\sqrt{5}+\epsilon) \\[2ex]
  \Longrightarrow \quad \left| q_i^2 f\left(\frac{p_i}{q_i}\right)
  \right| &< 1.
\end{align*}
However the left-hand side is a strictly positive integer. This is a
contradiction, for there are no integers in $(0,1)$---an extremely useful
fact.

The above argument shows that if $ x= \frac{\sqrt{5}+1}{2} $ then
there are at most finitely many rationals $p/q$ such that
$$
\left|x-\frac{p}{q}\right| < \frac{1}{(\sqrt{5}+\epsilon)q^2}.
$$
Therefore, there exists a constant $c(x)> 0$ such that
$$
\left|x-\frac{p}{q}\right| > \frac{c(x)}{q^2} \qquad \forall \ p/q \in
\Q \, .
$$
All of this shows that there exist numbers for which we can not
improve Dirichlet's theorem arbitrarily. These are called \emph{badly
approximable numbers}\index{badly approximable numbers} and are defined by
\begin{align*}
  \bad &:= \{x \in \RR: \inf_{q\in\N}q
         \|qx\|> 0\} \\[2ex]
       &= \{x \in \RR: c(x):=\liminf_{q\to\infty}q\|qx\| > 0\}.
\end{align*}
Note that if $x$ is badly approximable then for the associated badly
approximable constant $c(x)$ we have that
$$
0 < c(x) \le \frac{1}{\sqrt{5}}.
$$
Clearly, $ \bad \neq \varnothing$ since the golden ratio is badly
approximable.  Indeed, if $ x \in \bad$ then $ tx \in \bad$ for any
$ t \in \Z \setminus \{0\}$ and so $\bad$ is at least countable.

$\bad$ has a beautiful characterisation via continued fractions.

\begin{thm}\label{Bad}
  Let $x= [a_0;a_1,a_2,a_3, \ldots] $ be irrational.  Then
$$x \in \bad     \  \Longleftrightarrow  \   \exists  \ M=M(x) \ge 1  \ such \  that  \   a_i \leq M  \  \  \forall  \, i \ .$$
That is, $\bad$ consists exactly of the real numbers whose continued
fractions have bounded partial quotients.
\end{thm}

\begin{proof} It follows from \eqref{vbc1} that
  \begin{equation}\label{vbcslv1}
    \dfrac{1}{q_n^2(a_{n+1} +2)} \ \leq  \    \Big|  x- \frac{p_n}{q_n}  \Big|    \ < \ \dfrac{1}{a_{n+1}q_n^2},
  \end{equation}
  and from this it immediately follows that if $ x \in \bad$, then
  $a_n \leq \max \{|a_o|, 1/c(x) \} $.

  Conversely, suppose the partial quotients of $x$ are bounded, and
  take any $q\in\N$.  Then there is $n \geq 1$ such that
  $q_{n-1}\le q <q_n$. On using the fact that convergents are best
  approximates, it follows that
  \begin{equation*}
    \left|x-\frac pq\right| \ \ge \ \left|x-\frac{p_n}{q_n}\right| \ \ge \
    \frac{1}{q_n^2(M + 2) } \ = \ \frac{1}{q^2(M + 2) } \,
    \frac{q^2}{q_n^2}.
  \end{equation*}
  It is easily seen that
  \begin{equation*}
    \frac{q}{q_n}   \ge  \frac{q_{n-1}}{q_n}   \ge \frac{1}{M+1},
  \end{equation*}
  which proves that
  \begin{equation*}
c(x) \geq \frac{1}{(M+2)(M+1)^2}>0,
\end{equation*}
hence $x\in\bad$.
\end{proof}

Recall that a continued fraction of the form
$x=[a_0;\dots,a_n,\overline{a_{n+1},\dots,a_{n+m}}]$ is said to be
\emph{periodic}.  Also, recall that an irrational number $\alpha$ is
called a \emph{quadratic irrational} if $\alpha$ is a solution to a
quadratic equation with integer coefficients:
\begin{equation*}
  ax^2 + bx + c=0 \qquad\text{($a,b,c\in\mathbb{Z}$, $a \neq 0$)}.
\end{equation*}
It is a well-known fact that an irrational number $x$ has periodic
continued fraction expansion if and only if $x$ is a quadratic
irrational. This and Theorem~\ref{Bad} imply the following corollary.

\begin{cor}
  Every quadratic irrational is badly approximable.
\end{cor}

The simplest instance of this is the golden ratio, a root of
$x^2-x-1$, whose continued fraction is
\begin{equation*}
  \frac{\sqrt{5}+1}{2} = [1;1,1,1,\dots] := [\, \overline{1} \, ],
\end{equation*}
with partial quotients clearly bounded.

Indeed, much is known about the badly approximable numbers, yet
several simple questions remain unanswered. For example:

\theoremstyle{plain} \newtheorem*{fc}{Folklore Conjecture}
\begin{fc}
  The only algebraic irrationals that are in $\bad$ are the quadratic
  irrationals.
\end{fc}

\begin{rem}
  Though this conjecture is widely believed to be true, there is no
  direct evidence for it. That is, there is no single algebraic
  irrational of degree greater than two whose membership (or
  non-membership) in $\bad$ has been verified.
\end{rem}

A particular goal of these notes is to investigate the `size' of
$\bad$.
We will show:\\[3ex]
\hspace*{10ex}(a) ~~~$m(\bad)=0$\\[1ex]
\hspace*{10ex}(b) ~~~$\dim \bad= 1,$

\noindent where $\dim$ refers to the Hausdorff dimension (see \S\ref{DB}). In other words, we will see that $ \bad $ is a small set in
that it has measure zero in $\R$, but it is a large set in that it has
the same (Hausdorff) dimension as $\R$.

\paragraph{}Let us now return to Dirichlet's theorem (Theorem
\ref{thm:7.4}).  Every $x \in \R $ can be approximated by rationals
$p/q$ with `rate of approximation' given by $q^{-2}$---the right-hand
side of inequality \eqref{7.3} determines the `rate' or `error' of
approximation by rationals. The above discussion shows that this rate
of approximation cannot be improved by an arbitrary constant for every
real number---$\bad$ is non-empty. On the other hand, we have stated
above that $\bad$ is a $0$-measure set, meaning that the set of points
for which we \emph{can} improve Dirichlet's theorem by an arbitrary
constant is full. In fact, we will see that if we exclude a set of
real numbers of measure zero, then from a measure theoretic point of
view the rate of approximation can be improved not just by an
arbitrary constant but by a logarithm (see Remark~\ref{rem:log}).

\section{Metric Diophantine approximation: the classical Lebesgue
  theory \label{clt} }

In the previous section, we have been dealing with variations of
Dirichlet's theorem in which the right-hand side or rate of
approximation is of the form $ \epsilon q^{-2}$. It is natural to
broaden the discussion to general approximating functions. More
precisely, for a function $\psi:\N\to \Rp=[0,\infty)$, a real number
$x$ is said to be \emph{$\psi$--approximable} if there are infinitely
many $q\in\N$ such that
\begin{equation}\label{e:002}
  \|qx\|<\psi(q)  \ .
\end{equation}
The function $\psi$ governs the `rate' at which the rationals
approximate the reals and will be referred to as an
\emph{approximating function}\index{approximating function}.

One can readily verify that the set of $\p$-approximable numbers is
invariant under translations by integer vectors. Therefore without any
loss of generality, and to ease the `metrical' discussion which
follows, we shall restrict our attention to $\psi$--approximable
numbers in the unit interval $\I:=[0,1)$.  The set of such numbers is
clearly a subset of $\I$ and will be denoted by $W(\psi)$; i.e.
$$
W(\psi): = \{x\in \I \colon \|qx\|<\psi(q) \text{ for infinitely many
} q\in\N \} \ .
$$
Notice that in this notation we have that
\begin{equation*}
  \textrm{Dirichlet's Theorem (Theorem \ref{thm:7.4})} \quad \Longrightarrow \quad W (\psi) = \I \ \textrm{ if } \ \psi(q)=q^{-1} .
\end{equation*}
Yet, the existence of badly approximable numbers implies that there
exist approximating functions $\psi$ for which $W(\psi)\neq I$.
Furthermore, the fact that $m(\bad)=0$ implies that we can have
$W(\psi)\neq I$ while $m(W(\psi))=1$.

A key aspect of the classical theory of Diophantine approximation is
to determine the `size' of $W(\psi)$ in terms of
\begin{itemize}
\item[(a)] Lebesgue measure,
\item[(b)]Hausdorff dimension, and
\item[(c)] Hausdorff measure.
\end{itemize}
From a measure theoretic point of view, as we move from (a) to (c) in
the above list, the notion of size becomes subtler. In
this section we investigate the `size' of $W(\psi)$ in terms of $1$-
dimensional Lebesgue measure $m$.

\vspace*{2ex}

We start with the important observation that $W(\psi)$ is a $\limsup$
set of balls.  For a fixed $q \in \N$, let
\begin{eqnarray} \label{slv101}
  A_q (\psi)  &  :=   &  \{  x \in \I : \|qx\| \, < \, \psi(q)  \}   \nonumber \\[2ex]
              & := & \bigcup_{p=0}^q B \Big(\frac{p}{q},
                     \frac{\psi(q)}{q} \Big) \ \cap \I
\end{eqnarray}
Note that
\begin{equation} \label{slv103} m\big( A_q(\psi) \big) \leqslant
  2\psi(q)
\end{equation}
with equality when $ \psi(q) < 1/2 $ since then the intervals in
\eqref{slv101} are disjoint.

The set $W(\psi)$ is simply the set of real numbers in $\I$ which lie
in infinitely many sets $A_q(\psi)$ with $q=1,2,\dots$ i.e.\

\begin{equation*}
  W(\psi) =  \limsup_{q \to \infty} A_q(\psi) := \bigcap_{t=1}^{\infty} \bigcup_{q=t}^{\infty} A_q(\psi)
\end{equation*}
is a $\limsup$ set.  Now notice that for each $t\in \N$

\begin{equation*}
  W(\psi) \subset \bigcup_{q=t}^{\infty} A_q(\psi) \
\end{equation*}
i.e.\ for each $t$, the collection of balls $B(p/q, \psi(q)/q)$
associated with the sets $A_q(\psi): q=t, t+1, \dots $ form a cover
for $W(\psi)$. Thus, it follows via \eqref{slv103} that
\begin{eqnarray} \label{slv102}
  m \big( W(\psi) \big) &\le  &  m \left( \bigcup_{q=t}^{\infty} A_q(\psi) \right) \nonumber  \\
                        &\le &  \sum_{q=t}^{\infty} m\big( A_q(\psi) \big)  \nonumber  \\
                        &\le & 2 \sum_{q=t}^{\infty} \psi(q) \, .
\end{eqnarray}
Now suppose
\begin{equation*}
  \sum_{q=1}^{\infty} \psi(q) < \infty.
\end{equation*}
Then given any $\epsilon > 0$, there exists $t_0$ such that for all
$t \ge t_0$
\begin{equation*}
  \sum_{q=t}^{\infty} \psi(q) < \frac{\epsilon}{2}.
\end{equation*}
It follows from \eqref{slv102}, that
\begin{equation*}
  m\big( W(\psi) \big) < \epsilon.
\end{equation*}
But $\epsilon > 0$ is arbitrary, whence
\begin{equation*}
  m\big( W(\psi) \big) = 0 \,
\end{equation*}
and we have established the following statement.
\begin{thm} \label{convkt} Let $\psi:\N\to \Rp$ be a function such
  that
  \begin{equation*}
    \sum_{q=1}^{\infty} \psi(q) < \infty.
  \end{equation*}
  Then
  \begin{equation*}
    m ( W(\psi) ) =0.
  \end{equation*}
\end{thm}
This theorem is in fact a simple consequence of a general result in
probability theory.

\subsection{The Borel-Cantelli Lemma}

Let $(\Omega,\A,\mu)$ be a measure space with $\mu(\Omega)<\infty$ and
let $E_q$ ($q \in \mathbb{N}$) be a family of measurable sets in
$\Omega$.  Also, let
$$
E_{\infty}:=\limsup_{q \to \infty} E_q := \bigcap_{t=1}^{\infty}
\bigcup_{q=t}^{\infty} E_q \ ;
$$
i.e. $ E_{\infty} $ is the set of $x \in \Omega$ such that
$ x \in E_i $ for infinitely many $i \in \N$.

The proof of the Theorem \ref{convkt} mimics the proof of the
following fundamental statement from probability theory.

\begin{lem}[Convergence Borel-Cantelli] Suppose that
  $ \sum_{q=1}^{\infty} \mu(E_q) <~\infty$.  Then,
  $$\mu(E_{\infty}) = 0 \, . $$

\end{lem}

\begin{proof} Exercise.
\end{proof}

To see that Theorem \ref{convkt} is a trivial consequence of the above
lemma, simply put $\Omega = \I = [0,1]$, $\mu = m$ and
$E_q = A_q(\psi)$ and use \eqref{slv103}.

Now suppose we are in a situation where the sum of the measures
diverges.  Unfortunately, as the following example demonstrates, it is
not the case that if $\sum \mu(E_q) = \infty$ then
$\mu(E_{\infty}) = \mu(\Omega)$ or indeed that $\mu(E_{\infty}) > 0$.

\noindent \textbf{Example:} Let $E_q=(0,\frac1q)$. Then
$\sum_{q=1}^\infty m(E_q)=\sum_{q=1}^\infty\frac{1}{q}=\infty$.
However, for any $t\in\N$ we have that
$$
\bigcup_{q=t}^\infty E_q=E_t\,,
$$
and thus
$$
E_{\infty}=\bigcap_{t=1}^\infty E_t=\bigcap_{t=1}^\infty
(0,\tfrac1t)=\varnothing
$$
implying that $m(E_{\infty})=0$.

The problem in the above example is that the sets $E_q$ overlap `too
much'---in fact they are nested.  The upshot is that in order to have
$\mu(E_{\infty}) > 0$, we not only need the sum of the measures to
diverge but also that the sets $E_q$ ($q \in \mathbb{N}$) are in some
sense independent. Indeed, it is well-known that if we had pairwise
independence in the standard sense; \emph{i.e.} if
$$
\mu ( E_s \cap E_t ) = \mu(E_s) \mu(E_t) \qquad \forall s \neq t,
$$
then we would have $\mu( E_{\infty}) = \mu (\Omega)$.  However, we
very rarely have this strong form of independence in our applications.
What is much more useful to us is the following statement, whose proof
can be found in \cite{HarmanMNT, Sprindzuk}.

\begin{lem}[Divergence  Borel-Cantelli]
  Suppose that $\sum_{q=1}^{\infty} \mu(E_q) = \infty$ and that there
  exists a constant $C>0$ such that
  \begin{equation}\label{vbx1x}
    \sum_{s,t=1}^Q  \mu(E_s\cap E_t)\le C\left(\sum_{s=1}^Q  \mu(E_s)\right)^2
  \end{equation}
  holds for infinitely many $Q\in\N$.  Then
$$
\mu(E_{\infty}) \ge 1/C\,.
$$
\end{lem}

The independence condition \eqref{vbx1x} is often referred to as
\emph{quasi-independence on average}, and, together with the divergent
sum condition, it guarantees that the associated $\limsup$ set has
positive measure. It does not guarantee \emph{full} measure
(\emph{i.e.}~that $\mu(E_{\infty}) = \mu(\Omega)$), which is what we
are trying to prove, for example, in Khintchine's Theorem. But this is
not an issue if we already know (by some other means) that $E_\infty$
satisfies a zero-full law (which is also often called a zero-one law) with respect to the measure $\mu$, meaning a
statement guaranteeing that
\begin{equation*}
  \mu(E_{\infty}) =   \   0 \quad\text{or} \quad \mu(\Omega).
\end{equation*}
Happily, this is the case with the $\limsup $ set $ W(\psi)$ of
$\psi$-well approximable numbers \cite{Cassels-50:MR0036787,
  Casselshort, HarmanMNT}.

Alternatively, assuming $\Omega$ is equipped with a metric such that
$\mu$ becomes a doubling Borel measure, we can guarantee that
$\mu(E_{\infty}) = \mu(\Omega)$ if we can establish \emph{local
  quasi-independence on average} \cite[\S8]{BDV06}; \emph{i.e.}~we replace
\eqref{vbx1x} in the above lemma by the condition that
\begin{equation}\label{vbx1xslv}
  \sum_{s,t=1}^Q  \mu\big((B \cap E_s) \cap (B \cap E_t) \big)\le \frac{C}{\mu(B)} \left(\sum_{s=1}^Q  \mu(B \cap E_s)\right)^2 \,
\end{equation}
for any sufficiently small ball $B$ with center in $\Omega$ and
$\mu(B) > 0$. The constant $C$ is independent of the ball $B$. Recall
that $\mu$ is doubling if $\mu(2B)\ll \mu(B)$ for balls $B$ centred in
$\Omega$. In some literature such measures are also referred to
as Federer measures.

\paragraph{} The Divergence Borel-Cantelli Lemma is key to determining
$m( W(\psi)) $ in the case where $ \sum_{q=1}^{\infty} \psi(q)$
diverges---the subject of the next section and the main substance of
Khintchine's Theorem.  Before turning to this, let us ask ourselves
one final question regarding quasi-independence on average and
positive measure of $\limsup$ sets.

\noindent \textbf{Question.}  Is the converse to Divergence
Borel-Cantelli true?  More precisely, if $\mu(E_{\infty}) > 0$ then is
it true that the sets $E_t$ are quasi-independent on
average?

The following theorem is a consequence of a more general result established in
\cite{BCvbsv}.

\begin{thm}\label{BCSV}
  Let $(\Omega,d)$ be a compact metric space equipped with a Borel
  probability measure $\mu$.  Let $E_q$ $(q \in \mathbb{N})$ be a
  sequence of balls in $\Omega$ such that $\mu(E_{\infty}) > 0$.
  Then, there exists a strictly increasing sequence of integers
  $(q_k)_{k\in\N}$ such that
  $ \sum_{k=1}^{\infty} \mu(E_{q_k}) = \infty \, $ and the balls
  $E_{q_k}$ $(k\in\N)$ are quasi-independent on average.
\end{thm}

\subsection{Khintchine's Theorem}

The following fundamental statement in metric Diophantine
approximation (of which Theorem \ref{convkt} is the ``easy case'')
provides an elegant criterion for the `size' of the set $W(\psi)$
expressed in terms of Lebesgue measure.

\begin{thm}[Khintchine\index{Khintchine's Theorem}, 1924] \label{kg} Let $\psi:\N\to \Rp$ be a
  monotonic function.  Then
$$ m(W(\psi)) =\left\{
\begin{array}{ll}
  0 & {\rm if} \;\;\; \sum_{q=1}^{\infty} \;   \psi(q)  <\infty\;
      ,\\[4ex]
  1 & {\rm if} \;\;\; \sum_{q=1}^{\infty} \;   \psi(q)
      =\infty \; .
\end{array}\right.$$
\end{thm}

\begin{rem}
  It is worth mentioning that Khintchine's original statement
  \cite{Kh24} made the stronger assumption that $q \p(q)$ is
  monotonic.
\end{rem}

\begin{rem}
  The assumption that $\psi$ is monotonic is only required in the
  divergent case.  It cannot in general be removed---see
  \S\ref{dsconj} below.
\end{rem}

\begin{rem}\label{rem:log}
  Khintchine's Theorem implies that
  $$ m (W(\psi))= 1 \quad {\rm if } \quad \psi(q) = 1/q \log q \, . $$
  Thus, from a measure theoretic point of view the `rate' of
  approximation given by Dirichlet's theorem can be improved by a
  logarithm.
\end{rem}

\begin{rem}
  As mentioned in the previous section, in view of Cassels' zero-full
  law~\cite{Cassels-50:MR0036787} we know that $ m (W(\psi))= 0 $ or
  $1$ regardless of whether or not $\psi$ is monotonic.
\end{rem}

\begin{rem}
  A key ingredient to directly establishing the divergent part is to
  show that the sets
 $$A^*_s = A^*_s(\psi) := \bigcup_{2^{s-1} \le q < 2^s}    \bigcup_{p=0}^q B \Big(\frac{p}{q}, \frac{\psi(2^s)}{2^s} \Big)   \ \cap  \I \, .$$
 are quasi-independent on average. Notice that
 \begin{itemize}
 \item For $\psi$ monotonic,
   $W(\psi) \supset W^*(\psi):= \limsup_{s \to \infty} A^*_s(\psi)$.
 \item If $\psi(q) < q^{-1} $, the balls in $A^*_s(\psi)$ are disjoint
   and so
$$
m(A^*_s(\psi)) \ \asymp \ 2^s \psi(2^s) \, .
$$
\item For $\psi$ monotonic,
  $ \sum \psi(q) \asymp \sum 2^s \psi(2^s) $.
\end{itemize}
\end{rem}

\noindent {\em Notation.} Throughout, the Vinogradov symbols $\ll$ and
$\gg$ will be used to indicate an inequality with an unspecified
positive multiplicative constant.  If $a \ll b $ and $ a \gg b $, we
write $a \asymp b $ and say that the two quantities $a$ and $b$ are
comparable.

The following is a simple consequence of Khintchine's Theorem.

\begin{cor} Let $ \bad$ be the set of badly approximable numbers. Then
$$m(\bad) = 0 \, .$$
\end{cor}

\begin{proof}
  Consider the function $\psi(q) = 1/(q\log q)$ and observe that
$$
\bad \cap \I \subset \bad(\psi):= \I\setminus W(\psi) \, .
$$
By Khintchine's Theorem, $m(W(\psi)) = 1 $. Thus $m(\bad(\psi)) = 0$
and so $m (\bad \cap \I) = 0$.
\end{proof}

\subsubsection{The Duffin-Schaeffer Conjecture \label{dsconj}}
The main substance of Khintchine's Theorem is the divergent case and
it is where the assumption that $\psi$ is monotonic is necessary. In
1941, Duffin $\&$ Schaeffer \cite{DuffinSchaeffer} constructed a
non-monotonic approximating function $\vartheta$ for which the sum
$\sum_q \vartheta(q)$ diverges but $m(W(\vartheta))=0$.  We now
discuss the construction.  We start by recalling two well-known facts:
for any $N\in\N$, $p$ prime, and $s>0$,
\begin{itemize}
\item[] Fact 1.  $ \quad \sum_{q|N} q = \prod_{p|N} (1+p) $
\item[] Fact 2.  $ \quad \prod_{p} (1+p^{-s}) = \zeta(s)/\zeta(2s) $.
\end{itemize}
In view of Fact 2, we have that
$$
\prod_{p} (1+p^{-1}) = \infty.
$$
Thus, we can find a sequence of square free positive integers $N_i$
($i=1,2,\ldots$) such that $(N_i,N_j) = 1 $ ($i\neq j$) and
\begin{equation}\label{e:003sv}
  \prod_{p\mid N_i} (1+p^{-1})  > 2^i + 1 \, .
\end{equation}
Now let
\begin{equation}\label{e:003sv3}
  \vartheta(q)  =\left\{
    \begin{array}{ll}
      2^{-i-1} q/N_i  & {\rm if} \;\;\;  q > 1\quad\text{and}\quad q|N_i\quad\text{for some }i  \;
                        ,\\[4ex]
      0 & {\rm otherwise} \, .

    \end{array}\right.
\end{equation}

\noindent As usual let
$$
A_q:=A_q (\vartheta) = \bigcup_{p=0}^q B \Big(\frac{p}{q},
\frac{\vartheta(q)}{q} \Big) \ \cap \I
$$
and observe that if $ q|N_i$ ($q> 1$) then $A_q \subseteq A_{N_i} $
and so
$$
\bigcup_{q|N_i} A_q = A_{N_i} \, .
$$
In particular
$$
m\big(\bigcup_{q|N_i} A_q \big) = m( A_{N_i} ) = 2\vartheta(N_i) =
2^{-i} \, .
$$
By definition
$$
W(\vartheta) = \limsup_{q \to \infty} A_q = \limsup_{i \to \infty}
A_{N_i} \, .
$$
Now
$$
\sum_{i=1}^{\infty} m( A_{N_i} ) = 1
$$
and so the convergence Borel-Cantelli Lemma implies that
$$
m(W(\vartheta)) = 0 \, .
$$
However, it can be verified (\emph{exercise}) on using Fact 1 together
with \eqref{e:003sv} that
$$
\sum_{q=1}^{\infty} \vartheta(q) = \sum_{i=1}^{\infty} 2^{-i-1}
\frac{1}{N_i} \sum_{q> 1 \, : \, q|N_i} \!\! q \ \ = \infty \, .
$$

In the same paper \cite{DuffinSchaeffer}, Duffin and Schaeffer
provided an appropriate statement for arbitrary $\psi$ that we now
discuss.  The now famous Duffin-Schaeffer Conjecture represents a key
open problem in number theory.  The integer $p$ implicit in the
inequality (\ref{e:002}) satisfies
\begin{equation}\label{e:003}
  \left|x-\frac pq\right|<\frac{\psi(q)}q\,.
\end{equation}
To relate the rational $p/q$ with the error of approximation
$\psi(q)/q$ uniquely, we impose the coprimeness condition
$(p,q)=1$. In this case, let $W'(\psi)$ denote the set of $x$ in $\I$
for which the inequality (\ref{e:003}) holds for infinitely many
$(p,q)\in\Z\times\N$ with $(p,q)=1$.  Clearly,
$W'(\psi)\subset W(\psi)$.  For any approximating function
$\psi:\N\to \Rp$ one easily deduces that
$$
m(W'(\psi))= 0 \quad {\rm if} \quad \sum_{q=1}^\infty \, \varphi(q) \
\dfrac{\psi(q)}q \ < \ \infty \ .
$$
Here, and throughout, $\varphi$ is the Euler function.

\begin{conjecture}[Duffin-Schaeffer\index{Duffin-Schaeffer Conjecture}, 1941] \label{dsstate} For any
  function $\psi\colon \N\to \Rp$
$$
m(W'(\psi)) = 1 \quad if \quad \displaystyle\sum_{q=1}^\infty \,
\varphi(q) \ \dfrac{\psi(q)}q \ = \ \infty \ .
$$
\end{conjecture}

\begin{rem}  {\rm  Let $\vartheta$ be given by  \eqref{e:003sv3}.   On using the fact that $ \sum_{d|n} \varphi(d)  = n$, it is relatively easy to show (\emph{exercise}) that
$$
\displaystyle\sum_{q=1}^\infty \, \varphi(q) \ \dfrac{\vartheta(q)}q \
< \ \infty \, .
$$
Thus, although $\vartheta$ provides a counterexample to Khintchine's
Theorem without monotonicity, it is not a counterexample to the
Duffin-Schaeffer Conjecture.  }
\end{rem}

\begin{rem} {\rm It is known that $ m(W'(\psi)) = 0 $ or $ 1 $.  This
    is Gallagher's zero-full law \cite{Gallagher-61:MR0133297} and is
    the natural analogue of Cassels' zero-full law for $W(\psi)$.  }
\end{rem}

\noindent Although various partial results have been established
(see~\cite{HarmanMNT,Sprindzuk}), the full conjecture is one
of the most difficult and profound unsolved problems in metric number
theory.  In the case where $\psi$ is monotonic it is relatively
straightforward to show that Khintchine's Theorem and the
Duffin-Schaeffer Conjecture are equivalent statements
(\emph{exercise}).

\subsection{A limitation of the Lebesgue theory}
Let $ \tau > 0$ and write $ W(\tau)$ for $ W(\psi:q \to q^{-\tau})$.
The set $W(\tau)$ is usually referred to as the set of
\emph{$\tau$-well approximable numbers}.  Note that in view of
Dirichlet (Theorem \ref{thm:7.4}) we have that $ W (\tau) = \I $ if
$ \ \tau \leq 1 $ and so trivially $ m( W (\tau) ) = 1$ if
$ \tau \leq 1 $. On the other hand, if $\tau > 1$
$$
\textstyle{\sum_{q=1}^{\infty} \; q^{-\tau} <\infty }
$$
and Khintchine's Theorem implies that $ m( W (\tau) ) = 0$.  So for
any $ \tau > 1$, the set of $\tau$-well approximable numbers is of
measure zero. We cannot obtain any further information regarding the
`size' of $W(\tau)$ in terms of Lebesgue measure --- it is always
zero. Intuitively, the `size' of $W(\tau)$ should decrease as rate of
approximation governed by $\tau$ increases. For example we would
expect that $W(2015)$ is ``smaller'' than $W(2)$ -- clearly
$W(2015) \subset W(2) $ but Lebesgue measure is unable to distinguish
between them.  In short, we require a more delicate notion of `size'
than simply Lebesgue measure. The appropriate notion of `size' best
suited for describing the finer measure theoretic structures of
$W(\tau)$ and indeed $W(\psi )$ is that of Hausdorff measures.

\section{Metric Diophantine approximation: the classical Hausdorff
  theory}

\subsection{Hausdorff measure and dimension} \label{DB}

In what follows, a {\em dimension function\,} $f:\Rp \to \Rp $ is a left
continuous, monotonic function such that $f(0)=0$.  Suppose $F$ is a
subset of $\R^n$. Given a ball $B$ in $\R^n$, let $r(B)$ denote the
radius of $B$. For $\rho > 0$, a countable collection
$ \left\{B_{i} \right\} $ of balls in $\R^n$ with $r(B_i) \leq \rho $
for each $i$ such that $F \subset \bigcup_{i} B_{i} $ is called a {\em
  $ \rho $-cover for $F$}. Define
$$
{\cal H}^{f}_{\rho} (F) \, := \, \inf \ \sum_{i} f(r(B_i)),
$$
where the infimum is taken over all $\rho$-covers of $F$. Observe that
as $\rho$ decreases the class of allowed $\rho$-covers of $F$ is
reduced and so $\mathcal{H}_{\rho}^f(F)$ increases.  Therefore, the
following (finite or infinite) limit exists
$$ {\cal H}^{f} (F) := \lim_{ \rho \rightarrow 0+} {\cal
  H}^{f}_{\rho} (F) \; = \; \sup_{\rho > 0 } {\cal H}^{f}_{\rho} (F)
\; ,
$$
and is referred to as the {\it Hausdorff $f$--measure of $F$}\index{Hausdorff measure}.  In the
case that $$f(r) = r^s \ (s \ge 0),$$ the measure $ \cH^f $ is the
more common {\em$s$-dimensional Hausdorff measure} $\cH^s $, the
measure $\cH^0$ being the cardinality of $F$.  Note that when $s$ is a
positive integer, $\cH^s$ is a constant multiple of Lebesgue measure
in $\R^s$. (The constant is explicitly known!)
Thus if the $s$-dimensional Hausdorff measure of a set is
known for each $s>0$, then so is its $n$-dimensional Lebesgue measure
for each $n\ge1$. The following easy property
$$
\cH^{s}(F)<\infty\quad \Longrightarrow\quad
\cH^{s'}(F)=0\qquad\text{if }s'>s
$$
implies that there is a unique real point $s$ at which the Hausdorff
$s$-measure drops from infinity to zero (unless the set $F$ is finite
so that $\cH^s(F)$ is never infinite). This point is called the
\emph{Hausdorff dimension}\index{Hausdorff dimension}\/ of $F$ and is formally defined as
$$
\dim F := \inf \left\{ s>0 : \cH^{s} (F) =0 \right\} \, .
$$

\begin{itemize}
\item By the definition of $\dim F$ we have that

	\begin{equation*}
          \mathcal{H}^s(F) = \begin{cases}
            0 &\text{if }s > \dim F \\[2ex]
            \infty &\text{if }s < \dim F.
          \end{cases}
	\end{equation*}
	
      \item If $s = \dim F$, then $\mathcal{H}^s(F)$ may be zero or
        infinite or may satisfy
        \begin{equation*}
          0 < \mathcal{H}^s(F) < \infty ;
	\end{equation*}
        in this case $F$ is said to be an $s$-set.
	
      \item Let $\I = [0,1]$. Then $\dim \I = 1$ and
	
	\begin{equation*}
          2\mathcal{H}^s(\I) = \begin{cases}
            0 &\text{if }s > 1\\[1ex]
            1 &\text{if }s = 1\\[1ex]
            \infty &\text{if }s < 1.\\
          \end{cases}
	\end{equation*}
	
	\noindent Thus, $2\cH^1(\I) = m (\I) $ and $\I$ is an example
        of a $s$-set with $s=1$. Note that the present of the factor `2' here is because in the definition of Hausdorff measure we have used the radii of balls rather than diameters.
      \end{itemize}

      \noindent The Hausdorff dimension has been established for many
      number theoretic sets, e.g.  $W(\tau)$ (this is the
      \Jarnik-Besicovitch Theorem discussed below), and is easier than
      determining the Hausdorff measure. Further details regarding
      Hausdorff measure and dimension can be found
      in~\cite{FalcGFS,MattilaGS}.

      \vspace{2ex}

      To calculate $\dim F$ (say $\dim F = \alpha$), it is usually the
      case that we establish the upper bound $\dim F \le \alpha $ and
      lower bound $\dim F \ge \alpha $ separately.  If we can exploit
      a `natural' cover of $F$, then upper bounds are usually easier.

\begin{eg}\label{eg382}\rm
  Consider the middle third Cantor set\index{Middle third Cantor set} $K$ defined as follows:
  starting with $I_0=[0,1]$ remove the open middle thirds part of the
  interval. This gives the union of two intervals $[0,\tfrac13]$ and
  $[\tfrac23,1]$. Then repeat the procedure of removing the middle
  third part from each of the intervals in your given
  collection. Thus, at `level' $n$ of the construction we will have
  the union $E_n$ of $2^n$ closed intervals, each of length
  $3^{-n}$. The middle third Cantor set is defined by
$$
K=\bigcap_{n=0}^\infty E_n\,.
$$
This set consists exactly of all real numbers such that their
expansion to the base 3 does not contain the `digit' $1$.

Let $\{I_{n,j}\}$ be the collection of intervals in $E_n$. This is a
collection of $2^n$ intervals, each of length $3^{-n}$. Naturally,
$\{I_{n,j}\}$ is a cover of $K$. Furthermore, for any $\rho>0$ there
is a sufficiently large $n$ such that $\{I_{n,j}\}$ is a $\rho$-cover
of $K$. It follows that,

\begin{equation*}
  \mathcal{H}_{\rho}^s (K) \le  \sum_j  r(I_{n,j})^s \asymp  2^{n} 2^{-s} 3^{-ns}   \ll    \left(\frac{2}{3^s}\right)^n \to 0
\end{equation*}
as $n \to \infty$ (i.e.\ $\rho \to 0$) if
\begin{equation*}
  \frac{2}{3^s}<1 \Rightarrow s > \frac{\log 2}{\log 3}.
\end{equation*}
In other words
\begin{equation*}
  \mathcal{H}^s(K) = 0 \text{ if }s>\frac{\log 2}{\log 3}.
\end{equation*}
It follows from the definition of Hausdorff dimension
\begin{equation*}
  \dim K = \inf \{s :  \mathcal{H}^s(K) = 0\}
\end{equation*}
that $\dim K \leqslant \frac{\log 2}{\log 3}$.
\end{eg}

In fact, $\dim K= \frac{\log 2}{\log 3}$.  To prove that
\begin{equation*}
  \dim K \geqslant \frac{\log 2}{\log 3}
\end{equation*}
we need to work with arbitrary covers of $K$ and this is much
harder. Let $\{B_i\}$ be an arbitrary $\rho$-cover with $\rho < 1$.
$K$ is bounded and closed (intersection of closed intervals), i.e.\
$K$ is compact. Hence without loss of generality we can assume that
$\{B_i\}$ is finite. For each $B_i$, let $r_i$ and $d_i$ denote its
radius and diameter respectively, and let $k$ be the unique integer
such that
\begin{equation}\label{eq:6.4}
  3^{-(k+1)}\leqslant d_i < 3^{-k}.
\end{equation}
Then $B_i$ intersects at most one interval of $E_k$ as the intervals
in $E_k$ are separated by at least $3^{-k}$.

If $j \geqslant k$, then $B_i$ intersects at most
\begin{equation}\label{eq:6.5}
  2^{j-k} = 2^j3^{-sk} \leqslant 2^j 3^s d_i^s
\end{equation}
intervals of $E_j$, where $s := \frac{\log 2}{\log 3}$ and the final
inequality makes use of \eqref{eq:6.4}. These are the intervals that
are contained in the unique interval of $E_k$ that intersects $B_i$.

Now choose $j$ large enough so that
\begin{equation*}
  3^{-(j+1)} \leqslant d_i \quad \forall B_i \in  \{B_i\}  \, .
\end{equation*}
This is possible because the collection $\{B_i\}$ is finite. Then
$j \geqslant k$ for each $B_i$ and \eqref{eq:6.5} is valid. Furthermore,
since $\{B_i\}$ is a cover of $K$, it must intersect every interval of
$E_j$. There are $2^j$ intervals in $E_j$.  Thus
\begin{align*}
  2^j &= \#\{I \in E_j : \cup B_i \cap I \neq \varnothing\} \\
      &\le \sum_i \#\{I \in E_j : B_i \cap I \neq \varnothing\} \\
      &\le \sum_i 2^j3^sd_i^s  \, .
\end{align*}
The upshot is that for any arbitrary cover $\{B_i\}$, we have that
\begin{equation*}
  2^s \sum r_i^s   \, \asymp \,  \sum d_i^s    \ge 3^{-s} = \frac{1}{2}  \, .
\end{equation*}
By definition, this implies that implies
$\mathcal{H}^s(K) \ge 2^{-(1+s)}$ and so
$\dim K \ge \frac{\log 2}{\log 3}$.

Even for this simple Cantor set example, the lower bound for $\dim K$ is
much more involved than the upper bound. This is usually the case and
the number theoretic sets $W(\psi)$ and $W(\tau)$ are no exception.

\subsection{The \Jarnik-Besicovitch
Theorem}
Recall, the limsup nature of $W(\psi)$; namely that
\begin{equation*}
  W(\psi) =  \limsup_{q \to \infty} A_q(\psi) := \bigcap_{t=1}^{\infty} \bigcup_{q=t}^{\infty} A_q(\psi)
\end{equation*}
where
\begin{equation*}
  A_q(\psi)=\bigcup_{p=0}^q B \Big(\frac{p}{q}, \frac{\psi(q)}{q} \Big)   \ \cap  \I \, .
\end{equation*}
By definition, for each $t$, the collection of balls
$B(p/q, \psi(q)/q)$ associated with the sets
$A_q(\psi): q=t, t+1, \dots $ form a cover for $W(\psi)$.  Suppose for
the moment that \emph{$\psi$ is monotonic} and $\psi(q) < 1 $ for $q$
large. Now for any $\rho > 0 $, choose $t$ large enough so that
$\rho > \psi(t)/t$. Then the balls in $\{A_q(\psi)\}_{q \geqslant t}$
form a $\rho$ cover of $W(\psi)$. Thus,
\begin{equation*}
  \mathcal{H}_{\rho}^s\big(W(\psi)\big) \le \sum_{q = t}^{\infty} q\big(\psi(q)/q\big)^s   \to 0
\end{equation*}
as $t \to \infty$ (i.e.\ $\rho \to 0$) if
\begin{equation*}
  \sum_{q=1}^{\infty} q^{1-s}\psi^s(q) < \infty \, ;
\end{equation*}
i.e.\ $\mathcal{H}^s\big(W(\psi)\big)=0$ if the above $s$-volume sum
converges. Actually, monotonicity on $\psi$ can be removed
(\emph{exercise}) and we have proved the following Hausdorff measure
analogue of Theorem \ref{convkt}.  Recall, that $\mathcal{H}^1$ and
one-dimensional Lebesgue measure $m$ are comparable.
\begin{thm} \label{convja} Let $\psi:\N\to \Rp$ be a function and
  $ s \geq 0$ such that
  \begin{equation*}
    \sum_{q=1}^{\infty}q^{1-s} \psi^s(q) < \infty.
  \end{equation*}
  Then
  \begin{equation*}
    \mathcal{H}^s\big( W(\psi) \big) = 0\,.
  \end{equation*}
\end{thm}

Now put $\psi(q) = q^{-\tau} $ $(\tau \ge 1)$ and notice that for
$s > \frac{2}{\tau+1}$
\begin{equation*}
  \sum_{q=1}^{\infty}q^{1-s} \psi^s(q)  = \sum_{q=1}^{\infty}q^{-(\tau s  +s -1) } < \infty\,.
\end{equation*}
Then the following statement is a simple consequence of the above
theorem and the definition of Hausdorff dimension.

\begin{cor}\label{cor:6.2}
  For $\tau \ge 2$, we have that $\dim W(\tau) \le \frac{2}{\tau+1}$.
\end{cor}

Note that the above convergence result and thus the upper bound
dimension result, simply exploit the natural cover associated with the
limsup set under consideration.  The corollary constitutes the easy
part of the famous \Jarnik-Besicovitch Theorem.

\begin{thm}[The \Jarnik-Besicovitch Theorem\index{Jarnik-Besicovitch Theorem}] \label{jb} Let
  $\tau > 1$.  Then
  \begin{equation*}
    \dim \big(W(\tau)\big) = 2/(\tau +1) \, .
  \end{equation*}
\end{thm}

\Jarnik \ proved the result in 1928.  Besicovitch proved the same
result in 1932 by completely different methods.  The
\Jarnik-Besicovitch Theorem implies that
$$   \dim W(2) = 2/3   \quad    {\rm and }   \quad  \dim W(2015) = 2/2016   $$
and so $W(2015)$ is ``smaller'' than $W(2)$ as expected.  In view of
Corollary~\ref{cor:6.2}, we need to establish the lower bound result
$\dim \big(W(\tau)\big) \geq 2/(\tau +1)$ in order to complete the
proof of Theorem \ref{jb}.  We will see that this is a consequence of
\Jarnik's measure result discussed in the next section.

The dimension theorem is clearly an excellent result but it gives no
information regarding $\mathcal{H}^s$ at the critical exponent
$d:= 2/(\tau +1) $. By definition
\begin{equation*}
  \mathcal{H}^s(W(\tau)) = \begin{cases}
    0 &\text{if }s > d\\[1ex]
    \infty &\text{if }s < d\\
  \end{cases}
\end{equation*}
but
$$
\mathcal{H}^s(W(\tau)) = \textbf{?}  \quad \text{ if }s = d \,.
$$
In short, it would be highly desirable to have a Hausdorff measure
analogue of Khintchine's Theorem.

\subsection{\Jarnik's Theorem}

Theorem \ref{convja} is the easy case of the following fundamental
statement in metric Diophantine approximation.  It provides an elegant
criterion for the `size' of the set $W(\psi) $ expressed in terms of
Hausdorff measure.

\begin{thm}[\Jarnik's Theorem\index{Jarnik's Theorem}, 1931] \label{Jthm} Let $\psi:\N\to \Rp$
  be a monotonic function and $s \in (0,1)$. Then
  \begin{equation*}
    \mathcal{H}^s\big(W(\psi)\big) = \begin{cases}
      0 &{\rm if} \;\;\;  \sum_{q=1}^{\infty}q^{1-s} \psi^s(q)<\infty \\[2ex]
      \infty & {\rm if} \;\;\; \sum_{q=1}^{\infty}q^{1-s} \psi^s(q)=\infty
    \end{cases}
  \end{equation*}
\end{thm}

\begin{rem}  {\rm With $\psi(q) = q^{-\tau}  \  \ (\tau > 1)$, not only does the above  theorem imply  that  $\dim W(\tau) =  2/(1+ \tau) $ but it tells us that the Hausdorff measure at the critical exponent is infinite; i.e.
$$
\mathcal{H}^s\big(W(\tau)\big) =  \infty     \quad  {\rm at \ }  \ s = 2/(1+ \tau)  \, .
$$
}
\end{rem}

\begin{rem}  {\rm As in Khintchine's Theorem, the assumption that $\psi$ is monotonic is only required in the divergent case. In \Jarnik's original statement, apart from assuming stronger monotonicity conditions,   various technical conditions on $\psi$ and indirectly $s$ were imposed, which prevented $s=1$.   Note that even as stated, it is natural to exclude the case $s=1$ since
$$
\mathcal{H}^1\big(W(\psi)\big) \asymp m\big(W(\psi)\big) = 1 \, .
$$
The clear cut statement without the technical conditions was
established in \cite{BDV06} and it allows us to combine the theorems
of Khintchine and \Jarnik\ into a unifying statement. }
\end{rem}

\begin{thm}[Khintchine-\Jarnik \ 2006]
  \label{fulllmhm}
  Let $\psi:\N\to \Rp$ be a monotonic function and $s \in (0,1]$. Then
  \begin{equation*}
    \mathcal{H}^s\big(W(\psi)\big) = \begin{cases}
      0 &{\rm if} \;\;\;  \sum_{q=1}^{\infty}q^{1-s} \psi^s(q)<\infty\,, \\[2ex]
      \mathcal{H}^s(\I) & {\rm if} \;\;\; \sum_{q=1}^{\infty}q^{1-s} \psi^s(q)=\infty\,.
    \end{cases}
  \end{equation*}
\end{thm}

\noindent Obviously, the Khintchine-\Jarnik \ Theorem implies
Khintchine's Theorem.

In view of the Mass Transference Principle established in~\cite{BV06}
one actually has that
\begin{center} Khintchine's Theorem
  $ \hspace{4mm} \Longrightarrow \hspace{4mm} $ Jarn\'{\i}k's
  Theorem.  \end{center} Thus, the Lebesgue theory of $W(\psi)$
underpins the general Hausdorff theory. At first glance this is rather
surprising because the Hausdorff theory had previously been thought to
be a subtle refinement of the Lebesgue theory. Nevertheless, the Mass
Transference Principle allows us to transfer Lebesgue measure
theoretic statements for limsup sets to Hausdorff statements and
naturally obtain a complete metric theory.

\subsection{The Mass Transference Principle \label{mtpballs}}

Let $(\Omega,d)$ be a locally compact metric space and suppose there
exist constants $ \delta > 0$, $0<c_1<1<c_2<\infty$ and $r_0 > 0$ such
that
\begin{equation}\label{g}
  c_1\ r^\delta \le \cH^\delta(B)\le c_2\ r^\delta \ ,
\end{equation}
for any ball $B=B(x,r)$ with $x\in \Omega$ and radius $r\le r_0$. For
the sake of simplicity, the definition of Hausdorff measure and
dimension given in \S\ref{DB} is restricted to $\R^n$.  Clearly, it
can easily be adapted to the setting of arbitrary metric spaces --
see~\cite{FalcGFS,MattilaGS}.  A consequence of \eqref{g} is that
\begin{equation*}
  0 < \mathcal{H}^{\delta}(\Omega) < \infty    \quad {\rm and \ } \quad  \dim \Omega = \delta \, .
\end{equation*}
Next, given a dimension function $f$ and a ball $B=B(x,r)$ we define
the scaled ball
$$
B^f:=B\big(x,f(r)^{\frac{1}{\delta}}\big)\,.
$$
When $f(r) = r^s$ for some $s > 0 $, we adopt the notation $B^s$, i.e.
\begin{equation*}
  B^s := B\big(x, r^{\frac{s}{\delta}}\big)
\end{equation*}
and so by definition $B^{\delta}=B$.

The Mass Transference Principle \cite{BV06} allows us to transfer
$\cH^\delta$-measure theoretic statements for limsup subsets of
$\Omega$ to general $\cH^f$-measure theoretic statements. Note that in
the case $ \delta = k \in \N$, the measure $\cH^\delta$ coincides with
$k$-dimensional Lebesgue measure and the Mass Transference Principle
allows us to transfer Lebesgue measure theoretic statements for limsup
subsets of $\R^k$ to Hausdorff measure theoretic statements.

\begin{thm}\label{MTP}\index{Mass Transference Principle}
  Let $\{B_i\}_{i\in\N}$ be a sequence of balls in $\Omega$ with
  $r(B_i)\to 0$ as $i\to\infty$. Let $f$ be a dimension function such
  that $x^{-\delta}f(x)$ is monotonic.  For any ball $B \in \Omega$
  with $\cH^\delta(B) > 0$, if
$$
\cH^\delta\big(\/B\cap\limsup_{i\to\infty}B^f_i{}\,\big)=\cH^\delta(B)
\
$$
then
$$
\cH^f\big(\/B\cap\limsup_{i\to\infty}B^\delta_i\,\big)=\cH^f(B) \ .
$$
\end{thm}

\begin{rem} {\rm There is one point that is well worth making. The
    Mass Transference Principle is purely a statement concerning
    limsup sets arising from a sequence of balls. There is absolutely
    no monotonicity assumption on the radii of the balls.  Even the
    imposed condition that $r(B_i)\to0$ as $i\to\infty$ is redundant
    but is included to avoid unnecessary tedious discussion.  }
\end{rem}

\subsubsection{Khintchine's Theorem implies  \Jarnik's Theorem  \label{KimpliesJ}}
First of all let us dispose of the case that
$\psi(r)/r \nrightarrow 0 $ as $r \to \infty$. Then trivially,
$W(\psi)= \I$ and the result is obvious. Without loss of generality,
assume that $\psi(r)/r \to 0 $ as $r \to \infty$. With respect to the
Mass Transference Principle, let $\Omega = \I$, $ d $ be the supremum
norm, $\delta = 1$ and $f(r) = r^s $ with $s \in (0,1)$.  We are given
that $ \sum q^{1-s} \psi(q)^s = \infty $. Let
$ \theta(r) := q^{1-s} \psi(q)^s$. Then $\theta$ is an approximating
function and $ \sum \theta(q) = \infty $. Thus, Khintchine's Theorem
implies that $ \cH^1(B \cap W(\theta)) = \cH^1(B \cap \I)$ for any
ball $B$ in $\R$. It now follows via the Mass Transference Principle
that $ \cH^s(W(\psi)) = \cH^s(\I)= \infty $ and this completes the
proof of the divergence part of \Jarnik's Theorem.  As we have
already seen, the convergence part is straightforward.

\subsubsection{Dirichlet's Theorem implies  the \Jarnik-Besicovitch Theorem }
Dirichlet's theorem (Theorem \ref{thm:7.4}) states that for any
irrational $x \in \R$, there exist infinitely many reduced rationals
$p/q$ ($q>0$) such that $|x - p/q| \leq q^{-2}$; i.e. $W(1)=\I$. Thus,
with $f(r) := r^{d} \ (d:= 2/(1+\tau)) $ the Mass Transference
Principle implies that $\cH^{d} (W(\tau) ) = \infty $. Hence
$\dim W(\tau) \geq d $. The upper bound is trivial. Note that we have
actually proved a lot more than the \Jarnik-Besicovitch
theorem. We have proved that the $s$--dimensional Hausdorff measure
$\cH^{s}$ of $W(\tau)$ at the critical exponent $s=d$ is infinite.

\subsection{The Generalised Duffin-Schaeffer Conjecture}

As with Khintchine's Theorem, it is natural to seek an appropriate
statement in which one removes the monotonicity condition in \Jarnik's
Theorem.  In the case of Khintchine's Theorem, the appropriate
statement is the Duffin-Schaeffer Conjecture -- see \S\ref{dsconj}.
With this in mind, we work with the set $W'(\psi)$ in which the
coprimeness condition $(p,q)=1$ is imposed on the rational
approximates $p/q$.  For any function $\psi:\N\to \Rp$ and
$s \in (0,1]$ it is easily verified that
$$
\mathcal{H}^s\big(W(\psi)\big) = 0 \quad {\rm if} \quad
\sum_{q=1}^\infty \, \varphi(q) \ \Big(\dfrac{\psi(q)}q\Big)^s \ < \
\infty \ .
$$
In the case the above $s$-volume sum diverges it is reasonable to
believe in the truth of the following Hausdorff measure version of the
Duffin-Schaeffer Conjecture \cite{BV06}.

\begin{conjecture}[Generalised Duffin-Schaeffer Conjecture,
  2006] \label{gdsstate} For any function $\psi\colon \N\to \Rp$ and
  $s \in (0,1]$
$$
\mathcal{H}^s\big(W'(\psi)\big) = \mathcal{H}^s\big(\I\big) \quad {\rm
  if} \quad \sum_{q=1}^\infty \, \varphi(q) \
\Big(\dfrac{\psi(q)}q\Big)^s \ = \ \infty \ .
$$
\end{conjecture}

\begin{rem} {\rm If $s=1$, then $\cH^1(\I) = m (\I) $ and Conjecture
    \ref{gdsstate} reduces to the Lebesgue measure conjecture of
    Duffin $\&$ Schaeffer (Conjecture \ref{dsstate}).  }
\end{rem}
\begin{rem} {\rm In view of the Mass Transference Principle, it
    follows that
    $$ \text{ Conjecture \ref{dsstate} } \quad \Longrightarrow \quad
    \text{ Conjecture \ref{gdsstate}} $$ }
\end{rem}

\noindent \emph{Exercise:} Prove the above implication.


\section{The higher dimensional theory}

We start with a generalisation of Theorem \ref{Dir} to simultaneous
approximation in $\R^n$.

\begin{thm}[Dirichlet in $\R^n$]\label{nDir}
  Let $ (i_1, \ldots,i_n) $ be any $n$-tuple of numbers satisfying
  \begin{equation} \label{propi} 0 < i_1, \ldots,i_n < 1 \quad {\rm
      and } \quad \sum_{t=1}^{n} i_t = 1 \, .
  \end{equation}
  Then, for any $\vx=(x_1, \ldots,x_n) \in \RR^n$ and
  $N \in \mathbb{N}$, there exists $q \in \mathbb{Z}$ such that
  \begin{equation}\label{h7.0}
    \max \{ \|qx_1\|^{1/i_1} , \ldots , \ \|qx_n\|^{1/i_n} \,  \}  < N^{-1} \qquad { and }  \qquad  1 \leq  q \leq N \, .
  \end{equation}
\end{thm}

\begin{rem} {\rm The symmetric case corresponding to  $i_1= \ldots=i_n= 1/n$ is the more familiar form of the theorem. In this symmetric case, when $N$ is an $n$'th power, the  one-dimensional proof using the pigeon-hole principle can be easily adapted to prove the associated statement (\emph{exercise}). The above general form is a neat consequence of a fundamental theorem in the geometry of numbers; namely  Minkowski's theorem for systems of linear forms -- see \S\ref{MinT} below.  At this point simply observe that for a fixed  $q$ the first inequality in \eqref{h7.0} corresponds to considering rectangles centered at rational points $$ \Big( \frac{p_1}{q}, \ldots, \frac{p_n}{q} \Big) \quad {\rm of \ sidelength }  \quad \frac{2}{q N^{i_1}} ,\ldots ,\frac{2}{q N^{i_n}}\quad\text{respectively}  \, .
$$   Now the shape of the rectangles are clearly governed by $(i_1, \ldots,i_n)$. However the  volume is not. Indeed, for any  $(i_1, \ldots,i_n)$ satisfying \eqref{propi}, the $n$-dimensional Lebesgue measure $m_n$ of any rectangle centered at a rational point with denominator $q$  is $2^nq^{-n} N^{-1}$.
}
\end{rem}

\subsection{Minkowski's Linear Forms Theorem } \label{MinT}

We begin by introducing various terminology and establishing
Minkowski's theorem for convex bodies.

\begin{df}\rm
  A subset $B$ of $\R^n$ is said to be \emph{convex}\/ if for any two
  points $\vx,\vy\in B$
$$
\big\{\lambda\vx+(1-\lambda)\vy:0\le\lambda\le1\big\}\subset B\,,
$$
that is the line segment joining $\vx$ and $\vy$ is contained in $B$.
A \emph{convex body}\/ in $\R^n$ is a bounded convex set.
\end{df}

\begin{df}\rm
  A subset $B$ in $\R^n$ is said to be \emph{symmetric about the
    origin}\/ if for every $\vx\in B$ we have that $-\vx\in B$.
\end{df}

The following is a simple but nevertheless powerful observation
concerning symmetric convex bodies.

\begin{thm}[Minkowski's Convex Body Theorem]\label{min}
  Let $B$ be a convex body in $\R^n$ symmetric about the origin. If
  $\vol(B)>2^n$ then $B$ contains a non-zero integer point.
\end{thm}

\begin{proof} The following proof is attributed to Mordell.  For
  $m\in\N$ let $ A(m,B)=\{\vv a\in\Z^m: \vv a/m\in B\}\,.  $ Then we
  have that
$$
\lim_{m\to\infty}m^{-n}\#A(m,B)=\vol(B)\, .
$$
Since $\vol(B)>2^n$, there is a sufficiently large $m$ such that
$m^{-n}\#A(m,B)>2^n$, that is $\#A(m,B)>(2m)^n$. Since there are $2m$
different residue classes modulo $2m$ and each point in $A(Q,m)$ has
$n$ coordinates, there are two distinct points in $A(Q,m)$, say
$\vv a=(a_1,\dots,a_n)$ and $\vv b=(b_1,\dots,b_n)$ such that
$$
a_i\equiv b_i\pmod{2m}\quad\text{for each }i=1,\dots,n\,.
$$
Hence
$$
\vv z=\frac12\,\frac{\vv a}{m}+\frac12\left(-\frac{\vv
    b}{m}\right)=\frac{\vv a-\vv b}{2m}\in\Z^n\setminus\{\vv 0\}\,.
$$
Since $B$ is symmetric about the origin, $-\vv b/m\in B$ and since $B$
is convex $\vv z\in B$. The proof is complete.
\end{proof}

The above convex body result enables us to prove the following
extremely useful statement.

\begin{thm}[Minkowski's theorem for systems of linear forms\index{Minkowski's theorem}]\label{SLF}
  Let $\beta_{i,j}\in\R$, where $1\le i,j\le k$, and let
  $C_1,\dots,C_k>0$. If
  \begin{equation}\label{vbx1}
    |\det(\beta_{i,j})_{1\le i,j\le k}|\le \prod_{i=1}^kC_i,
  \end{equation}
  then there exists a non-zero integer point $\vv x=(x_1,\dots,x_k)$
  such that
  \begin{equation}\label{vbx2}
    \left\{\begin{array}{lll}
             |x_1\beta_{i,1}+\dots+x_k\beta_{i,k}|<C_i && (1\le i\le k-1)\\[1ex]
             |x_1\beta_{k,1}+\dots+x_n\beta_{k,k}|\le C_k
           \end{array}\right.
       \end{equation}
     \end{thm}

\begin{proof}
  The set of $(x_1,\dots,x_k)\in\R^k$ satisfying \eqref{vbx2} is a
  convex body symmetric about the origin. First consider the case when
  $\det(\beta_{i,j})_{1\le i,j\le k}\neq0$ and \eqref{vbx1} is strict.
  Then
$$
\vol(B)=\frac{\prod_{i=1}^k(2C_i)}{|\det(\beta_{i,j})_{1\le i,j\le
    k}|}>2^n\,.
$$
Then, by Theorem~\ref{min}, the body contains a non-zero integer point
$(x_1,\dots,x_k)$ as required.

If $\det(\beta_{i,j})_{1\le i,j\le k}=0$ then $B$ is unbounded and has
infinite volume. Then there exists a sufficiently large $m\in\N$ such
that $B_m=B\cap[-m,m]$ has volume $\vol(B_m)>2^n$. Next, $B_m$ is
convex and symmetric about the origin, since it is the intersection of
2 sets with these properties. Again, by Theorem~\ref{min}, $B_m$
contains a non-zero integer point $(x_1,\dots,x_k)$. Since
$B_m\subset B$ we again get the required statement.

Finally, consider the situation when \eqref{vbx1} is an equation. In
this case $\det(\beta_{i,j})_{1\le i,j\le k}\neq0$. Define
$C^\ve_k=C_k+\ve$ for some $\ve>0$. Then
\begin{equation}\label{vbx1++}
  |\det(\beta_{i,j})_{1\le i,j\le k}|< \prod_{i=1}^{k-1}C_i\times C^{\ve}_k
\end{equation}
and by what we have already shown there exists a non-zero integer
solution $\vv x_\ve=(x_1,\dots,x_k)$ to the system
\begin{equation}\label{vbx2++}
  \left\{\begin{array}{lll}
           |x_1\beta_{i,1}+\dots+x_k\beta_{i,k}|<C_i && (1\le i\le k-1)\\[1ex]
           |x_1\beta_{k,1}+\dots+x_n\beta_{k,k}|\le C^{\ve}_k\,.
         \end{array}\right.
     \end{equation}
     For $\ve\le1$ all the points $\vv x_\ve$ satisfy \eqref{vbx2++}
     with $\ve=1$. That is they lie in a bounded body. Hence, there
     are only finitely many of them. Therefore there is a sequence
     $\ve_i$ tending to $0$ such that $\vv x_{\ve_i}$ are all the
     same, say $\vv x_0$. On letting $i\to\infty$ within
     \eqref{vbx2++} we get that \eqref{vbx2} holds with
     $\vv x=\vv x_0$.
   \end{proof}

   \medskip

   It is easily verified that Theorem \ref{nDir} (Dirichlet in $\R^n$)
   is an immediate consequence of Theorem~\ref{SLF} with $k=n+1$ and
$$ C_t =  N^{-i_{t}}  \quad (1\le t\le k-1)    \quad {\rm and }   \qquad C_k = N $$
and
$$
(\beta_{i,j})=\left(\begin{array}{cccccc}
                      -1 & 0 & 0 &  \dots &\alpha_1 \\
                      0 & -1 & 0  & \dots & \alpha_2\\
                      0 & 0 & -1  & \dots & \\
                      \vdots &  &  &   \ddots &\alpha_n \\
                      0 & 0 & 0  & \dots & 1\\
                    \end{array}
                  \right).
$$

Another elegant application of Theorem~\ref{SLF} is the following
statement.

\begin{cor} \label{cordualdirichlet} For any
  $(\alpha_1,\dots,\alpha_n)\in\R^n$ and any real $N>1$, there exist
  $q_1,\dots,q_n,p \in \mathbb{Z}$ such that
$$
| q_1\alpha_1+\dots+q_n\alpha-p| < N^{-n}\qquad\text{and}\qquad 1 \le
\max_{1\le i\le n}|q_i| \le N \,.
$$
In particular, there exist infinitely many
$((q_1,\dots,q_n),p) \in \mathbb{Z}^n\setminus\{0\} \times \mathbb{Z}$
such that
$$
| q_1\alpha_1+\dots+q_n\alpha-p| < \Big(\max_{1\le i\le n}|q_i|
\Big)^{-n} \,.
$$
\end{cor}

\begin{proof} Exercise
\end{proof}

\subsection{$\bad$ in $\R^n$} \label{BadinR^n}

An important consequence of Dirichlet's theorem (Theorem \ref{nDir})
is the following higher dimensional analogue of Theorem \ref{thm:7.4}.

\begin{thm}\label{nthm:7.4}
  Let $ (i_1, \ldots,i_n) $ be any $n$-tuple of real numbers
  satisfying \eqref{propi}. Let
  $\vx=(x_1, \ldots,x_n) \in \mathbb{R}^n$. Then there exist
  infinitely many integers $q>0$ such that
  \begin{equation}\label{h7.3}
    \max \{ \|qx_1\|^{1/i_1} , \ldots , \ \|qx_n\|^{1/i_n} \,  \}  < q^{-1} \, .
  \end{equation}
\end{thm}

\noindent Now just as in the one-dimensional setup we can ask the
following natural question.

\noindent \textbf{Question.} Can we replace the right-hand side of
(\ref{h7.3}) by $ \epsilon q^{-1} $ where $\epsilon > 0$ is arbitrary?

\vspace*{1ex}

\begin{itemize}
\item[] \textbf{No.}  For any $ (i_1, \ldots,i_n) $ satisfying
  \eqref{propi}, there exists \emph{$(i_1, \ldots,i_n)$-badly
    approximable points}.
\end{itemize}

\noindent Denote by $\bad(i_1, \ldots,i_n)$ the set of
$(i_1, \ldots,i_n)$-badly approximable points; that is the set of
$(x_1, \ldots,x_n) \in \RR^n$ such that there exists a positive
constant $ c(x_1, \ldots,x_n) > 0 $ so that \vspace*{-1ex}
\begin{equation*}
  \max \{ \|qx_1\|^{1/i_1} , \ldots , \ \|qx_n\|^{1/i_n} \,  \} >
  c(x_1, \ldots,x_n) \ q^{-1} \quad \forall   q \in \NN   \ .
\end{equation*}

\begin{rem}\label{4.2} {\rm  Let $n=2$ and note that if  $(x,y)\in \bad(i,j)$ for some pair $(i,j)$, then it would imply that
$$
~ \qquad \liminf_{q\to\infty}q\|qx\|\|qy\|=0 .
$$
Hence $\cap_{i+j=1}\bad(i,j)=\varnothing$ would imply that Littlewood's
Conjecture is true. We will return to this famous conjecture in \S\ref{multsec}.  }
\end{rem}

\begin{rem} {\rm Geometrically speaking, $\bad(i_1, \ldots,i_n)$
    consists of points $\vx \in \R^n$ that avoid all rectangles of
    size
    $ c^{i_1} q^{-(1+i_1)} \times \ldots \times c^{i_n} q^{-(1+i_n) }
    $
    centred at rational points $(p_1/q, \ldots, p_n/q)$ with
    $c=c(\vx)$ sufficiently small.  Note that in the symmetric case
    $i_1= \ldots=i_n= 1/n$, the rectangles are squares (or essentially
    balls) and this makes a profound difference when investigating the
    `size' of $\bad(i_1, \ldots,i_n)$ -- it makes life significantly
    easier! }
\end{rem}

Perron \cite{per} in 1921 observed that
$ (x,y) \in \bad(\frac12,\frac12) $ whenever $x$ and $y$ are linearly
independent numbers in a cubic field; e.g
$x=\cos \frac{2\pi}{7}, y= \cos \frac{4\pi}{7}$.  Thus, certainly
$\bad(\frac12,\frac12)$ is not the empty set. It was shown by
Davenport in 1954 that $\Bad(\frac12,\frac12)$ is uncountable and
later in \cite{dav} he gave a simple and more illuminating proof of
this fact.  Furthermore, the ideas in his 1964 paper show that
$\bad(i_1, \ldots,i_n)$ is uncountable.  In 1966, Schmidt
\cite{Schmidt-1966} showed that in the symmetric case the
corresponding set $\bad(\frac1n, \ldots,\frac1n)$ is of full Hausdorff
dimension. In fact, Schmidt proved the significantly stronger
statement that the symmetric set is winning in the sense of his now
famous $(\alpha,\beta)$-games (see \S\ref{sgrt} below). Almost forty
years later it was proved in \cite{Pollington-Velani-02:MR1911218}
that
 $$\dim \bad(i_1, \ldots,i_n)=n  \, . $$

 Now let us return to the symmetric case of Theorem
 \ref{nthm:7.4}.  It implies that every point
 $\vx=(x_1, \ldots,x_n) \in \mathbb{R}^n$ can be approximated by
 rational points $(p_1/q, \ldots p_n/q)$ with rate of approximation
 given by $q^{-{(1+\frac1n})}$. The above discussion shows that this
 rate of approximation cannot in general be improved by an arbitrary
 constant---$ \bad(\frac1n, \ldots,\frac1n) $ is non-empty. However,
 if we exclude a set of real numbers of measure zero, then from a
 measure theoretic point of view the rate of approximation can be
 improved, just as in the one-dimensional setup.

 \subsection{Higher dimensional Khintchine}

 Let $\I^n:=[0,1)^n$ denote the unit cube in $\R^n$ and for
 $ \vx=(x_1, \ldots,x_n) \in \R^n$ let
$$
\|q \vx\| := \max_{1 \le i \le n} \|q x_i\| \, .
$$
Given $\psi:\N\to \Rp$, let
$$
W(n, \psi): = \{\vx\in \I^n \colon \|q \vx\|<\psi(q) \text{ for
  infinitely many } q\in\N \} \
$$
denote the set of \emph{simultaneously $\psi$-well approximable}
points $\vx\in \I^n$.  Thus, a point $\vx\in \I^n$ is $\psi$-well
approximable if there exist infinitely many rational points
$$ \Big( \frac{p_1}{q}, \ldots, \frac{p_n}{q} \Big) $$
with $q >0$, such that the inequalities
$$
\Big|x_i - \frac{p_i}{q} \Big| \, < \, \frac{\psi(q)}{q} \,
$$
are simultaneously satisfied for $ 1 \le i \le n$.  For the same
reason as in the $n=1$ case there is no loss of generality in
restricting our attention to the unit cube.  In the case
$ \psi : q \to q^{-\tau} $ with $\tau > 0 $, we write $ W(n, \tau)$
for $ W(n, \psi)$.  The set $W(n, \tau)$ is the set of
\emph{simultaneously $\tau$-well approximable numbers}.  Note that in
view of Theorem \ref{nthm:7.4} we have that
\begin{equation} \label{simexp+} W (n,\tau) = \I^n \ \ {\rm if \ } \ \
  \tau \leq \frac1n .
\end{equation}

The following is the higher dimensional generalisation of Theorem
\ref{kg} to simultaneous approximation.  Throughout, $m_n$ will denote
$n$-dimensional Lebesgue measure.

\begin{thm}[Khintchine's Theorem in $\R^n$] \label{nkh} Let $\psi:\N\to \Rp$ be
  a monotonic function.  Then
$$ m_n(W(n,\psi)) =\left\{
\begin{array}{ll}
  0 & {\rm if} \;\;\; \sum_{q=1}^{\infty} \;   \psi^n(q)  <\infty\;
      ,\\[4ex]
  1 & {\rm if} \;\;\; \sum_{q=1}^{\infty} \;   \psi^n(q)
      =\infty \; .
\end{array}\right.$$
\end{thm}

\begin{rem} {\rm The convergent case is a straightforward consequence
    of the Convergence Borel-Cantelli Lemma and does not require
    monotonicity.  }
\end{rem}

\begin{rem} \label{GallRemark} {\rm The divergent case is the main
    substance of the theorem.  When $n\ge2$, a consequence of a
    theorem of Gallagher \cite{Gallagher65} is that the monotonicity
    condition can be dropped.  Recall, that in view of the
    Duffin-Schaeffer counterexample (see \S\ref{dsconj}) the
    monotonicity condition is crucial when $n=1$.}
\end{rem}

\begin{rem}  {\rm Theorem \ref{nkh} implies that $$ m_n (W(n, \psi))= 1   \quad   {\rm if }  \quad  \psi(q) =  1/(q \log q)^{\frac1n}  \, . $$   Thus, from a measure theoretic point of view the `rate' of approximation given by Theorem \ref{nthm:7.4} can be improved  by  (logarithm$)^{\frac1n}$.

  }
\end{rem}

\begin{rem} {\rm Theorem \ref{nkh} implies that
    $ m_n(\bad(\frac1n, \ldots,\frac1n)) = 0$.  }
\end{rem}

\begin{rem} \label{4ell} {\rm For a generalisation of Theorem
    \ref{nkh} to Hausdorff measures---that is, the higher dimension
    analogue of Theorem \ref{fulllmhm} (Khintchine-\Jarnik{}
    Theorem))---see Theorem \ref{ohgod} with $m=1$ in \S\ref{DAKT}.
    Also, see \S\ref{ubskj}.  }
\end{rem}

In view of Remark \ref{GallRemark}, one may think that there is
nothing more to say regarding the Lebesgue theory of $\psi$-well
approximable points in $\R^n$.  After all, for $n \ge 2$ we do not
even require monotonicity in Theorem \ref{nkh}. For ease of discussion
let us restrict our attention to the plane $\R^2$ and assume that the
$n$-volume sum in Theorem \ref{nkh} diverges.  So we know that almost
all points $(x_1, x_2)$ are $\psi$-well approximable but it tells us
nothing for a given fixed $x_1$.  For example, are there any points
$(\sqrt2, x_2) \in \R^2$ that are $\psi$-well approximable?  This will
be discussed in \S\ref{KonFibers} and the more general question of
approximating points on a manifold will be the subject of
\S\ref{manifolds}.

\subsection{Multiplicative approximation: Littlewood's
  Conjecture \label{multsec}}

For any pair of real numbers $(\a, \b) \in \I^2$, there exist
infinitely many $q \in \N$ such that
$$\|q\a\| \, \|q \b\| \le q^{-1} \, . $$ This is a simple consequence
of Theorem \ref{nthm:7.4} or indeed the one-dimensional Dirichlet
theorem and the trivial fact that $\|x\| < 1 $ for any $x$.  For any
arbitrary $\epsilon > 0$, the problem of whether or not the statement
remains true by replacing the right-hand side of the inequality by
$ \epsilon \, q^{-1} $ now arises. This is precisely the content of
Littlewood's conjecture.

\vskip 10pt

\begin{thlittle} For any pair $(\a, \b) \in \I^2$,
$$
\liminf_{q \to \infty} q \, ||q\a|| \, ||q\b|| = 0 \; .
$$
\end{thlittle}

\noindent Equivalently, for any pair $(\a, \b) \in \I^2$ there exist
infinitely many rational points $(p_1/q,p_2/q)$ such that
\begin{equation*}
  \Big| \alpha - \frac{p_1}{q}  \Big|  \, \Big|   \beta - \frac{p_2}{q} \Big|  \ < \ \frac{\epsilon}{q^3}   \quad (\epsilon > 0  \ \ {\rm arbitrary})  \, .
\end{equation*}

\medskip

\noindent Thus geometrically, the conjecture states that every point
in the $(x,y)$-plane lies in infinitely many hyperbolic regions given
by $|x| \cdot |y| < \epsilon/q^3$ centred at rational points.

The analogous conjecture in the one-dimensional setting is
false---Hurwitz's theorem tells us that the set $\bad$ is
nonempty. However, in the multiplicative situation the problem is
still open.

\vspace{2ex}

We make various simple observations: \vskip 3pt

\noindent (i) {\em The conjecture is true for pairs $(\a,\b)$ when
  either $\a$ or $\b$ are not in $\Bad$}. Suppose $\b \notin \Bad$ and
consider its convergents $p_n/q_n$.  It follows from the right-hand
side of inequality \eqref{vbc1} that
$ q_n ||q_n \a || \, ||q_n \b || \leq 1/a_{n+1} $ for all $n$. Since
$\b$ is not badly approximable the partial quotients $a_i$ are
unbounded and the conjecture follows. Alternatively, by definition if
$\b \notin \Bad$, then $ \liminf_{q \to \infty} q \, ||q\b|| = 0 $ and
we are done.  See also Remark~\ref{4.2}.

\vskip 3pt

\noindent (ii) {\em The conjecture is true for pairs $(\a,\b)$ when
  either $\a$ or $\b$ lie in a set of full Lebesgue measure}. This
follows at once from Khintchine's theorem.  In fact, one has that for
all $\a$ and almost all $\b \in \I$,
\begin{equation} \label{littlelog} q \, \log q \, \|q\a\| \, \|q\b\|
  \leq 1 \quad{\rm for \ infinitely \ many \ } q \in \N \,
\end{equation}
or even
$$
\liminf_{q \to \infty} q \,\log q \, ||q\a|| \, ||q\b|| = 0 \; .
$$

\vskip 4pt We now turn our attention to `deeper' results regarding
Littlewood.

\vskip 4pt

\noindent \textbf{Theorem (Cassels $\!\! \& \!\!$ Swinnerton-Dyer,
  1955).} {\em If $\a, \b$ are both cubic irrationals in the same
  cubic field then Littlewood's Conjecture is true.}

\medskip

This was subsequently strengthened by Peck \cite{Pe}.

\vskip 4pt

\noindent \textbf{Theorem (Peck, 1961).} {\em If $\a, \b$ are both
  cubic irrationals in the same cubic field then $(\a, \b)$ satisfy
  \eqref{littlelog} with the constant $1$ on the right hand side
  replaced by a positive constant dependent on $\a$ and $\b$.  }

\medskip

\noindent In view of (ii) above, when dealing with Littlewood we can
assume without loss of generality that both $\a$ and $\b$ are in
$\Bad$.  As mentioned in Chapter \ref{intro}, it is conjectured (the
Folklore Conjecture) that the only algebraic irrationals which are
badly approximable are the quadratic irrationals. Of course, if this
conjecture is true then the Cassels \& Swinnerton--Dyer result follows
immediately. On restricting our attention to just badly approximable
pairs we have the following statement \cite{PV2000}.

\vskip 4pt

\noindent \textbf{Theorem PV (2000).} {\em Given $\a \in \bad$ we have
  that }
$$
\dim \big( \, \{ \b \in \bad: (\a,\b) \ {\rm satisfy \ }
\eqref{littlelog} \} \, \big) = 1 \, .
$$

\medskip

Regarding, potential counterexamples to Littlewood we have the
following elegant statement \cite{EKL}.

\vskip 4pt

\noindent \textbf{Theorem EKL (2006).}
$ \quad \dim \big( \{ (\a,\b) \in \I^2: \displaystyle{\liminf_{q \to
    \infty}} \ q \, ||q\a|| \, ||q\b|| > 0 \} \big) = 0 .  $

Now let us turn our attention to non-trivial, purely metrical
statements regarding Littlewood.  The following result due to
Gallagher \cite{Ga} is the analogue of Khintchine's simultaneous
approximation theorem (Theorem \ref{nkh}) within the multiplicative
setup.  Given $\psi:\N\to \Rp$ let
\begin{equation} \label{multdef} W^{\times}(n,\psi): = \{ \vx \in {\rm
    I}^n \colon \| q x_1 \| \, \ldots \, \| q x_n \| <\psi(q) \text{
    for infinitely many } q\in\N \} \
\end{equation}
denote the set of multiplicative $\psi$-well approximable points
$\vx\in \I^n$.

\begin{thm}[Gallagher, 1962]\label{gallmult}
  Let $\psi:\N\to \Rp$ be a monotonic function.  Then
$$ m_n(W^{\times}(n,\psi)) =\left\{
\begin{array}{ll}
  0 & {\rm if} \;\;\; \sum_{q=1}^{\infty} \;   \psi(q) \log^{n-1} q  <\infty\;
      ,\\[4ex]
  1 & {\rm if} \;\;\; \sum_{q=1}^{\infty} \;    \psi(q) \log^{n-1} q
      =\infty \; .
\end{array}\right.$$
\end{thm}

\vspace*{2ex}

\begin{rem}\label{nonmonremark} {\rm In the case of convergence, we
    can remove the condition that $\psi$ is monotonic if we replace
    the above convergence condition by
    $\sum \psi(q) \, |\log \psi(q)|^{n-1} <\infty \, $; see \cite{BHV01}
    for more details. }
\end{rem}

\vspace*{1ex}

An immediate consequence of Gallagher's Theorem is that almost all
$(\alpha,\beta) $ beat Littlewood's Conjecture by `log squared';
equivalently, almost surely Littlewood's Conjecture is true with a
`log squared' factor to spare.

\begin{cor}\label{corgall}
  For almost all $ (\a,\b) \in \R^2$
  \begin{equation} \label{littlelog2} \liminf_{q \to \infty} q
    \,\log^2 q \, ||q\a|| \, ||q\b|| = 0 \; .
  \end{equation}
\end{cor}

Recall, that this is beyond the scope of what Khintchine's theorem can
tell us; namely that
\begin{equation}\label{log}
  \liminf_{q \to \infty} q \,\log  q \, ||q\a|| \, ||q\b|| = 0  \quad \forall   \  \alpha \in \R  \ \quad  \mbox{and}   \quad \  \mbox{for almost all } \beta \in \R  \, .
\end{equation}

\noindent However the extra $\log$ factor in the corollary comes at a
cost of having to sacrifice a set of measure zero on the $\alpha$
side. As a consequence, unlike with (\ref{log}) which is valid for any
$\alpha$, we are unable to claim that the stronger `$\log$ squared'
statement (\ref{littlelog2}) is true for say when $\alpha = \sqrt2$.
Obviously, the role of $\alpha$ and $\beta$ in (\ref{log}) can be
reversed. This raises the natural question of whether
(\ref{littlelog2}) holds for every $\alpha$.  If true, it would mean
that for any $\alpha$ we still beat Littlewood's Conjecture by `log
squared' for almost all $\beta$.

\subsubsection{Gallagher on fibers \label{Gonfibres} }

The following result is established in \cite{bhv}.

\begin{thm}\label{t8}
  Let $\alpha \in \I$ and $\psi:\N\to\Rp$ be a monotonic function such
  that
  \begin{equation}\label{yy}
    \sum_{q=1}^{\infty} \,  \psi (q)\,\log q   \, = \, \infty
  \end{equation}
  and such that
  \begin{equation}\label{yy++}
    \exists\ \delta>0\qquad \liminf_{n\to\infty}q_n^{3-\delta}\psi(q_n)\ge 1\, ,
  \end{equation}
  where $q_n$ denotes the denominators of the convergents of $\alpha$.
  Then for almost every $\beta\in\I$, there exists infinitely many
  $q\in\N$ such that
  \begin{equation}\label{ineq}
    \|q\alpha\| \,  \|q\beta\| <\psi(q) \, .
  \end{equation}
\end{thm}

\begin{rem} {\rm Condition \eqref{yy++} is not particularly
    restrictive. It holds for all $\alpha$ with Diophantine exponent $\tau(\alpha)<3$.
    By definition,
    $$
    \tau(x)=\sup\{\tau>0:\|q\alpha\|<q^{-\tau}\quad\text{for infinitely many }q\in\N\}\,.
    $$
    Recall
    that by the \Jarnik-Besicovitch theorem (Theorem~\ref{jb}), the
    complement is of relatively small dimension; namely
    $ \dim\{\alpha\in\R:\tau(\alpha)\ge 3\}=\frac{1}{2}\,.  $ }
\end{rem}

\noindent The theorem can be equivalently formulated as
follows. Working within the $(x,y)$-plane, let $ {\rm L}_{x} $ denote
the line parallel to the $ y$-axis passing through the point
$(x,0)$. Then, given $\alpha \in \I$, Theorem \ref{t8} simply states
that
$$
m_1(W^{\times}(2,\psi) \, \cap \, {\rm L}_{\a} ) = 1 \quad {\rm if \ \
  \ \psi \ statisfies \ \eqref{yy}\text{ and }\eqref{yy++}. }
$$

An immediate consequence of the theorem is that (\ref{littlelog2})
holds for every $\alpha$ as desired.

\begin{cor}\label{corgall2}
  For every $\alpha\in \R$ one has that
$$
\liminf_{q \to \infty} q \,\log^2 q \, ||q\a|| \, ||q\b|| = 0 \quad \
\mbox{for almost all } \beta \in \R \, .
$$
\end{cor}

\noindent \textsc{Pseudo sketch proof of Theorem \ref{t8}. }  Given
$\alpha$ and $\psi$, rewrite \eqref{ineq} as follows:
\begin{equation}\label{ineqslv}
  \|q\beta\| <\Psi_{\alpha}(q) \, \quad {where}  \quad   \Psi_{\alpha}(q):= \frac{ \psi(q)} {\|q\alpha\|}   \, .
\end{equation}
We are given \eqref{yy} rather than the above divergent sum
condition. So we need to show that
\begin{equation}\label{rubish2}
  \sum_{q=1}^{\infty} \,  \psi (q)\,\log q   \, = \, \infty \, \quad  \Longrightarrow  \quad   \sum_{q=1}^{\infty} \Psi_{\alpha}(q)  \, = \, \infty \, \, .
\end{equation}
This follows (\emph{exercise}) on using partial summation together
with the following fact established in \cite{bhv}. For any irrational
$\alpha$ and $Q \ge 2 $
\begin{equation}\label{rubish3}
  \sum_{q=1}^{Q} \,  \frac{ 1} {\|q\alpha\|}  \, \ge  \,  2 \, Q \log Q \, \, .
\end{equation}
This lower bound estimate strengthens a result of Schmidt
\cite{Schmidt64} -- his result is for almost all $\alpha$ rather than
all irrationals.  Now, if $\Psi_{\alpha}(q)$ was a monotonic function
of $q$ we could have used Khintchine's Theorem, which would then imply
that
\begin{equation}\label{rubish}
  m_1(W(\Psi_\alpha) )   = 1  \quad {\rm if \ \ \ }   \sum_{q=1}^{\infty} \Psi_{\alpha}(q)  \, = \, \infty \, \, .
\end{equation}
Unfortunately, $\Psi_{\alpha}$ is not monotonic. Nevertheless, the
argument given in \cite{bhv} overcomes this difficulty.
\hfill\raisebox{-1ex}{$\boxtimes$}

\medskip

It is worth mentioning that Corollary \ref{corgall2} together with
Peck's theorem and Theorem PV adds weight to the argument made in
\cite{Badziahin-Velani-MAD} for the following strengthening of
Littlewood's Conjecture.

\begin{conjecture} For any pair $(\a, \b) \in \I^2$,
$$
\liminf_{q \to \infty} q \,\log q \, ||q\a|| \, ||q\b|| < + \infty \;
.
$$
\end{conjecture}

\noindent Furthermore, it is argued in \cite{Badziahin-Velani-MAD} that the
natural analogue of $\bad$ within the multiplicative setup is the set:
$$
\mad:=\{(\alpha,\beta)\in \R^2\;:\; \liminf_{q\to\infty} q \cdot \log
q \cdot ||q\alpha||\cdot ||q\beta||>0\}.
$$

\medskip

Regarding the convergence counterpart to Theorem~\ref{t8}, the following statement is established in \cite{bhv}.

\begin{thm}\label{t12}
Let $\alpha \in \R$ be any irrational real number and let $\psi:\N\to\Rp$ be such that
$$
\sum_{q=1}^{\infty} \,  \psi (q)\,\log q   \, < \, \infty
$$
Furthermore,  assume either of the following two conditions\,{\rm:}\\[-5ex]
\begin{itemize}
  \item[{\rm(i)}] $n\mapsto n\psi(n)$ is decreasing and
\begin{equation}\label{cond}
\sum_{n=1}^N\frac{1}{n\|n\alpha\|} \ll(\log N)^2\qquad\text{for all  $N \ge 2 $}\,;
\end{equation}
  \item[{\rm(ii)}] $n\mapsto \psi(n)$ is decreasing and
  \begin{equation}\label{cond+}
\sum_{n=1}^N\frac{1}{\|n\alpha\|} \ll N\log N \qquad\text{for all  $N \ge 2 $}\,.
\end{equation}
\end{itemize}
Then for almost all $\beta\in\R$, there exist only finitely many $n\in\N$ such that
\begin{equation}\label{ineq+}
\|n\alpha\| \,  \|n\beta\| <\psi(n) \, .
\end{equation}
\end{thm}

The behaviour of the sums \eqref{cond} and \eqref{cond+} is explicitly studied in term of the continued fraction expansion of $\alpha$. In particular, it is shown in \cite{bhv} that \eqref{cond} holds for almost all real numbers  $\alpha$ while \eqref{cond+} fails for almost all real numbers $\alpha$. An intriguing question formulated in \cite{bhv} concerns the behaviour of the above sums for algebraic $\alpha$ of degree $\ge3$. In particular, it is conjectured that \eqref{cond} is true for  any real algebraic number  $\alpha$ of degree $\ge3$. As is shown in \cite{bhv},  this is equivalent to the following  statement.

\begin{conjecture}\label{conj2}
For any algebraic $\alpha=[a_0;a_1,a_2,\dots]  \in \R\setminus\Q$, we have that
$$
\sum_{k=1}^n a_k \ll  n^2\,.
$$
\end{conjecture}

\medskip

\begin{rem}
Computational evidence for specific algebraic numbers does support this conjecture \cite{Sky}.
\end{rem}

\subsection{Khintchine on fibers \label{KonFibers} }

In this section we look for a strengthening of Khintchine simultaneous
theorem (Theorem \ref{nkh}) akin to the strengthening of Gallagher's
multiplicative theorem described above in \S\ref{Gonfibres}. For ease
of discussion, we begin with the case that $n=2$ and whether or not
Theorem \ref{nkh} remains true if we fix $ \alpha \in \I$.  In other
words, if $ {\rm L}_{\alpha} $ is the line parallel to the $ y$-axis
passing through the point $(\alpha,0)$ and $\psi$ is monotonic, then
is it true that
$$ m_1(W(2,\psi)  \cap {\rm L}_{\alpha} ) =\left\{
\begin{array}{ll}
  0 & {\rm if} \;\;\; \sum_{q=1}^{\infty} \;   \psi^2(q)  <\infty\;
  \\[4ex]
  1 & {\rm if} \;\;\; \sum_{q=1}^{\infty} \;   \psi^2(q)
      =\infty \;
\end{array}\right.   \ \  \textbf{????} $$
The question marks are deliberate. They emphasize that the above
statement is a question and not a fact or a claim.  Indeed, it is easy
to see that the convergent statement is false.  Simply take $\alpha $
to be rational, say, $\alpha = \frac{a}{b}$.  Then, by Dirichlet's
theorem, for any $\beta $ there exist infinitely many $ q \in \N$ such
that $\| q \beta \| < q^{-1}$ and so it follows that
$$
\| bq \beta \| < \frac{b}{q} = \frac{b^2}{bq} \quad {\rm and } \quad
\| bq \alpha \| = 0 < \frac{b^2}{bq} \, .
$$
This shows that every point on the rational vertical line
${\rm L}_{\alpha} $ is $\psi(q) = b^2q^{-1} $ -~approximable and so
$$ m_1(W(2,\psi) \cap {\rm L}_{\alpha} ) = 1 \quad {\rm but } \quad
\sum_{q=1}^{\infty} \; \psi^2(q) = \sum_{q=1}^{\infty} b^4q^{-2} <
\infty \ .
$$

Now, concerning the divergent statement, we claim it is true.

\begin{conjecture}Let $\psi:\N\to \Rp$ be a monotonic
  function and $\alpha \in \I$. Then
  \begin{equation}\label{fibconj}
    m_1(W(2,\psi)  \cap {\rm L}_{\alpha} ) =
    1    \quad  {\rm if}    \quad  \sum_{q=1}^{\infty} \;   \psi^2(q)
    =\infty   \, .
  \end{equation}
\end{conjecture}

In order to state the current results, we need the notion of the
Diophantine exponent of a real number.  For $\vv x \in \R^n$, we let
\begin{equation}\label{dioexp}
  \tau(\vv x ) :=  \sup \{ \tau :  \vv x \in W(n, \tau) \}
\end{equation}
denote the \emph{Diophantine exponent of $\vv x $}. A word of warning,
this notion of Diophantine exponent should not be confused with the
Diophantine exponents introduced later in \S\ref{patel}.  Note that in
view of \eqref{simexp+}, we always have that $\tau(\vv x ) \geq 1/n $.
In particular, for $\alpha \in \R$ we have that
$\tau(\alpha) \geq 1 $.  The following result is established
in~\cite{crss}.

\begin{thm}[F. Ram\'{\i}rez, D. Simmons, F. S\"{u}ess]\label{tkfib2}
  Let $\psi:\N\to \Rp$ be a monotonic function and $\alpha \in \I$.
  \begin{enumerate}
  \item[\rm \textbf{A.}] If $\tau(\alpha) < 2$, then \eqref{fibconj}
    is true.
  \item[\rm \textbf{B.}] If $\tau(\alpha) > 2$ and for $\epsilon > 0$,
    $\psi(q) > q^{-\frac{1}{2}-\epsilon} $ for $q$ large enough, then
    $ W(2,\psi) \cap {\rm L}_{\alpha} = \I^2 \cap {\rm L}_{\alpha} $.
    In particular, $m_1(W(2,\psi) \cap {\rm L}_{\alpha} ) = 1 $.
  \end{enumerate}
\end{thm}

\medskip

\begin{rem}
  Though we have only stated it for lines in the plane,
  Theorem~\ref{tkfib2} is actually true for lines in $\R^n$. There, we
  fix an $(n-1)$-tuple of coordinates
  $\boldsymbol{\alpha} = (\alpha_1, \dots, \alpha_{n-1})$, and
  consider the line $L_{\boldsymbol{\alpha}}\subset\R^n$. We obtain
  the same result, with a ``cut-off'' at $n$ in the \emph{dual
    Diophantine exponent} of $\boldsymbol{\alpha}\in\R^{n-1}$. The dual
  Diophantine exponent $\tau^*(\vv x )$  of a vector $ \vv x \in \R^n$ is defined similarly to the
  (simultaneous) Diophantine exponent, defined above by \eqref{dioexp}, and in the case of
  numbers (\emph{i.e.}, one-dimensional vectors), the two notions
  coincide -- see \S\ref{patel} for the formal definition of $\tau^*(\vv x )$.
\end{rem}

\medskip

\begin{rem}
  This cut-off in Diophantine exponent, which in Theorem~\ref{tkfib2}
  happens at $\tau(\alpha)=2$, seems quite unnatural: why should real
  numbers with Diophantine exponent $2$ be special? Still, such points are
  inaccessible to our methods. We will see the obstacle in the
  counting estimate~\eqref{count} which is used for the proof of Part
  A and is unavailable for $\tau(\alpha)=2$, and in our application of
  Khintchine's Transference Principle for the proof of Part B.
\end{rem}

\medskip

\begin{rem} {\rm Note that in Part B, the `in particular' full measure
    conclusion is immediate and does not even require the divergent
    sum condition associated with \eqref{fibconj}.  }
\end{rem}

Regarding the natural analogous conjecture for higher-dimensional
subspaces, we have the following statement from~\cite{crss} which
provides a complete solution in the case of affine co-ordinate
subspaces of dimension at least two.

\begin{thm}\label{tkfibn}
  Let $\psi:\N\to \Rp$ be a monotonic function and given
  $ \bm{\alpha} \in \I^{n-d}$ where $ 2 \le d \le n-1$, let
  ${\rm L}_{\bm{\alpha}} := \{\bm{\alpha}\} \times \R^d$.  Then
  \begin{equation}\label{fibconjn}
    m_d(W(n,\psi)  \cap {\rm L}_{\bm{\alpha}} ) =
    1    \quad  {\rm if}    \quad  \sum_{q=1}^{\infty} \;   \psi^n(q)
    =\infty   \, .
  \end{equation}
\end{thm}

\begin{rem}
  Notice that Theorem~\ref{tkfibn} requires $d\geq 2$, thereby
  excluding lines in $\R^n$. In this case, the obstacle is easy to
  describe: the proof of Theorem~\ref{tkfibn} relies on Gallagher's
  extension of Khintchine's theorem, telling us that the monotonicity
  assumption can be dropped in higher dimensions (see Remark
  \ref{GallRemark}). In the proof of Theorem~\ref{tkfibn} we find a
  natural way to apply this directly to the fibers, therefore, we must
  require $d \geq 2$.

  But this is again only a consequence of the chosen method of proof,
  and not necessarily a reflection of reality.  Indeed,
  Theorem~\ref{tkfib2} (and its more general version for lines in
  $\R^n$) suggests that we should be able to relax
  Theorem~\ref{tkfibn} to include the case where $d=1$.
\end{rem}

\begin{rem}
  The case when $d=n-1$ was first treated in~\cite{ram}. There, a
  number of results are proved in the direction of
  Theorem~\ref{tkfibn}, but with various restrictions on Diophantine
  exponent, or on the approximating function.
\end{rem}

Regarding the proof of Theorem \ref{tkfib2}, Part B makes use of
Khintchine's Transference Principle (see \S\ref{patel} below) while
the key to establishing Part~A is the following measure theoretic
statement (cf. Theorem \ref{thm:slv1}) and ubiquity (see \S\ref{ubi}
below).

\begin{prop}\label{thm:slv45} Let $\psi:\N\to \Rp$ be a monotonic
  function such that for all $\epsilon > 0$ we have
  $\psi(q) > q^{-\frac{1}{2}-\epsilon}$ for all $q$ large enough.  Let
  $\alpha \in \R$ be a number with Diophantine exponent
  $\tau(\alpha) < 2$.  Then for any $0 < \epsilon < 1$ and integer
  $ k \ge k_0(\epsilon)$, we have that
$$m_1\left( \bigcup_{k^{n-1} < q \le k^n: \atop{\| q \alpha \| \le   \psi(k^n)} }  \ \ \ \bigcup_{p=0}^q
  \ \textstyle{B\left(\frac{p}{q}, \frac{k}{k^{2n} \psi(k^{n} )}
    \right) } \right) \ \geq \ 1- \epsilon \, .
$$
\end{prop}

\begin{rem} \label{re27}

  Note that within the context of Theorem \ref{tkfib2}, since
  $ \alpha$ is fixed it is natural to consider only those $q \in \N$
  for which $\| q \alpha \| \le \psi(q) $ when considering solutions
  to the inequality $\| q \beta \| \le \psi(q) $.  In other words, if
  we let
$$
{\cal A}_{\alpha}(\psi) := \{ q \in \N: \| q \alpha \| \le \psi(q) \}
$$
then by definition
$$
W(2,\psi) \cap {\rm L}_{\alpha} = \{ (\alpha,\beta) \in {\rm
  L}_{\alpha} \cap \I^2: \| q \beta \| \le \psi(q) \ {\rm for \
  infinitely \ many \ } q \in {\cal A}_{\alpha}(\psi) \} \, .
$$
It is clear that the one-dimensional Lebesgue measure $m_1$ of this set is the same as that of
$$\{ \beta \in \I :  \| q \beta \| \le   \psi(q) \ {\rm for \ infinitely  \  many \ } q \in {\cal A}_{\alpha}(\psi) \} \, .
$$
\end{rem}

\noindent \textsc{Sketch proof of Proposition \ref{thm:slv45}. }  In
view of Minkowski's theorem for systems of linear forms, for any
$(\alpha, \beta) \in \R^2 $ and integer $N \ge 1$, there exists an
integer $q \ge 1$ such that
\begin{eqnarray*} \label{slv1011}
  \| q \alpha \|  & \le  &    \psi(N)    \nonumber \\[1ex]
  \| q \beta \|  & \le  &     \frac{1}{N\, \psi(N)  }   \nonumber \\[1ex]
  q & \le & N \, .
\end{eqnarray*}
The desired statement follows on exploiting this with $N= k^n $
together with the following result which is a consequence of a general
counting result established in~\cite{bhv}: given $\psi$ and $\alpha$
satisfying the conditions imposed in Proposition~\ref{thm:slv45}, then
for $n$ sufficiently large
\begin{equation}\label{count}
  \# \{  q  \le k^{n-1} : \| q \alpha \| \le   \psi(k^n) \}    \; \le \; 31 \,  \psi(k^n) \, k^{n-1}   \, .
\end{equation}
(An analogous count is established in~\cite{crss} for vectors $\boldsymbol{\alpha}\in\R^{n-1}$.)
\noindent \emph{Exercise:} Fill in the details of the above sketch.
\hfill\raisebox{-1ex}{$\boxtimes$}


\subsection{Dual approximation and Khintchine's
  Transference \label{DAKT}}

Instead of simultaneous approximation by rational points as considered
in the previous section, one can consider the closeness of the point
$\vv x=(x_1,\dots,x_m)\in\R^m$ to rational hyperplanes given by the
equations $\vv q\cdot\vv x=p$ with $p\in\Z $ and $\q\in\Z^m$. The
point $\vv x\in\R^n$ will be called \emph{dually $\psi$-well
  approximable}\/ if the inequality
$$
|\vv q\cdot \vv x-p|<\psi(|\vv q|)
$$
\noindent holds for infinitely many $(p,\vv q)\in\Z\times\Z^m$ with
$|\vv q|:=|\vv q|_\infty = \max\{|q_1|,\dots,|q_m|\}>0$. The set of
dually $\psi$-approximable points in $\I^m$ will be denoted by
$W^*(m,\psi)$.  In the case $ \psi : q \to q^{-\tau} $ with
$\tau > 0 $, we write $ W^*(m, \tau)$ for $ W^*(m, \psi)$.  The set
$W^*(n, \tau)$ is the set of \emph{dually $\tau$-well approximable
  numbers}.  Note that in view of Corollary \ref{cordualdirichlet} we
have that
\begin{equation} \label{dualexp+} W^*(m,\tau) = \I^m \ \ {\rm if \ } \
  \ \tau \leq m .
\end{equation}

The simultaneous and dual forms of approximation are special cases of
a system of linear forms, covered by a general extension due to
A.~V.~Groshev (see \cite{Sprindzuk}).  This treats real $m \times n$
matrices $X$, regarded as points in $\R^{mn}$, which are
$\psi$-approximable.  More precisely, $X=(x_{ij}) \in \R^{mn}$ is said
to be $\psi$-approximable if the inequality
$$
\| \q X \| < \psi(|\q|)
$$
\noindent is satisfied for infinitely many $\q \in \Z^m$. Here $\q X$
is the system
$$
q_1x_{1j} + \dots + q_m x_{m,j} \hspace*{6ex} ( 1\leq j \leq n )
$$
of $n$ real linear forms in $m$ variables and
$\| \q X \| := \max_{1\leq j \leq n } \| \q \cdot X^{(j)} \|$, where
$ X^{(j)}$ is the $j$'th column vector of $X$. As the set of
$\psi$-approximable points is translation invariant under integer
vectors, we can restrict attention to the $mn$-dimensional unit cube
$\I^{mn}$. The set of $\psi$-approximable points in $\I^{mn}$ will be
denoted by
$$
W(m,n,\psi) := \{X\in \I^{mn}:\|\q X \| < \psi(|\q|) {\text { for
    infinitely many }} \q \in \Z^m \}.
$$
Thus, $W(n,\psi)=W(1,n,\psi)$ and $W^*(m,\psi)=W(m,1,\psi)$. The
following result naturally extends Khintchine's simultaneous theorem
to the linear forms setup. For obvious reasons, we write $|X |_{mn} $
rather than $m_{mn}(X) $ for $mn$-dimensional Lebesgue measure of a
set $X \subset \R^{mn}$.

\begin{thm}[Khintchine-Groshev, 1938]
  \label{Groshevthm} Let $\psi:\N\to\Rp$. Then
$$
|W(m,n,\psi)|_{mn} = \begin{cases} 0
  & \ \ds\text{if } \  \sum_{r =1}^\infty r^{m-1}\psi(r)^n<\infty, \\[3ex]
  1 & \ \ds\text{if } \ \sum_{r=1}^\infty r^{m-1}\psi(r)^n=\infty \
  \text{ and $\psi$ is monotonic}.
\end{cases}
$$
\end{thm}

The counterexample due to Duffin and Schaeffer mentioned in
\S\ref{dsconj} means that the monotonicity condition cannot be dropped
from Groshev's theorem when $m=n=1$. To avoid this situation, let
$mn> 1$. Then for $m=1$, we have already mentioned (Remark
\ref{GallRemark}) that the monotonicity condition can be removed.
Furthermore, the monotonicity condition can also be removed for
$m > 2$ -- see \cite[Theorem 8]{Beresnevich-Bernik-Dodson-Velani-Roth}
and \cite[Theorem 14]{Sprindzuk}.  The $m=2$ situation was resolved
only recently in \cite{BVkg}, where it was shown that the monotonicity
condition can be safely removed.  The upshot of this discussion is
that we only require the monotonicity condition in the
Khintchine-Groshev theorem in the case when $mn=1$.

Naturally, one can ask for a Hausdorff measure generalisation of the
Khintchine-Groshev theorem.  The following is such a statement and as
one should expect it coincides with Theorem~\ref{fulllmhm} when
$m=n=1$. In the simultaneous case ($m=1$), the result was alluded to
within Remark \ref{4ell} following the simultaneous statement of
Khintchine's theorem.

\begin{thm}
  \label{ohgod} Let $\psi:\N\to\Rp$. Then
$$
\cH^s(W(m,n,\psi))=\left\{\begin{array}{lll} 0 & \ds\text{if } \;\;\;
    \sum_{r=1}^\infty \;r^{m(n+1)-1-s}\psi(r)^{s-n(m-1)} <\infty \, ,
    \\[2ex] & \\ \cH^s(\I^{mn}) & \ds\text{if } \;\;\;
                                  \sum_{r=1}^\infty \; r^{m(n+1)-1-s}\psi(r)^{s-n(m-1)} =\infty \ \\
                                               & ~ \hspace*{18ex}
                                                 \text{ and $\psi$ is
                                                 monotonic}\,. &
                          \end{array}\right.
$$
\end{thm}

This Hausdorff theorem follows from the corresponding Lebesgue
statement in the same way that Khintchine's theorem implies
\Jarnik's theorem via the Mass Transference Principle---see
\S\ref{KimpliesJ}. The Mass Transference Principle introduced in
\S\ref{mtpballs} deals with $\limsup$ sets which are defined by a
sequence of balls. However, the `slicing' technique introduced in
\cite{BV06Slicing} extends the Mass Transference Principle to deal
with $\limsup$ sets defined by a sequence of neighborhoods of
`approximating' planes.  This naturally enables us to generalise the
Lebesgue measure statements for systems of linear forms to Hausdorff
measure statements.  The last sentence should come with a warning.  It
gives the impression that in view of the discussion preceding Theorem
\ref{Groshevthm}, one should be able to establish Theorem \ref{ohgod}
directly, without the monotonicity assumption except when
$m=n=1$. However, as things currently stand we also need to assume
monotonicity when $m=2$. For further details see~\cite[\S8]{Beresnevich-Bernik-Dodson-Velani-Roth}.

Returning to Diophantine approximation in $\R^n$, we consider the
following natural question.

\noindent \textbf{Question.} Is there a connection between the
simultaneous ($m=1$) and dual ($n=1$) forms of approximating points in
$\R^n$?

\subsubsection{Khintchine's Transference \label{patel}}

The simultaneous and dual forms of Diophantine approximation are
related by a `transference' principle in which a solution of one form
is related to a solution of the other.  In order to state the
relationship we introduce the quantities $ \omega^*$ and
$ \omega$.  For $ \vv x=(x_1,\dots,x_n)\in\R^n$, let
\begin{equation*}
  \omega^*(\vv x) \, :=  \, \sup\left\{ \omega \in \R : \vv x \in W^*(n,n + \omega)  \right\}
\end{equation*}
and
\begin{equation*}
  \omega(\vv x) \, :=  \, \sup\left\{ \omega \in \R : \vv x \in W(n,  \textstyle{\frac{1+ \omega}{n}})   \right\}   \, .
\end{equation*}
Note that
 $$
 \tau(\vv x ) = \frac{ 1+ \omega(\vv x)}{n} \
 $$
 where $ \tau (\vv x)$ is the Diophantine exponent of $\vv x$ as defined
 by \eqref{dioexp}. For the sake of completeness we mention that the quantity
  $$
 \tau^*(\vv x ) = n+\omega^*(\vv x)
 $$
 is called the dual Diophantine exponent.
The following statement provides a relationship
 between the dual and simultaneous Diophantine exponents.
 \begin{thm}[Khintchine's Transference Principle] \label{KTPP} For
   $ \vv x \in \R^n$, we have that
   \begin{equation*}
     \frac{ \omega^*(\vv x) }{ n^2 + (n-1) \omega^*(\vv x)} \leq      \omega(\vv x)  \leq \omega^*(\vv x)   \,
   \end{equation*}
   with the left hand side being interpreted as $1/(n-1)$ if $ \omega^*(\vv x)$ is infinite.
 \end{thm}

 \begin{rem} \label{reKTPP} {\rm The transference principle implies
     that given any $ \epsilon > 0$, if
     $ \vv x \in W(n, \textstyle{\frac{1+ \epsilon}{n}}) $ then
     $ \vv x \in W^*(n, n + \epsilon^*) $ for some $\epsilon^*$
     comparable to $\epsilon$, and vice versa.  }
 \end{rem}

 \noindent \textbf{Proof of Part B of Theorem \ref{tkfib2}}

 Part B of Theorem \ref{tkfib2} follows by plugging $n=2$ and $d=1$
 into the following proposition, which is in turn a simple consequence
 of Khintchine's Transference Principle.

\begin{prop}\label{propCRSS}
  Let $\psi:\N\to \Rp$ be a monotonic function and given
  $ \bm{\alpha} \in \I^{n-d}$ where $ 1 \le d \le n-1$, let
  ${\rm L}_{\bm{\alpha}} := \{\bm{\alpha}\} \times \R^d$.  Assume that
  $\tau({\bm{\alpha}}) > \frac{1+d}{n-d} $ and for $\epsilon > 0$,
  $\psi(q) > q^{-\frac{1}{n}-\epsilon} $ for $q$ large enough.  Then
$$
W(n,\psi) \cap {\rm L}_{\bm{\alpha}} = \I^n \cap {\rm L}_{\bm{\alpha}}
\, . $$
In particular, $ m_d(W(n,\psi) \cap {\rm L}_{\bm{\alpha}} ) = 1 $.
\end{prop}

\begin{proof}

  We are given that $\tau({\bm{\alpha}}) > \frac{1+d}{n-d} $ and so by
  definition $\omega({\bm{\alpha}}) > d $.  Thus, by Khintchine's
  Transference Principle, it follows that
  $\omega^*({\bm{\alpha}}) > d $ and so $\omega^*(\vv x) > 0 $ for any
  point $ \vv x = ({\bm{\alpha}}, {\bm{\beta}}) \in \R^n $; i.e.
  $ {\bm{\beta}} \in \R^d$ and $\vv x$ is a point on the
  $d$-dimensional plane $ {\rm L}_{\bm{\alpha}}$. On applying
  Khintchine's Transference Principle again, we deduce that
  $\omega(\vv x) > 0 $ which together with the growth condition
  imposed on $\psi$ implies the desired conclusion.
\end{proof}

\section{Ubiquitous systems of points \label{ubi}}

In \cite{BDV06}, a general framework is developed for establishing
divergent results analogous to those of Khintchine and \Jarnik{} for
a natural class of $\limsup $ sets.  The framework is based on the
notion of `ubiquity', which goes back to \cite{BS70} and \cite{DRV}
and captures the key measure theoretic structure necessary to prove
such measure theoretic laws.  The `ubiquity' introduced below is a
much simplified version of that in \cite{BDV06}.  In particular, we
make no attempt to incorporate the linear forms theory of metric
Diophantine approximation. However this does have the advantage of
making the exposition more transparent and also leads to cleaner
statements which are more than adequate for the application we have in
mind; namely to systems of points.

\subsection{The general framework and fundamental problem
  \label{gfub}} The {\it general framework of ubiquity} considered
within is as follows.

\begin{itemize}
\item $(\Omega,d)$ is a compact metric space.
\item $\mu$ is a Borel probability measure supported on $\Omega$.
\item There exist positive constants $\delta$ and $r_o$ such that for
  any $ x \in \Omega $ and $ r \leq r_0 $,
  \begin{equation}\label{M}
    a \, r ^{\delta}  \ \leq  \  \mu(B(x,r))  \ \leq  \   b \, r
    ^{\delta} .
  \end{equation}
  The constants $a$ and $b$ are independent of the ball
  $B(x,r):= \{y \in \Omega : d(x,y) < r \}$.
\item $\cR=(\ra)_{\alpha \in J }$ a sequence of points $\ra$ in
  $\Omega$ indexed by an infinite countable set $J$. The points
  $R_\alpha$ are commonly referred to as \emph{resonant points}.
\item $\beta: J \to \R^+ : \alpha \mapsto \ma $ is a positive function
  on $J$. It attaches a `weight' $\ma$ to the resonant point $\ra$.
\item To avoid pathological situations:
  \begin{equation}\label{path}
    \#\{\alpha\in J: \ma\le x\}<\infty\quad\text{for any $x\in\R$.}
  \end{equation}
\end{itemize}

\begin{rem} {\rm The measure condition \eqref{M} on the
      \emph{ambient measure} $\mu$ implies that $\mu$ is non-atomic,
      that is $\mu(\{x\})=0$ for any $x\in\Omega$, and that
$$
\mu(\Omega) :=  1 \asymp {\cal H}^{\delta} (\Omega)    \quad {\rm and \ }  \quad \dim \Omega = \delta  \, .
$$
Indeed, $\mu$ is comparable to $\delta$--dimensional Hausdorff
measure ${\cal H}^{\delta}$.  }
\end{rem}

\vspace*{2ex}


Given a decreasing function $\Psi : \R^+ \to \R^+ $ let
$$\La(\Psi)=\{x\in\Omega:x\in B(\ra,\Psi(\ma))\ \mbox{for\
  infinitely\ many\ }\alpha\in J\} \ . $$
The set $\La(\Psi)$ is a `$\limsup$' set; it consists of points in
$\Omega$ which lie in infinitely many of the balls $B(\ra,\Psi(\ma))$
centred at resonant points. As in the classical setting introduced in
\S\ref{clt}, it is natural to refer to the function $\Psi$ as the {\em
  approximating function}. It governs the `rate' at which points in
$\Omega$ must be approximated by resonant points in order to lie in
$\La(\Psi)$. In view of the finiteness condition \eqref{path}, it
follows that for any fixed $k>1$, the number of $ \alpha$ in $J$ with
$ k^{t-1} < \ma \leq k^t $ is finite regardless of the value of
$t\in \N $. Therefore $\La(\Psi)$ can be rewritten as the limsup set
of
$$
\Upsilon(\Psi,k,t) := \!\!\!\!\!  \bigcup_{\alpha\in J \ : \ { k^{t-1}
    < \ma \leq k^t}} \!\!\!\!\!\!\!\!\!\! B(\ra,\Psi(\ma))\,;
$$
that is
$$
\La(\Psi) = \limsup_{t \to \infty} \Upsilon(\Psi,k,t) :=
\bigcap_{m=1}^{\infty} \bigcup_{t=m}^{\infty} \Upsilon(\Psi,k,t) \
. $$

It is reasonably straightforward to determine conditions under which
$\mu(\La(\Psi)) = 0$. In fact, this is implied by the convergence part of
the Borel--Cantelli lemma from probability theory whenever
\begin{equation}\label{*1}
  \textstyle{\sum_{t=1}^\infty \, \mu(\Upsilon(\Psi,k,t)) < \infty  \ .}
\end{equation}

\noindent In view of this it is natural to consider the following
\textbf{fundamental problem}:
\begin{center}
  \emph{Under what conditions is $\mu(\La(\p)) >0 $ and more generally
    $\cH^s(\Lambda(\Psi))~>~0$ ?}
\end{center}

\noindent Ideally, we would like to be able to conclude the full
measure statement $ \cH^s(\Lambda(\Psi)) = \cH^s(\Omega) \, .  $
Recall that when $s=\delta$, the ambient measure $\mu$ coincides with
$\cH^{\delta}$.  Also, if $s < \delta$ then $ \cH^s(\Omega) = \infty$.

\subsubsection{The basic example \label{beub}} In order to illustrate
and clarify the above general setup, we show that the set $W(n, \psi)$
of simultaneously $\psi$-well approximable points
$\vx\in \I^n:=[0,1]^n$ can be expressed in the form of $\La(\Psi)$.
With this in mind, let
\begin{itemize}
\item[$\circ$] $\Omega:= \I^n $ \quad {\rm and } \quad
  $d(\vx,\vy):=\max\limits_{1\le i\le n}|x_i-y_i|$,
\item[$\circ$] $\mu$ be Lebesgue measure restricted to $\I^n$ \quad
  {\rm and } \quad $\delta:=n$,
\item[$\circ$] $J:= \{ (\vv p,q) \in \Z^n \times \N : \vv p/q\in \I^n\} $
  \quad {\rm and } \quad $\alpha := (\vv p,q) \in J$,
\item[$\circ$] $\cR:=(\vv p/q)_{(\vv p,q)\in J}$ \quad {\rm and }
  \quad $\beta_{(\vv p,q)}:=q$.
\end{itemize}

\noindent Thus, the resonant points $\ra$ are simply rational points
$\vv p/q := (p_1/q, \ldots,p_n/q)$ in the unit cube $\I^n$.  It is
readily verified that the measure condition \eqref{M} and the
finiteness condition \eqref{path} are satisfied and moreover that for
any decreasing function $\psi : \N \to \R^+ $,
\begin{equation*}\label{claim1}
  \La(\Psi)=W(n,\psi)\,\quad\text{with}  \quad \Psi(q):=\psi(q)/q\,.
\end{equation*}

\medskip

For this basic example, the solution to the fundamental problem is
given by the simultaneous Khintchine-\Jarnik{} Theorem (see
Theorem \ref{ohgod} with $m=1$ in \S\ref{DAKT}).

\subsection{The notion of ubiquity\label{thesystems}}

The following `system' contains the key measure theoretic structure
necessary for our attack on the fundamental problem.

Let $\r : \R^+ \to \R^+ $ be a function with $\r(r) \to 0 $ as
$r \to \infty $ and let
$$ \De(\r,k,t) := \bigcup_{\alpha\in J\,:\,\beta_\alpha\le k^t} B(\ra,\r(k^t))
\ ,$$
where $k > 1 $ is a fixed real number. Note that when $\rho=\Psi$ the
composition of $\De(\r,k,t)$ is very similar to that of
$ \Upsilon(\Psi,k,t)$.

\medskip


\noindent\textbf{Definition (Ubiquitous system\index{Ubiquitous system})\ } Let $B=B(x,r)$
denote an arbitrary ball with centre $x$ in $\Omega$ and radius
$r \leq r_0$. Suppose there exists a function $\r$ and absolute
constants $\ka>0$ and $ k > 1$ such that for any ball $B$ as above
\begin{equation}\label{ub}
  \mu\left( B \cap \De(\r,k,t) \right) \ \ge \ \ka \ \mu(B) \quad \mbox{for $t \geq t_0(B)$} .
\end{equation}
Then the pair $(\cR,\beta)$ is said to be a {\em local
  $\mu$-ubiquitous system relative to $(\r,k)$}.  If \eqref{ub} does
not hold for arbitrary balls with centre $x$ in $\Omega$ and radius
$r \leq r_0$, but does hold with $B= \Omega$, the pair $(\cR,\beta)$
is said to be a {\em global $\mu$-ubiquitous system relative to
  $(\r,k)$}.

\medskip

Loosely speaking, the definition of local ubiquity says that the set
$\De(\r,k,t)$ locally `approximates' the underlying space $\Omega$ in
terms of the measure $\mu$.  By `locally' we mean balls centred at
points in $\Omega$.  The function $\r$ is referred to as the {\em
  ubiquitous function}. The actual values of the constants $\ka$ and
$k$ in the above definition are irrelevant---it is their existence
that is important.  In practice, the $\mu$-ubiquity of a system can
be established using standard arguments concerning the distribution of
the resonant points in $\Omega$, from which the function $\r$ arises
naturally. To illustrate this, we return to the basic example of
\S\ref{beub}.

\begin{prop} \label{52} There is a constant $k>1$ such that the pair
  $(\cR,\beta)$ defined in \S\ref{beub} is a local $\mu$-ubiquitous
  system relative to $(\r,k)$ where
  $\r : r \mapsto {\rm const}\times r^{-(n+1)/n}$.
\end{prop}

The one-dimensional case of this proposition follows from
Theorem~\ref{thm:slv1}.

\noindent \emph{Exercise:} Prove the above proposition for arbitrary
$n$. Hint: you will need to use the multidimensional version of
Dirichlet's theorem, or Minkowski's theorem.

\subsection{The ubiquity statements}

Before stating the main results regarding ubiquity we introduce one
last notion. Given a real number $k >1$, a function $h:\R^+\to\R^+$
will be said to be \emph{$k$-regular} if there exists a strictly
positive constant $\lambda < 1$ such that for $t$ sufficiently large
\begin{equation} h(k^{t+1}) \leq \lambda \, h(k^t) \ .  \label{afmh}
\end{equation}
The constant $\lambda$ is independent of $t$ but may depend on $k$. A
consequence of local ubiquity is the following result.

\begin{thm}[Ubiquity - the Hausdorff measure case]\label{UL}
  Let $(\Omega,d)$ be a compact metric space equipped with a
  probability measure $\mu$ satisfying condition \eqref{M} and such
  that any open subset of $\Omega$ is $\mu$-measurable. Suppose that
  $(\cR,\beta)$ is a locally $\mu$-ubiquitous system relative to
  $(\r,k)$ and that $\Psi$ is an approximating function. Furthermore,
  suppose that $s \in (0,\delta]$, that $\rho$ is $k$-regular and that
  \begin{equation} \label{unif} \sum_{t=1}^{\infty} \
    \frac{\Psi(k^t)^s}{\r(k^t)^{\delta} } \ = \ \infty\,.
  \end{equation}
  Then
$${\cal H}^s  \left( \Lambda(\Psi) \right) \ = \ {\cal H}^s
\left( \Omega \right) \, .
$$
\end{thm}

As already mentioned, if $s < \delta$ then $ \cH^s(\Omega) =
\infty$.
On the other hand, if $ s = \delta$, the Hausdorff measure
$\cH^{\delta}$ is comparable to the ambient measure $\mu$ and the
theorem implies that
$$
\mu\left( \Lambda(\Psi) \right)= \mu (\Omega) := 1.
$$
Actually, the notion of global ubiquity has implications in the
ambient measure case.

\begin{thm}[Ubiquity  - the ambient measure case]\label{UL2}
  Let $(\Omega,d)$ be a compact metric space equipped with a measure
  $\mu$ satisfying condition \eqref{M} and such that any open subset
  of $\Omega$ is $\mu$-measurable. Suppose that $(\cR,\beta)$ is a
  globally $\mu$-ubiquitous system relative to $(\r,k)$ and that
  $\Psi$ is an approximating function. Furthermore, suppose that
  either $\rho$ or $\Psi$ is $k$-regular and that
  \begin{equation} \label{unif2} \sum_{t=1}^{\infty} \
    \left(\frac{\Psi(k^t)}{\r(k^t)}\right)^{\delta} \ = \ \infty\,.
  \end{equation}
  Then
$$
\mu\left( \Lambda(\Psi) \right)>0.
$$
If in addition $(\cR,\beta)$ is a locally $\mu$-ubiquitous system
relative to $(\r,k)$, then
$$
\mu\left( \Lambda(\Psi) \right)=1.
$$
\end{thm}

\begin{rem} {\rm Note that in Theorem \ref{UL2} we can get away with
    either $\rho$ or $\Psi$ being $k$-regular.  In the ambient measure
    case, it is also possible to weaken the measure condition
    \eqref{M} (see Theorem 1 in \cite[\S3]{BDV06}).  }
\end{rem}

\begin{rem} {\rm If we know via some other means that
    $ \Lambda(\Psi) $ satisfies a zero-full law (as indeed is the case
    for the classical set of $W(n,\psi)$ of $\psi$-well approximable
    points), then it is enough to show that
    $\mu\left( \Lambda(\Psi) \right)>0$ in order to conclude full
    measure.  }
\end{rem}

The above results constitute the main theorems appearing in
\cite{BDV06} tailored to the setup considered here.  In fact, Theorem
\ref{UL} as stated appears in \cite{BVparis} for the first time.
Previously, the Hausdorff and ambient measure cases had been thought
of and stated separately.

The concept of ubiquity was originally formulated by Dodson, Rynne
$\&$ Vickers \cite{DRV} to obtain lower bounds for the Hausdorff
dimension of $\limsup$ sets.  Furthermore, the ubiquitous systems of
\cite{DRV} essentially coincide with the regular systems of Baker $\&$
Schmidt \cite{BS70} and both have proved very useful in obtaining
lower bounds for the Hausdorff dimension of $\limsup$ sets. However,
unlike the framework developed in \cite{BDV06}, both \cite{BS70} and
\cite{DRV} fail to shed any light on establishing the more desirable
divergent Khintchine and \Jarnik{} type results. The latter clearly
implies lower bounds for the Hausdorff dimension. For further details
regarding regular systems and the original formulation of ubiquitous
systems see \cite{BDV06,BD99}.

\subsubsection{The basic example and the simultaneous
  Khintchine-\Jarnik{} Theorem \label{ubskj} }

Regarding the basic example of \S\ref{beub}, recall
that \begin{equation*} \La(\Psi)=W(n,\psi)\,\quad\text{with } \quad
  \Psi(q):=\psi(q)/q
\end{equation*}
and that Proposition \ref{52} states that for $k$ large enough, the
pair $(\cR,\beta)$ is a local $\mu$-ubiquitous system relative to
$(\r,k)$ where $$\r : r \mapsto {\rm const}\times r^{-(n+1)/n} \, . $$

\noindent Now, clearly the function $\rho $ is $k$-regular. Also note
that the divergence sum condition \eqref{unif} associated with
Theorem~\ref{UL} becomes
$$
\sum_{t=1}^{\infty} \ k^{t(n+1-s)}\psi(k^t)^s \, = \, \infty \, .
$$
If $\psi$ is monotonic, this is equivalent to
$$
\sum_{q=1}^{\infty} \ q^{n-s}\psi(q)^s \, = \, \infty \, ,
$$
and Theorem~\ref{UL} implies that

$$
\cH^s(W(n,\psi))=\cH^s(\I^n) \, .
$$

\noindent The upshot is that Theorem~\ref{UL} implies the divergent
case of the simultaneous Khintchine-Jarn\'{\i}k Theorem; namely,
Theorem \ref{ohgod} with $m=1$ in \S\ref{DAKT}.

\vspace*{2ex}

\begin{rem} {\rm It is worth standing back a little and thinking about
    what we have actually used in establishing the classical
    results---namely, local ubiquity.  Within the classical setup,
    local ubiquity is a simple measure theoretic statement concerning
    the distribution of rational points with respect to Lebesgue
    measure---the natural measure on the unit interval. From this we
    are able to obtain the divergent parts of both Khintchine's
    Theorem (a Lebesgue measure statement) and Jarn\'{\i}k's Theorem
    (a Hausdorff measure statement).  In other words, the Lebesgue
    measure statement of local ubiquity underpins the general
    Hausdorff measure theory of the $\limsup$ set $W(n, \p)$.  This of
    course is very much in line with the subsequent discovery of the
    Mass Transference Principle discussed in \S\ref{mtpballs}.  }
\end{rem}

The applications of ubiquity are widespread, as demonstrated in
\cite[\S12]{BDV06}. We now consider a more recent application of
ubiquity to the `fibers' strengthening of Khintchine's simultaneous
theorem described in \S\ref{KonFibers}.

\subsubsection{Proof of Theorem \ref{tkfib2}: Part A }

Let $\psi:\N\to \Rp$ be a monotonic function and $\alpha \in \I$ such
that it has Diophantine exponent $\tau(\alpha) < 2$. In view of Remark
\ref{re27} in \S\ref{KonFibers}, establishing Theorem~\ref{tkfib2} is
equivalent to showing that
\begin{equation*}
  m( \Pi (\psi, \alpha) ) =
  1    \quad  {\rm if}    \quad  \sum_{q=1}^{\infty} \;   \psi^2(q)
  =\infty   \,
\end{equation*}
where
$$\Pi (\psi, \alpha) := \{ \beta \in \I :  \| q \beta \| \le   \psi(q) \ {\rm for \ infinitely  \  many \ } q \in {\cal A}_{\alpha}(\psi) \} \, .
$$
\noindent Recall,
$$
{\cal A}_{\alpha}(\psi) := \{ q \in \N: \| q \alpha \| \le \psi(q) \}
\, .
$$

\begin{rem} \label{remfib} {\rm Without loss of generality, we can
    assume that
\begin{equation} \label{ty}
 q^{-\frac12}  (\log q)^{-1}   \ \leq \ \psi(q)  \ \leq \  q^{-\frac12}   \quad \forall \ q \in \N \, . \end{equation}
 \noindent\emph{Exercise:}  Verify that this is indeed the case. For the right-hand side  of \eqref{ty}, consider the auxiliary function
 $$
 \tilde{\psi} : q \to \tilde{\psi}:= \min \{q^{-\frac12} , \psi(q) \}
 $$
 and show that   $\sum_{q=1}^{\infty} \;   \tilde{\psi}^2(q)
 =\infty  $. For the left-hand side of \eqref{ty}, consider the auxiliary function
 $$
  \tilde{\psi}  : q \to  \tilde{\psi} (q) := \max \{\hat{\psi}(q) := q^{-\frac12}  (\log q)^{-1} , \psi(q) \}
 $$
 and show that $m(\Pi (\hat{\psi}, \alpha)) = 0$ by making use of the
 counting estimate \eqref{count} and the convergence Borel-Cantelli
 Lemma.  }
\end{rem}

We now show that the set $\Pi (\psi, \alpha)$ can be expressed in the
form of $\La(\Psi)$.  With this in mind, let
\begin{itemize}
\item[$\circ$] $\Omega:=[0,1]$ \quad {\rm and } \quad
  $d(x,y):= |x-y|$,
\item[$\circ$] $\mu$ be Lebesgue measure restricted to $\I$ \quad {\rm
    and } \quad $\delta:=1$,
\item[$\circ$]
  $J:= \{ (p,q) \in \Z \times {\cal A}_{\alpha}(\psi) : p/q\in \I\} $
  \quad {\rm and } \quad $\alpha := (p,q) \in J$,
\item[$\circ$] $\cR:=(p/q)_{(p,q)\in J}$ \quad {\rm and } \quad
  $\beta_{(p,q)}:=q$.
\end{itemize}

\noindent Thus, the resonant points $\ra$ are simply rational points
$p/q$ in the unit interval $\I$ with denominators $q $ restricted to
the set ${\cal A}_{\alpha}(\psi) $.  It is readily verified that the
measure condition \eqref{M} and the finiteness condition \eqref{path}
are satisfied and moreover that for any decreasing function
$\psi : \N \to \R^+ $,
\begin{equation*}\label{claim33}
  \La(\Psi)=\Pi (\psi, \alpha)\,\quad\text{with}  \quad \Psi(q):=\psi(q)/q\,.
\end{equation*}
Note that since $\psi$ is decreasing, the function $\Psi$ is
$k$-regular. Now, in view of Remark \ref{remfib}, the conditions of
Proposition~\ref{thm:slv45} are satisfied and we conclude that for $k$
large enough, the pair $(\cR,\beta)$ is a global $m$-ubiquitous system
relative to $(\r,k)$ where
$$\r : r \mapsto \frac{k}{r^{2}\psi(r)} \, . $$

\noindent Now, since $\psi$ is monotonic
$$
\sum_{t=1}^\infty\frac{\Psi(k^t)}{\r(k^t)} \ = \ \sum_{t=1}^\infty
k^{t-1}\psi^2(k^t)=\infty \quad \Longleftrightarrow \quad
\sum_{q=1}^\infty\psi^2(q)=\infty
$$
and Theorem~\ref{UL2} implies that
$$
\mu\Big(\Pi (\psi, \alpha)\Big)>0\,.
$$
Now observe that $ \Pi (\psi, \alpha) $ is simply the set
$W(\bar{\psi})$ of $ \bar{\psi} $--well approximable numbers with
$ \bar{\psi}(q) := \psi(q) $ if $ q \in {\cal A}_{\alpha}(\psi) $ and
zero otherwise.  Thus, Cassels' zero-full law
\cite{Cassels-50:MR0036787} implies the desired statement; namely that
$$
\mu\Big(\Pi (\psi, \alpha)\Big)=1 \, .
$$


\section{Diophantine approximation on manifolds\label{manifolds}}

\emph{Diophantine approximation on manifolds} (as coined by Bernik \&
Dodson in their Cambridge Tract \cite{BD99} ) or \emph{Diophantine
  approximation of dependent quantities} (as coined by Sprind\v zuk in
his monograph \cite{Sprindzuk}) refers to the study of Diophantine
properties of points in $\R^n$ whose coordinates are confined by
functional relations or equivalently are restricted to a sub-manifold
$\cM$ of $\R^n$. Thus, in the case of simultaneous Diophantine
approximation one studies sets such as
$$
\cM \cap W(n,\psi) \, .
$$

\noindent To some extent we have already touched upon the theory of
Diophantine approximation on manifolds when we considered Gallagher
multiplicative theorem on fibers in \S\ref{Gonfibres} and Khintchine
simultaneous theorem on fibers in \S\ref{KonFibers}. In these sections
the points of interest are confined to an affine co-ordinate subspace
of $\R^n$; namely the manifold
$$
\text{${\rm L}_{\bm{\alpha}} := \{\bm{\alpha}\} \times \R^d$, where
  $ 1 \le d \le n-1$ and $ \bm{\alpha} \in \I^{n-d}$.}
$$

In general, a manifold $\cM$ can locally be given by a system of
equations, for instance, the unit sphere in $\R^3$ is given by the
equation
$$
x^2+y^2+z^2=1;
$$
or it can be immersed into $\R^n$ by a map $\vv f:\R^d\to\R^n$ (the
actual domain of $\vv f$ can be smaller than $\R^d$), for example, the
Veronese curve is given by the map
$$
x\mapsto(x,x^2,\dots,x^n)\,.
$$
Such a map $\vv f$ is often referred to as a parameterisation and
\emph{without loss of generality we will assume that the domain of
  $\vv f$ is $\I^d$ and that the manifold $\cM \subseteq \I^n$ }.
Locally, a manifold given by a system of equations can be
parameterised by some map $\vv f$ and, conversely, if a manifold is
immersed by a map $\vv f$, it can be written using a system of $n-d$
equations, where $d$ is the dimension of the manifold.

\noindent\emph{Exercise:} Parameterise the upper hemisphere
$x^2+y^2+z^2=1$, $z>0$, and also write the Veronese curve (see above)
by a system of equations.

\vspace*{3ex}

In these notes we will mainly concentrate on the simultaneous (rather
than dual) theory of Diophantine approximation on manifolds. In
particular, we consider the following two natural problems.

\medskip

~\hspace*{4ex} {\em Problem 1.} To develop a Lebesgue theory for
$ \cM \cap W(n,\psi) $.

\medskip

~\hspace*{4ex} {\em Problem 2.} To develop a Hausdorff theory for
$ \cM \cap W(n,\psi) $.

\medskip

\noindent In short, the aim is to establish analogues of the two
fundamental theorems of Khintchine and Jarn\'\i k, and thereby provide
a complete measure theoretic description of the sets
$ \cM \cap W(n,\psi)$. The fact that the points $\vv x \in \R^n$ of
interest are of dependent variables, which reflects the fact that
$\vv x \in {\cal M}$, introduces major difficulties in attempting to
describe the measure theoretic structure of $ \cM \cap W(n,\psi)
$.
This is true even in the specific case that ${\cal M}$ is a planar
curve. More to the point, even for seemingly simple curves such as the
unit circle or the parabola the above problems are fraught with
difficulties.  In these notes we will concentrate mainly on describing
the Lebesgue theory.

{\em Unless stated otherwise, the approximating function
  $\psi:\N\to \Rp$ throughout this section is assumed to be
  monotonic. }

\subsection{The Lebesgue theory for manifolds}

The goal is to obtain a Khintchine type theorem that describes the
Lebesgue measure of the set $\cM \cap W(n,\psi)$ of simultaneously
$\psi$--approximable points lying on ${\cal M}$. First of all notice
that if the dimension $d$ of the manifold ${\cal M}$ is strictly less
than $n$ then $ m_n(\cM \cap W(n,\psi)) = 0 $ irrespective of the
approximating function $\psi$.  Thus, in attempting to develop a
Lebesgue theory for $ \cM \cap W(n,\psi)$ it is natural to use the
induced $d$-dimensional Lebesgue measure on $\cM$. Alternatively, if
$\cM$ is immersed by a map $\vv f:\I^d\to\R^n$ we use the
$d$-dimensional Lebesgue measure $m_d$ on the set of parameters of
$\vv f$; namely $\I^d$.  In either case, the measure under
consideration will be denoted by $ | \ . \ |_{\cM} $.

\vspace*{2ex}

\begin{rem} \label{rem300} Notice that for $\tau \leq 1/n $, we have
  that
  $|\cM \cap W(n,\tau)|_{{\cal M}} = |{\cal M}|_{{\cal M}} :=
  \mbox{F{\scriptsize ULL}} $
  as it should be since, by Dirichlet's theorem, we have that
  $W(n,\tau)= \I^n $.
\end{rem}

The two-dimension fiber problem considered in \S\ref{KonFibers}, in
which the manifold $ \cM $ is a vertical line ${\rm L}_{\alpha} $,
shows that it is not possible to obtain a Khintchine type theorem
(both the convergence and divergence aspects) for all manifolds.
Indeed, the convergent statement fails for vertical lines.  Thus, in a
quest for developing a general Khintchine type theory for manifolds
(cf. Problem 1 above) , it is natural to avoid lines and more
generally hyperplanes.  In short, we insist that the manifold under
consideration is ``sufficiently'' curved.

\subsubsection{Non-degenerate manifolds}\label{rpnm1}

In order to make any reasonable progress with Problems 1 \& 2 above,
we assume that the manifolds $ {\cal M}$ under consideration are {\bf
  non-degenerate} \cite{KM98}. Essentially, these are smooth
sub-manifolds of $\R^n$ which are sufficiently curved so as to deviate
from any hyperplane. Formally, a manifold $\cM$ of dimension $d$
embedded in $\R^n$ is said to be non-degenerate if it arises from a
non--degenerate map $\vv f:U\to \R^n$ where $U$ is an open subset of
$\R^d$ and $\cM:=\vv f(U)$. The map
$\vv f:U\to \R^n:\vv x\mapsto \vv f(\vv x )=(f_1(\vv x),\dots,f_n(\vv
x))$
is said to be \emph{non--degenerate at} $\vv x \in U$ if there exists
some $l\in\N$ such that $\vv f$ is $l$ times continuously
differentiable on some sufficiently small ball centred at $\vv x$ and
the partial derivatives of $\vv f$ at $\vv x$ of orders up to $l$ span
$\R^n$. The map $\vv f$ is \emph{non--degenerate} if it is
non--degenerate at almost every (in terms of $d$--dimensional Lebesgue
measure) point in $U$; in turn the manifold $\cM=\vv f(U)$ is also
said to be non--degenerate.  Any real, connected analytic manifold not
contained in any hyperplane of $\R^n$ is non--degenerate.  Indeed, if
$\cM$ is immersed by an analytic map
$\vv f = (f_1,\dots,f_n):U\to \R^n$ defined on a ball
$U \subset \R^d$, then $\cM$ is non-degenerate if and only if the
functions $ 1, f_1,\dots,f_n$ are linearly independent over $\R$.

\emph{Without loss of generality, we will assume that $U$ is $\I^d$
  and that the manifold $\cM \subseteq \I^n$ }

\noindent Note that in the case the manifold $\cM$ is a planar curve
${\cal C}$, a point on ${\cal C}$ is non-degenerate if the curvature
at that point is non-zero. Thus, ${\cal C}$ is a non-degenerate planar
curve if the set of points on ${\cal C}$ at which the curvature
vanishes is a set of one--dimensional Lebesgue measure zero. Moreover,
it is not difficult to show that the set of points on a planar curve
at which the curvature vanishes but the curve is non-degenerate is at
most countable. In view of this, the curvature completely describes
the non-degeneracy of planar curves. Clearly, a straight line is
degenerate everywhere.

The claim is that the notion of non-degeneracy is the right
description for a manifold $\cM$ to be ``sufficiently'' curved in
order to develop a general Khintchine type theory (both convergent and
divergent cases) for $\cM \cap W(n,\psi)$.  With this in mind, the key
then lies in understanding the distribution of rational points
``close'' to such manifolds.

\subsubsection{Rational points near manifolds: the heuristics }\label{rpnm}
Given a point $\vv x=(x_1,\dots,x_n) \in \R^n $ and a set
$A \subseteq \R^n$, let
$$
{\rm dist}(\vx , A ) := \inf\{d(\vx , \vv a) \, : \, \vv a \in A \}
$$
where as usual $d(\vx,\vv a):=\max\limits_{1\le i\le n}|x_i-a_i|$.
Now let $\vv x \in \cM \cap W(n,\psi) $.  Then by definition there
exist infinitely many $q\in\N$ and $\vv p\in\Z^n$ such that
$$
{\rm dist}\Big(\cM,\frac{\vv p}{q}\Big)\le d \Big(\vv x,\frac{\vv
  p}{q}\Big)<\frac{\psi(q)}{q}\, .
$$
This means that the rational points
$$ \frac{\vv p}{q} := \Big(\frac{p_1}{q},\dots,\frac{p_n}{q} \Big)$$
of interest must lie within the $\frac{\psi(q)}{q}$--neighbourhood of
$\cM$. In particular, assuming that $\psi$ is decreasing, we have that
the points $ \vv p / q $ of interest with $ k^{t-1} < q \le k^{t} $
are contained in the $\frac{\psi(k^{t-1})}{k^{t-1}}$--neighbourhood of
$\cM$. Let us denote this neighbourhood by $\Delta^+_k(t,\psi)$ and by
$N^+_k(t,\psi)$ the set of rational points with
$ k^{t-1} < q \le k^{t} $ contained in $\Delta^+_k(t,\psi)$. In other
words,
\begin{equation}\label{vb1}
  N^+_k(t,\psi) := \left\{ \vv p / q \in \I^n: k^{t-1} <  q \le k^{t}  \ {\rm \  and \ } \  {\rm dist}\big(\cM,\vv p/q \big)\le   \textstyle{\frac{\psi(k^{t-1})}{k^{t-1}} } \right\}   \, .
\end{equation}
Recall, that $\cM \subseteq \I^n$.  Hence, regarding the
$n$-dimensional volume of the neighbourhood $\Delta^+_k(t,\psi)$, it
follows that
$$
m_n\Big( \Delta^+_k(t,\psi) \Big) \ \asymp \
\left(\frac{\psi(k^{t-1})}{k^{t-1}}\right)^{n-d}\,.
$$
Now let $Q _k(t)$ denote the set of rational points with
$ k^{t-1} < q \le k^{t} $ lying in the unit cube $\I^n$. Then,
$$
\# Q _k(t) \ \asymp \ (k^t)^{n+1}
$$
and if we assume that the points in $ Q _k(t) $ are ``fairly''
distributed within $\I^n$, we would expect that
$$
\begin{array}{c}
  \text{the number of  these  points   that fall into $\Delta^+_k(t,\psi) $}\\
  \text{is proportional to the measure of $\Delta^+_k(t,\psi) $}\, .
\end{array}
$$

\noindent In other words and more formally, under the above
distribution assumption, we would expect that
\begin{equation}\label{vb1+}
  \# \{  Q _k(t)  \cap \Delta^+_k(t,\psi)  \}  \     \asymp   \  \#  Q _k(t)  \times    m_n\Big( \Delta^+_k(t,\psi)  \Big)  \,
\end{equation}
and since the left-hand side is $\# N^+_k(t,\psi)$, we would be able
to conclude that
\begin{equation}\label{x1}
  \# N^+_k(t,\psi)  \  \asymp \  (k^t)^{n+1}\left(\frac{\psi(k^{t-1})}{k^{t-1}}\right)^{n-d}  \ \asymp  \
  (k^{t-1})^{d+1}\psi(k^{t-1})^{n-d}\,.
\end{equation}

\noindent For the moment, let us assume that \eqref{vb1+} and hence
\eqref{x1} are fact.  Now
\begin{eqnarray*}
  \cM \cap W(n,\psi)  & =  & \bigcap_{m=1}^\infty\bigcup_{t=m}^{\infty}  \ \ \bigcup_{k^{t-1}<q\le k^t} \ \ \bigcup_{\vv p\in\Z^n :\vv p/q \in \I^n }B\Big( \textstyle{\frac{\vv p}{q},\frac{\psi(q)}{q}}\Big)  \cap \cM
  \\[2ex]
                      & \subset  & \bigcap_{m=1}^\infty\bigcup_{t=m}^{\infty}A^+_k(t,\psi, \cM) \,
\end{eqnarray*}
where
$$
A^+_k(t,\psi, \cM):= \bigcup_{k^{t-1}<q\le k^t} \ \ \bigcup_{\vv
  p\in\Z^n :\vv p/q \in \I^n }B\Big(\textstyle{\frac{\vv
    p}{q},\frac{\psi(k^{t-1})}{k^{t-1}}}\Big) \cap \cM \, .
$$
It is easily verified that

\begin{eqnarray*}
  | A^+_k(t,\psi, \cM)  |_{\cM} & \le  &
                                         \sum_{k^{t-1}<q\le k^t} \ \ \sum_{\vv p\in\Z^n: \vv p/q \in \I^n} \ \ \underbrace{\Big| B\Big(\textstyle{  \frac{\vv p}{q},\frac{\psi(k^{t-1})}{k^{t-1}}}\Big)   \cap \cM  \Big|_{\cM} }_{\ll (\psi(k^{t-1})/k^{t-1})^d} \\[3ex]
                                & \ll & \# N^+_k(t,\psi)\ \  (\psi(k^{t-1})/k^{t-1})^d \\[3ex]
                                &
                                  \stackrel{\eqref{x1}}{\asymp} &(k^{t-1})^{d+1}\psi(k^{t-1})^{n-d} (\psi(k^{t-1})/k^{t-1})^d \\[3ex]
                                &
                                  \asymp  &   k^{t-1}\psi(k^{t-1})^{n}\,.
\end{eqnarray*}
Hence
\begin{equation}\label{vb1++}
  \sum_{t=1}^\infty | A^+_k(t,\psi, \cM)  |_{\cM} \ \ll \
  \sum_{t=1}^\infty k^t\psi(k^t)^{n} \  \asymp \
  \sum_{q=1}^\infty \psi(q)^{n}\,.
\end{equation}

All the steps in the above argument apart from \eqref{vb1+} and hence
\eqref{x1}, can be turned into a rigorous proof.  Indeed, the estimate
\eqref{x1} is not always true.

\noindent\emph{Exercise.} Consider the circle $\cC_{\!\sqrt3}$ in
$\R^2$ given by the equation $x^2+y^2=3$. Prove that $\cC$ does not
contain any rational points. Next let $\psi(q)=q^{-1-\ve}$ for some
$\ve>0$. Prove that
$$
\cC_{\!\sqrt3} \cap W(2,\psi)=\varnothing\,.
$$
The upshot is that even for non-degenerate manifolds, we cannot expect
the heuristic estimate \eqref{x1} to hold for any decreasing $\psi$ --
some restriction on the rate at which $\psi$ decreases to zero is
required.  On the other hand, affine subspaces of $\R^n$ may contain
too many rational points, for instance, if $\cM$ is a linear subspace
of $\R^n$ with a basis of rational vectors. Of course, such manifolds
are not non-degenerate.

However, {\em whenever the upper bound associated with the heuristic
  estimate \eqref{x1} is true}, inequality \eqref{vb1++} together with
the convergence Borel-Cantelli Lemma implies that
$$
| \cM \cap W(n,\psi) |_{\cM} = 0 \quad {\rm if } \quad
\sum_{q=1}^{\infty} \; \psi(q)^n <\infty \, .
$$
This statement represents the convergent case of the `dream' theorem
for manifolds -- see \S\ref{rpnm2} immediately below. Note that the
associated sum $ \sum \psi(q)^n $ coincides with the sum appearing in
Theorem \ref{nkh} (Khintchine in $\R^n$) but the associated measure
$ | \ . \ |_{\cM} $ is $d$-dimensional Lebesgue measure (induced on
$\cM$) rather than $n$-dimensional Lebesgue measure.

\subsubsection{The Dream Theorem and its current status }\label{rpnm2}

\noindent\textbf{The Dream Theorem. \ } Let $\cM$ be a non-degenerate
sub-manifold of $\R^n$.  Let $\psi:\N\to \Rp$ be a monotonic
function. Then
\begin{equation}\label{vb5}
  | \cM\cap W(n,\psi)  |_{\cM}
  =\left\{
    \begin{array}{ll}
      0 & {\rm if} \;\;\; \sum_{q=1}^{\infty} \;   \psi(q)^n  <\infty\;
          ,\\[4ex]
      1 & {\rm if} \;\;\; \sum_{q=1}^{\infty} \;   \psi(q)^n
          =\infty \; .
    \end{array}\right.    
\end{equation}

\noindent 
\emph{We emphasize that the Dream Theorem is a desired statement
  rather than an established fact.}

\vspace*{2ex}

As we have already demonstrated, the convergence case of the Dream
Theorem would follow on establishing the upper bound estimate
\begin{equation}\label{x1ub}
  \# N^+_k(t,\psi)  \  \ll  \
  (k^{t-1})^{d+1}\psi(k^{t-1})^{n-d}\,
\end{equation}
for non-degenerate manifolds.  Recall that the rational points of
interest are given by the set
$$
N_k(t,\psi) := \left\{ \vv p / q \in \I^n: k^{t-1} < q \le k^{t} \
  {\rm and \ } \ {\rm dist}\big(\cM,\vv p/q \big)\le
  \textstyle{\frac{\psi(q)}{q} } \right\} \, ,
$$
and that $ \# N^+_k(t,\psi) $ is an upper bound for
$ \# N_k(t,\psi) $.  Obviously, a lower bound for $\#N_k(t,\psi) $ is
given by $ \# N^-_k(t,\psi)$ where

$$
N^-_k(t,\psi) := \left\{ \vv p / q \in \I^n: k^{t-1} < q \le k^{t} \
  {\rm and \ } \ {\rm dist}\big(\cM,\vv p/q \big)\le
  \textstyle{\frac{\psi(k^{t})}{k^{t}} } \right\} \, ,
$$
and if $\psi$ is $k$-regular (see \eqref{afmh}) then
$ N^+_k(t,\psi) \asymp N^-_k(t,\psi)$. In particular, whenever we are
able to establish the heuristic estimate \eqref{x1} or equivalently
the upper bound estimate \eqref{x1ub} together with the lower bound
estimate
\begin{equation}\label{x1lb}
  \# N^-_k(t,\psi)  \  \gg  \
  (k^{t-1})^{d+1}\psi(k^{t-1})^{n-d}\, ,
\end{equation}
we would have that
\begin{equation}\label{xx1}
  \# N_k(t,\psi)  \  \asymp   \
  (k^{t-1})^{d+1}\psi(k^{t-1})^{n-d}\,.
\end{equation}
It is worth stressing that the lower bound estimate \eqref{x1lb} is by
itself not enough to prove the divergence case of the Dream
Theorem. Loosely speaking, we also need to know that rational points
associated with the set $N^-_k(t,\psi)$ are ``ubiquitous'' within the
$\frac{\psi(k^{t})}{k^{t}}$--neighbourhood of $\cM$.  Indeed, when
establishing the divergence case of Khintchine's Theorem (Theorem
\ref{kg}), we trivially have the right count of $k^{2t}$ for the
number of rational points $p/q \in \I$ with $ k^{t-1} < q \le k^{t} $.
The crux is to establish the associated distribution type result given
by Theorem~\ref{thm:slv1}.  This in turn implies that the rational
points under consideration give rise to a ubiquitous system -- see
\S\ref{ubskj}.

\medskip

We now turn our attention to reality and describe various `general'
contributions towards the Dream Theorem.

\begin{itemize}

\item { \em Extremal manifolds.}  A sub-manifold $\cM$ of $\R^n$ is
  called {\em extremal} if
$$
\left| \cM\cap W(n, \textstyle{\frac{1+\ve}{n}}) \right|_{\cM} =0
\qquad \forall \ \ve>0 \,.
$$
Note that $ \cM\cap W(n, \textstyle{\frac{1}{n}} ) = \cM$ -- see
Remark \ref{rem300}.  In their pioneering work \cite{KM98} published
in 1998, Kleinbock $\& $ Margulis proved that any non-degenerate
sub-manifold $\cM$ of $\R^n$ is extremal.  It is easy to see that this
implies the convergence case of the Dream Theorem for functions of the
shape
$$
\psi_\ve(q) := q^{-\frac{1+\ve}{n}} \, .
$$
Indeed,
$$
\textstyle{\sum_{q=1}^\infty\psi_\ve(q)^n=\sum_{q=1}^\infty
  q^{-(1+\ve)}<\infty\, }
$$
and so whenever the convergent case of \eqref{vb5} is fulfilled, the
corresponding manifold is extremal.

\item {\em Planar curves.}  The Dream Theorem is true when $n=2$; that
  is, when $\cM$ is a non-degenerate planar curve.  The convergence
  case of \eqref{vb5} for planar curves was established in
  \cite{Vaughan-Velani-2007} and subsequently strengthened in
  \cite{BZ}.  The divergence case of \eqref{vb5} for planar curves was
  established in \cite{Beresnevich-Dickinson-Velani-07:MR2373145}.

\item {\em Beyond planar curves.} The divergence case of the Dream
  Theorem is true for analytic non-degenerate sub-manifolds of $\R^n$
  \cite{B12}. Recently, the divergence case of \eqref{vb5} has been
  shown to be true for non-degenerate curves and manifolds that can be
  `fibred' into such curves \cite{BVVZdiv}.  The latter includes
  $C^{\infty}$ non-degenerate sub-manifolds of $\R^n$ which are not
  necessarily analytic.  The convergence case of the Dream Theorem is
  true for a large subclass of $2$-non-degenerate sub-manifolds of
  $\R^n$ with dimension $d$ strictly greater than $(n+1)/2$
  \cite{BVVZcon}. Earlier, manifolds satisfying a geometric
  (curvature) condition were shown to satisfy the convergence case of
  the Dream Theorem \cite{DRV2}.

\end{itemize}

\noindent The upshot of the above is that the Dream Theorem is in
essence fact for a fairly generic class of non-degenerate
sub-manifolds $\cM$ of $\R^n$ apart from the case of convergence when
$n \ge 3 $ and $d \le (n+1)/2$.

\vspace*{2ex}

\begin{rem} \label{remMP} The theory of Diophantine approximation
  stems from Mahler's problem (1932) regarding the {\em extremality}
  of the Veronese curve
  $ \cV := \{ (x,x^2,\dots,x^n) : x \in \R^n \} $.  Following a
  substantial number of partial results (initially for $n=2$, then
  $n=3$ and some for higher $n$), a complete solution to the problem
  was given by Sprind\v zuk in 1965.  For a historical account of the
  manifold theory we refer the reader to the monographs \cite{BD99,
    Sprindzuk} and the introduction given in the paper
  \cite{Beresnevich-Dickinson-Velani-07:MR2373145}.
\end{rem}

\vspace*{2ex}

\begin{rem} \label{extem=} Note that in view of the Khintchine's
  Transference Principle, we could have easily defined
  extremality via the dual form of Diophantine approximation (see
  Remark \ref{reKTPP}); namely, $\cM$ is extremal if
$$
\left| \cM\cap W^*(n, n + \ve) \right|_{\cM} =0 \qquad \forall \ \ve>0
\,.
$$
The point is that both definitions are equivalent. This is not the
case in the inhomogeneous setup considered in \S\ref{ITPtheory}.
\end{rem}

\vspace*{2ex}

\begin{rem} \label{remKM} It is worth mentioning that in \cite{KM98},
  Kleinbock $\& $ Margulis established a stronger (multiplicative)
  form of extremality (see \S\ref{multmanifolds} below) that settled
  the Baker-Sprind\v zuk Conjecture from the eighties.  Not only did
  their work solve a long-standing fundamental problem, but it also
  developed new techniques utilising the link between Diophantine
  approximation and homogeneous dynamics. Without doubt the work of
  Kleinbock $\& $ Margulis has been the catalyst for the subsequent
  contributions towards the Dream Theorem described above.
\end{rem}

\bigskip

\subsection{The Hausdorff theory for manifolds \label{HMtheory}}

The goal is to obtain a Jarn\'\i k type theorem that describes the
Hausdorff measure $\cH^s$ of the set $\cM \cap W(n,\psi)$ of
simultaneously $\psi$--approximable points lying on ${\cal M}$.  In
other words, we wish to obtain a Hausdorff measure version of the
Dream Theorem.  In view of this, by default, we consider approximating
functions $\psi$ which decrease sufficiently rapidly so that the
$d$-dimensional Lebesgue measure of $ \cM \cap W(n,\psi) $ is
zero. Now, as the example in \S\ref{rpnm} demonstrates, in order to
obtain a coherent Hausdorff measure theory we must impose some
restriction on the rate at which $\psi$ decreases.  Indeed, with
reference to that example, the point is that
$\cH^s(\cC_{\!\sqrt3}\cap W(2,1+\ve)) = 0 $ irrespective of $ \ve > 0$
and the measure $\cH^s$.  On the other hand, for the unit circle
$\cC_1$ in $\R^2$ given by the equation $x^2+y^2=1$, it can be shown
\cite[Theorem 19]{BDV06} that for any $\ve > 0$
$$
\cH^s(\cC_{\!1}\cap W(2,1+\ve)) = \infty \quad {\rm with } \quad
\textstyle{s=\frac{1}{2 + \ve} } \, .
$$

Nevertheless, it is believed that if the rate of decrease of $\psi$ is
`close' to the approximating function $q^{-1/n}$ associated with
Dirichlet's Theorem, then the behaviour of
$ \cH^s(\cM\cap W(n,\psi)) $ can be captured by a single, general
criterion.  In the following statement, the condition on $\psi$ is
captured in terms of the deviation of $\cH^s$ from $d$-dimensional
Lebesgue measure.

\vspace*{2ex}

\noindent\textbf{The Hausdorff Dream Theorem. }  Let $\cM$ be a
non-degenerate sub-manifold of $\R^n$, $d:= \dim \cM$ and
$m:= {\rm codim} \, \cM$.  Thus, $d+m=n$. Let $\psi:\N\to \Rp$ be a
monotonic function. Then, for any $s \in (\frac{m}{m+1} d, d \big)$

\begin{equation}\label{vb58}
  \cH^s(\cM\cap W(n,\psi)) =\left\{
    \begin{array}{ll}
      0 & {\rm if} \;\;\; \displaystyle\sum_{q=1}^{\infty} \;   \psi^{s+m}(q)q^{-s+d}  <\infty\
          ,\\[4ex]
      \infty & {\rm if} \;\;\; \displaystyle\sum_{q=1}^{\infty} \;   \psi^{s+m}(q)q^{-s+d}
               =\infty  .
    \end{array}\right.
\end{equation}

\noindent 
\emph{We emphasize that the above is a desired statement rather than
  an established fact.}

\vspace*{2ex}

We now turn our attention to reality and describe various `general'
contributions towards the Hausdorff Dream Theorem.

\begin{itemize}
\item {\em Planar curves.}  As with the Dream Theorem, the convergence
  case of \eqref{vb58} for planar curves ($n=2, d=m=1$) was
  established in \cite{Vaughan-Velani-2007} and subsequently
  strengthened in \cite{BZ}.  The divergence case of \eqref{vb58} for
  planar curves was established in
  \cite{Beresnevich-Dickinson-Velani-07:MR2373145}.

\item {\em Beyond planar curves.} The divergence case of the Hausdorff
  Dream Theorem is true for analytic non-degenerate sub-manifolds of
  $\R^n$ \cite{B12}.  The convergence case is rather fragmented. To
  the best of our knowledge, the partial results obtained in
  \cite[Corollaries 3 $\& $ 5]{BVVZcon} for $2$-non-degenerate
  sub-manifolds of $\R^n$ with dimension $d$ strictly greater than
  $(n+1)/2$, represent the first significant coherent contribution
  towards the convergence case.
\end{itemize}

\vspace*{0ex}

\noindent{\em Exercise.} Prove the convergent case of \eqref{vb58}
assuming the heuristic estimate \eqref{x1} for the number of rational
points near $\cM$ -- see \S\ref{rpnm}.

\vspace*{2ex}

\begin{rem}
  Regarding the divergence case of \eqref{vb58}, it is tempting to
  claim that it follows from the divergence case of the (Lebesgue)
  Dream Theorem via the Mass Transference Principle introduced in
  \S\ref{mtpballs}.  After all, this is true when $\cM = \I^n$; namely
  that Khintchine's Theorem implies Jarn\'{\i}k's Theorem as
  demonstrated in \S\ref{KimpliesJ}. However, this is far from the
  truth within the context of manifolds. The reason for this is
  simple.  With respect to the setup of the Mass Transference
  Principle, the set $\Omega$ that supports the $\cH^\delta$-measure
  (with $\delta = \dim \cM$) is the manifold $\cM$ itself and is
  embedded in $\R^n$.  The set $\cM \cap W(n,\psi) \subset \Omega$ of
  interest can be naturally expressed as the intersection with $\cM$
  of the $\limsup$ set arising from balls
  $B(\frac{\vv p}{q},\frac{\psi(q)}{q}) $ centred at rational points
  $\vv p/q \in \R^n$.  However, the centre of these balls do not
  necessarily lie in the support of the measure $\Omega= \cM$ and this
  is where the problem lies. A prerequisite for the framework of the
  Mass Transference Principle is that $\{B_i\}_{i\in\N}$ is a sequence
  of balls in $\Omega$.
\end{rem}


\subsection{Inhomogeneous Diophantine approximation \label{IHtheory}}

When considering the well approximable sets $W(n, \psi)$ or indeed the
badly approximable sets $\bad(i_1, \ldots,i_n)$, we are in essence
investigating the behaviour of the fractional part of $q\vx$ about the
origin as $q$ runs through $\N$.  Clearly, we could consider the setup
in which we investigate the behaviour of the orbit of $\{q\vx\}$ about
some other point. With this in mind, given $\psi:\N\to \Rp$ and a
fixed point $\bm\g=(\g_1, \dots \g_n) \in \R^n$, let
$$
W_{\bm\g}(n, \psi): = \{\vx\in \I^n \colon \|q \vx - \bm\g \|<\psi(q)
\text{ for infinitely many } q\in\N \} \
$$
denote the \emph{inhomogeneous} set of \emph{simultaneously
  $\psi$-well approximable} points $\vx\in \I^n$.  Thus, a point
$\vx\in W_{\bm\g}(n, \psi)$ if there exist infinitely many `shifted'
rational points
$$ \Big( \frac{p_1- \g_1}{q}, \ldots, \frac{p_n - \g_n}{q} \Big) $$
with $q >0$, such that the inequalities
$$
|x_i - (p_i - \g_i)/q | \, < \, \psi(q)/q \,
$$
are simultaneously satisfied for $ 1 \le i \le n$. The following is
the natural generalisation of the simultaneous Khintchine-Jarn\'{\i}k
theorem to the inhomogeneous setup. For further details, see
\cite{Beresnevich-Bernik-Dodson-Velani-Roth, BDV06} and references
within.

\vspace*{1ex}

\begin{thm}[Inhomogeneous Khintchine-\Jarnik]
  \label{inhomKJ}
  Let $\psi:\N\to \Rp$ be a monotonic function, $\bm\g\in \R^n$ and
  $s \in (0,n]$.  Then
$$
\cH^s( W_{\bm\g}(n,\psi))=\left\{\begin{array}{lll} 0 & \ds\text{if }
    \;\;\; \sum_{r=1}^\infty \;r^{n-s}\psi(r)^{s} <\infty \, , \\[2ex]
    & \\ \cH^s(\I^{n}) & \ds\text{if } \;\;\; \sum_{r=1}^\infty \;
    r^{n-s}\psi(r)^{s} =\infty \ \,. &
                                 \end{array}\right.
$$
\end{thm}

\vspace*{2ex}

\begin{rem} \label{inhomDir} For the sake of completeness we state the
  inhomogeneous analogue of Hurwitz's Theorem due to Khintchine
  \cite[\S10.10]{Hua}: {\em for any irrational $ x \in \R $,
    $\gamma \in \R$ and $\ve >0$, there exist infinitely many integers
    $q >0$ such that}
  \begin{equation*} \label{inhomdir} q \, \| q x - \gamma\| \leq
    (1+\ve)/\sqrt{5} \, .
  \end{equation*}
  Note that presence of the $\ve$ term means that the inhomogeneous
  statement is not quite as sharp as the homogeneous one (i.e. when
  $\gamma=0$).  Also, for obvious reasons, in the inhomogeneous
  situation it is necessary to exclude the case that $x $ is rational.
\end{rem}

\vspace*{2ex}

We now swiftly move on to the inhomogeneous theory for manifolds.  In
short, the heuristics of \S\ref{rpnm}, adapted to the inhomogeneous
setup, gives evidence towards the following natural generalisation of
the Dream Theorem.

\vspace*{2ex}

\noindent\textbf{The Inhomogeneous Dream Theorem. \ } Let $\cM$ be a
non-degenerate sub-manifold of $\R^n$.  Let $\psi:\N\to \Rp$ be a
monotonic function and $\bm\g\in \R^n$. Then
\begin{equation*}\label{vb51}
  | \cM\cap W_{\bm\g}(n,\psi)  |_{\cM}
  =\left\{
    \begin{array}{ll}
      0 & {\rm if} \;\;\; \sum_{q=1}^{\infty} \;   \psi(q)^n  <\infty\;
          ,\\[4ex]
      1 & {\rm if} \;\;\; \sum_{q=1}^{\infty} \;   \psi(q)^n
          =\infty \; .
    \end{array}\right.
\end{equation*}

\vspace*{2ex}

Regarding what is known, the current state of knowledge is absolutely
in line with the homogeneous situation.  The inhomogeneous analogue of
the extremality result of Kleinbock $\& $ Margulis \cite{KM98} is
established in \cite{Beresnevich-Velani-10:MR2734962,
  Beresnevich-Velani-Moscow}.  We will return to this in
\S\ref{ITPtheory} below. For planar curves, the Inhomogeneous Dream
Theorem is established in
\cite{Beresnevich-Vaughan-Velani-11:MR2777039}. Beyond planar curves,
the results in \cite{BVVZcon, BVVZdiv} are obtained within the
inhomogeneous framework.  So in summary, the Inhomogeneous Dream
Theorem is in essence fact for non-degenerate sub-manifolds $\cM$ of
$\R^n$ apart from the case of convergence when $n \ge 3 $ and
$d \le (n+1)/2$.

\subsubsection{Inhomogeneous extremality and a transference
  principle \label{ITPtheory}}

First we need to decide on what precisely we mean by inhomogeneous
extremality.  With this in mind, a manifold $\cM$ is said to be
\emph{simultaneously inhomogeneously extremal}\/ ($\SIE$ for short) if
for every $\bm\gamma \in\R^n$,
\begin{equation}\label{e:004}
  \left|  \cM\cap W_{\bm\gamma}(n,  \textstyle{\frac{1+\ve}{n}}) \right|_{\cM}   =0   \qquad  \forall   \  \ve>0  \,.
\end{equation}
On the other hand, a manifold $\cM$ is said to be \emph{dually
  inhomogeneously extremal}\/ ($\DIE$ for short) if for every
$\gamma \in\R$,
\begin{equation*}
  \left|  \cM\cap W_{\gamma}^*(n,  n + \ve)  \right|_{\cM}   =0    \qquad  \forall   \  \ve>0   \,.
\end{equation*}
Here, given $\tau >0$ and a fixed point $\g \in \R$,
$W_{\gamma}^*(n, \tau) $ is the \emph{inhomogeneous} set of
\emph{dually $\tau$-well approximable} points consisting of points
$\vx\in \I^n$ for which the inequality
$$
\|\vv q\cdot \vv x-\gamma\|<|\vv q|^{-\tau}
$$
holds for infinitely many $\vv q\in\Z^n$.  Moreover, a manifold $\cM$
is simply said to be \emph{inhomogeneously extremal} if it is both
$\SIE$ and $\DIE$.

As mentioned in Remark \ref{extem=}, in the homogeneous case
($\bm\gamma$=0) the simultaneous and dual forms of extremality are
equivalent. Recall that this is a simply consequence of Khintchine's
Transference Principle (Theorem \ref{KTPP}).  However, in the
inhomogeneous case, there is no classical transference principle that
allows us to deduce SIE from DIE and vice versa. The upshot is that
the two forms of inhomogeneous extremality have to be
treated separately. It turns out that establishing the dual form of
inhomogeneous extremality is technically far more complicated than
establishing the simultaneous form
\cite{Beresnevich-Velani-Moscow}. The framework developed in
\cite{Beresnevich-Velani-10:MR2734962} naturally incorporates both
forms of inhomogeneous extremality and indeed other stronger
(multiplicative) notions associated with the inhomogeneous analogue of
the Baker-Sprind\v zuk Conjecture.

\bigskip


\noindent\textbf{Conjecture.  } {\it Let $\cM$ be a non-degenerate
  sub-manifold of $\R^n$. Then $\cM$ is inhomogeneously extremal. }

The proof given in \cite{Beresnevich-Velani-10:MR2734962} of this
inhomogeneous conjecture relies very much on the fact that we know
that the homogeneous statement is true.  In particular, the general
inhomogeneous transference principle of
\cite[\S5]{Beresnevich-Velani-10:MR2734962} enables us to establish
the following transference for non-degenerate manifolds:

\begin{equation} \label{amen} \cM\text{ is extremal }\iff \cM\text{ is
    inhomogeneously extremal}.
\end{equation}
Clearly, this enables us to conclude that:
$$
\cM\text{ is SIE }\iff \cM\text{ is DIE}.
$$
In other words, a transference principle between the two forms of
inhomogeneous extremality does exist at least for the class of
non-degenerate manifolds.

\vspace*{2ex}

Trivially, inhomogeneous extremality implies (homogeneous)
extremality.  Thus, the main substance of \eqref{amen} is the reverse
implication.  This rather surprising fact relies on the fact that the
inhomogeneous $\limsup$ sets
$\cM\cap W_{\bm\gamma}(n, \textstyle{\frac{1+\ve}{n}}) $ and the
induced measure $ | \ . \ |_{\cM} $ on non-degenerate manifolds
satisfy the {\em intersection property} and the {\em contracting
  property} described in \cite[\S5]{Beresnevich-Velani-10:MR2734962}.
These properties are at the heart of the Inhomogeneous Transference
Principle \cite[Theorem 5]{Beresnevich-Velani-10:MR2734962} that
enables us to transfer zero measure statements for homogeneous
$\limsup$ sets to inhomogeneous $\limsup$ sets.  The general setup,
although quite natural, is rather involved and will not be reproduced
in these notes.  Instead, we refer the reader to the papers \cite{
  Beresnevich-Velani-10:MR2734962, Beresnevich-Velani-Moscow}.  We
advise the reader to first look at \cite{Beresnevich-Velani-Moscow} in
which the easier statement \begin{equation} \label{amen2} \cM\text{ is
    extremal } \quad \Longrightarrow \quad \cM\text{ is SIE}
\end{equation}
is established. This has the great advantage of bringing to the
forefront the main ideas of \cite{Beresnevich-Velani-10:MR2734962}
while omitting the abstract and technical notions that come with
describing the inhomogeneous transference principle in all its glory.
In order to illustrate the basic line of thinking involved in
establishing \eqref{amen2} and indeed \eqref{amen} we shall prove the
following statement concerning extremality on $\I=[0,1]$:
\begin{equation} \label{amen3} m(W(1+\ve)) = 0 \quad \Longrightarrow
  \quad m(W_{\gamma}(1+\ve))=0 \ \ \forall \ \ve >0.
\end{equation}
Of course it is easy to show that the inhomogeneous set
$W_{\gamma}(1+\ve)$ is of zero Lebesgue measure $m$ by using the
convergence Borel-Cantelli Lemma.  However, the point here is to
develop an argument that exploits the fact that we know the
homogeneous set $ W_{0}(1+\ve):=W(1+\ve)$ is of zero Lebesgue measure.

\vspace*{2ex}

To prove \eqref{amen3}, we make use of the fact that
$W_{\gamma}(1+\ve)$ is a $\limsup$ set given by
\begin{equation}\label{e:024}
  W_{\gamma}(1+\ve) \  = \  \bigcap_{s=1}^\infty\ \bigcup_{q=s}^{\infty} \bigcup_{
    p\in\Z} \ B^{\gamma}_{p,q}(\ve)\cap \I\, ,
\end{equation}
where, given $q\in \N $, $p\in\Z$, $\gamma \in\R$ and $\ve>0$
$$
B^{\gamma}_{p,q}(\ve):=\{\,y\in \R:|q y+ p+ \gamma|<|q|^{-1-\ve} \, \}
\ .
$$

\noindent As usual, if $B=B( x,r)$ denotes the ball (interval) centred
at $ x$ and of radius $r>0$, then it is easily seen that
$$
B^{\gamma}_{p,q}(\ve)= B\Big( \textstyle{\frac{
    p+\gamma}{q}},|q|^{-2-\ve} \Big) \ .
$$

Now we consider `blown up' balls $B^{\gamma}_{p,q}(\ve/2)$ and observe
that Lebesgue measure $m$ satisfies the following contracting property:
for any choice $q\in \N $, $p\in\Z$, $\gamma \in\R$ and $\ve>0$ we
have that
\begin{equation}\label{e:017}
  m \Big( B^{\gamma}_{p,q}(\ve)  \Big)   \ = \ \frac{2}{q^{2+ \ve} }  \  =   \  q^{-\frac{   \ve}{2} } \frac{2}{q^{2+ (\ve/2)} }     \ =   \  q^{-\frac{   \ve}{2} }   \ \   m \Big( B^{\gamma}_{p,q}(\ve/2)  \Big)   \, .
\end{equation}

Next we separate the balls $B^{\gamma}_{p,q}(\ve) $ into classes of
disjoint and non-disjoint balls. Fix $q\in\N$ and $p \in\Z $. Clearly,
there exists a unique integer $t=t(q)$ such that $2^t\le q <
2^{t+1}$.
The ball $B^{\gamma}_{p,q}(\ve) $ is said to be \emph{disjoint}\/ if
for every $q'\in\N$ with $2^t\le q'<2^{t+1}$ and every $p'\in\Z$
\begin{equation*}
  B^{\gamma}_{p,q}(\ve/2)\cap B^{\gamma}_{p',q'}(\ve/2)  \cap \I =\varnothing\,.
\end{equation*}
Otherwise, the ball $ B^{\gamma}_{p,q}(\ve/2) $ is said to be
\emph{non-disjoint}. This notion of disjoint and non-disjoint balls
enables us to decompose the $ W_{\gamma}(1+\ve)$ into the two limsup
subsets:
\begin{equation*}
  D^{\gamma}(\ve) \ := \ \bigcap_{s=0}^\infty\
  \bigcup_{t= s}^{\infty} \ \bigcup_{2^t\le
    |q|<2^{t+1}}\bigcup_{\stackrel{\scriptstyle
      p\in\Z}{B^{\gamma}_{p,q}(\ve)\text{ is disjoint}}}  \!\!\!\!\!\! B^{\gamma}_{p,q}(\ve)\cap \I \,,
\end{equation*}
and
\begin{equation*}
  N^{\gamma}(\ve) \ := \ \bigcap_{s=0}^\infty\
  \bigcup_{t= s}^{\infty} \ \bigcup_{2^t\le
    |q|<2^{t+1}}\bigcup_{\stackrel{\scriptstyle
      p\in\Z}{B^{\gamma}_{p,q}(\ve)\text{ is non-disjoint}}}  \!\!\!\!\!\! B^{\gamma}_{p,q}(\ve)\cap \I \,.
\end{equation*}

\noindent Formally,
$$
W_{\gamma}(1+\ve) \ = \ \bigcap_{s=1}^\infty\ \bigcup_{q=s}^{\infty}
\bigcup_{ p\in\Z} \ B^{\gamma}_{p,q}(\ve)\cap \I \ = \
D^{\gamma}(\ve) \ \cup \ N^{\gamma}(\ve) \ .
$$

\noindent We now show that
$m (D^{\gamma}(\ve)) = 0 = m( N^{\gamma}(\ve)) $.  This would clearly
imply \eqref{amen3}. Naturally, we deal with the disjoint and
non-disjoint sets separately.

\medskip

\noindent{\em The disjoint case: } By the definition of disjoint
balls, for every fixed $t$ we have that
\begin{equation*}\label{e:028}
  \begin{array}{rcl}
    \displaystyle\sum_{2^t\le q<2^{t+1}}\sum_{\stackrel{\scriptstyle
    p\in\Z}{B^{\gamma}_{ p,q}(\ve)\text{ is disjoint}}} \!\!\!\!  m( B^{\gamma}_{p,q}(\ve/2)\cap \I) & = & \displaystyle m \Big(\bigcup_{2^t\le q<2^{t+1}}\bigcup_{\stackrel{\scriptstyle
                                                                                                           p\in\Z}{B^{\gamma}_{p,q}(\ve)\text{ is disjoint}}} \!\!\!\!\!\!\!\! B^{\gamma}_{p,q}(\ve/2)\cap
                                                                                                           \I\Big) \\[7ex]
                                                                                                     & \le & m(\I)  \,   = \, 1 .
  \end{array}
\end{equation*}
This together with the contracting property \eqref{e:017} of the
measure $m$, implies that

\begin{eqnarray*}
  \displaystyle m \Big(\bigcup_{2^t\le q<2^{t+1}}  \!\! \bigcup_{\stackrel{\scriptstyle
  p\in\Z}{B^{\gamma}_{p,q}(\ve)\text{ is disjoint}}} \!\!\!\!\!\!\!\! B^{\gamma}_{p,q}(\ve)\cap
  \I\Big)  & = & \displaystyle\sum_{2^t\le q<2^{t+1}}\sum_{\stackrel{\scriptstyle
                 p\in\Z}{B^{\gamma}_{ p,q}(\ve)\text{ is disjoint}}} \!\!\!\!  m( B^{\gamma}_{p,q}(\ve)\cap \I)  \\[3ex]
           & \le &  \displaystyle\sum_{2^t\le q<2^{t+1}}\sum_{\stackrel{\scriptstyle
                   p\in\Z}{B^{\gamma}_{ p,q}(\ve)\text{ is disjoint}}} \!\!\!\! q^{-\frac{\ve}{2}} \ \  m( B^{\gamma}_{p,q}(\ve/2)\cap \I) \\[3ex]
           & \le & 2^{- t \frac{\ve}{2}} \displaystyle\sum_{2^t\le q<2^{t+1}}\sum_{\stackrel{\scriptstyle
                   p\in\Z}{B^{\gamma}_{ p,q}(\ve)\text{ is disjoint}}} \!\!\!\!   m( B^{\gamma}_{p,q}(\ve/2)\cap \I) \\[1ex] & \le & 2^{- t \frac{\ve}{2}} \, .
\end{eqnarray*}

\noindent Since
$ \sum_{t=1}^{\infty} 2^{- t \frac{\ve}{2}} < \infty $, the
convergence Borel-Cantelli Lemma implies that
$$
m (D^{\gamma}(\ve)) = 0 \, .
$$

\medskip

\noindent{\em The non-disjoint case: } Let $B^{\gamma}_{ p,q}(\ve)$ be
a non-disjoint ball and let $t=t(q)$ be as above. Clearly
$$
B^{\gamma}_{p,q}(\ve)\subset B^{\gamma}_{ p,q}(\ve/2)\,.
$$
By the definition of non-disjoint balls, there is another ball
$B^{\gamma}_{p',q'}(\ve/2)$ with $2^t\le q<2^{t+1}$ such that

\begin{equation}\label{e:031}
  B^{\gamma}_{p,q}(\ve/2)\cap B^{\gamma}_{\vv
    p',q'}(\ve/2)\cap \I\not=\varnothing\,.
\end{equation}

\noindent It is easily seen that $q'\not=q$, as otherwise we would
have that
$B^{\gamma}_{ p,q}(\ve/2)\cap B^{\gamma}_{
  p',q}(\ve/2)=\varnothing$.
The point here is that rationals with the same denominator $q$ are
separated by $1/q$.  Take any point $ y $ in the non-empty set
appearing in (\ref{e:031}).  By the definition of
$B^{\gamma}_{p,q}(\ve/2)$ and $B^{\gamma}_{ p',q'}(\ve/2)$, it follows
that
\begin{equation*}\label{e:032}
  |q y + p+\gamma| \ < \ q^{-1-\frac\ve2} \ \le \
  2^{t(-1- \frac\ve2)}
\end{equation*}
and
\begin{equation*}\label{e:033}
  |q'y +p'+\gamma| \ < \ (q')^{-1-\frac\ve2} \ \le \
  2^{t(-1-\frac\ve2)}\,.
\end{equation*}
On combining these inequalities in the obvious manner and assuming
without loss of generality that $q > q'$, we deduce that
\begin{equation}\label{e:034}
  |\underbrace{(q-q')}_{q''} y +\underbrace{( p- p')}_{
    p''}| \ < \ 2\cdot 2^{t(-1-\frac\ve2)}  \ < \
  2^{(t+2)(-1-\frac\ve3)}
\end{equation}
for all $t$ sufficiently large. Furthermore, $0<q''\le 2^{t+2}$ which
together with (\ref{e:034}) yields that
\begin{equation*}\label{ee:033}
  |q'' y + p''| \ < \ (q'')^{-1-\frac\ve3}\,.
\end{equation*}
If the latter inequality holds for infinitely many different
$q''\in\N$, then $y \in W(1+ \ve/3) $.  Otherwise, there is a fixed
pair $( p'', q'') \in \Z \times \N$ such that (\ref{e:034}) is
satisfied for infinitely many $t$. Thus, we must have that
$q'' y+p''=0$ and so $y$ is a rational point.  The upshot of the
non-disjoint case is that
$$
N^{\gamma}(\ve) \ \subset \ W(1+ \ve/3) \ \cup \ \Q \ .
$$
However, we are given that the homogeneous set $ W(1+ \ve/3)$ is of
measure zero and since $\Q$ is countable, it follows that
$$
m(N^{\gamma}(\ve)) \ = \ 0 \ .
$$
This completes the proof of \eqref{amen3}.

\vspace*{4ex}

\subsection{The inhomogeneous multiplicative theory}

For completeness, we include a short section surveying recent striking
developments in the theory of inhomogeneous multiplicative Diophantine
approximation. Nevertheless, we start by highlighting the fact that
there remain gapping holes in the theory.

Given $\psi:\N\to \Rp$ and a fixed point
$\bm\g=(\g_1, \dots \g_n) \in \R^n$, let
\begin{equation} \label{inhommultdef} W^{\times}_{\bm\g}(n,\psi): = \{
  \vx \in {\rm I}^n \colon \| q x_1 - \gamma_1 \| \, \ldots \, \| q
  x_n - \gamma_n \| <\psi(q) \text{ for infinitely many } q\in\N \} \
\end{equation}
denote the \emph{inhomogeneous} set of \emph{multiplicatively
  $\psi$-well approximable} points $\vx\in \I^n$.  When
$\bm\g = \{\mathbf{0} \}$, the corresponding set
$ W^{\times}_{\bm\g}(n,\psi)$ naturally coincides with the homogeneous
set $ W^{\times}(n,\psi)$ given by \eqref{multdef} in
\S\ref{multsec}. It is natural to ask for an inhomogeneous
generalisation of Gallagher's Theorem (\S\ref{multsec}, Theorem
\ref{gallmult}). A straightforward `volume' argument making use of the
$\limsup$ nature of $ W^{\times}_{\bm\g}(n,\psi)$, together with the
convergence Borel-Cantelli Lemma implies the following statement.

\begin{lem}[Inhomogeneous Gallagher:
  convergence] \label{inhomconvgall} Let $\psi:\N\to \Rp$ be a
  monotonic function and $\bm\g\in \R^n$.  Then
$$
m_n(W^{\times}_{\bm\g}(n,\psi)) = 0 \quad if \quad \displaystyle
\sum_{q=1}^{\infty} \; \psi(q) \log^{n-1} q <\infty \ .
$$
\end{lem}

\vspace*{2ex}

\noindent The context of Remark \ref{nonmonremark} remains valid in
the inhomogeneous setup; namely, we can remove the condition that
$\psi$ is monotonic, if we replace the above convergence sum condition
by $\sum \psi(q) |\log \psi(q)|^{n-1} <\infty$.

\vspace*{1ex}

Surprisingly, the divergence counterpart of Lemma \ref{inhomconvgall}
is not known.

\vspace*{2ex}

\begin{conjecture}[Inhomogeneous Gallagher:
  divergence] \label{inhomdivgall} Let $\psi:\N\to \Rp$ be a monotonic
  function and $\bm\g\in \R^n$.  Then
$$
m_n(W^{\times}_{\bm\g}(n,\psi)) = 1 \quad if \quad \displaystyle
\sum_{q=1}^{\infty} \; \psi(q) \log^{n-1} q = \infty \ .
$$
\end{conjecture}

\vspace*{1ex}

Restricting our attention to $n=2$, it is shown in \cite[Theorem
13]{bhv} that the conjecture is true if given
$\bm\g=(\gamma_1,\gamma_2) \in \R^2$, either $\gamma_1=0$ or
$\gamma_2=0$. In other words, we are able to deal with the situation
in which one of the two ``approximating quantities'' is inhomogeneous
but not both. For further details see \cite[\S2.2]{bhv}.

\vspace*{1ex}

We now turn our attention to the Hausdorff theory. Given that the
Lebesgue theory is so incomplete, it would be reasonable to have low
expectations for a coherent Hausdorff theory. However, when $n=2$, we
are bizarrely in pretty good shape.  To begin with note that
\begin{equation}\label{vb45}
  \text{if $ \ s\le 1  \ $ then $  \ \cH^s(W^{\times}_{\bm\g}(2,\psi))=\infty  \ $ irrespective of approximating function $ \ \psi $.}
\end{equation}
To see this, given $\bm\g=(\gamma_1,\gamma_2) \in \R^2$, we observe
that for any $\alpha \in W_{\gamma_1}(1,\psi) $ the whole line
$x_1=\alpha$ within the unit interval is contained in
$W^{\times}_{\bm\g}(2,\psi)$. Hence,
\begin{equation}\label{vb35}
  W_{\gamma_1}(1,\psi)  \times\I\subset W^{\times}_{\bm\g}(2,\psi) \,.
\end{equation}
It is easy to verify that $W_{\gamma_1}(1,\psi)$ is an infinite set
for any approximating function $\psi$ and so \eqref{vb35} implies
\eqref{vb45}.  Thus, when considering the $s$-dimensional Hausdorff
measure of $W^{\times}_{\bm\g}(2,\psi)$, there is no loss of
generality in assuming that $s \in (1,2]$. The following inhomogeneous
multiplicative analogue of Jarn\'ik's theorem is established in
\cite[Theorem 1]{BVdani}.

\begin{thm}\label{t1dani}
  Let $\psi:\N\to \Rp$ be a monotonic function, $\bm\g\in \R^2$ and
  $s \in (1,2)$. Then
  \begin{equation}\label{e:006+}
    \cH^{s}\big(W^{\times}_{\bm\g}(2,\psi)  \big) \, = \,
    \left\{\begin{array}{ll}
             0  & {\rm if} \;\;\;
                  \textstyle{\sum_{q=1}^\infty} \; q^{2-s}\psi^{s-1}(q) \; <\infty \; ,\\[1ex]
                &
             \\
             \infty & {\rm if} \;\;\;
                      \textstyle{ \sum_{q=1}^\infty } \;  q^{2-s}\psi^{s-1}(q) \;
                      =\infty \; .
           \end{array}\right.
       \end{equation}
     \end{thm}

     \medskip

     \begin{rem} {\rm Recall that Gallagher's multiplicative statement
         and its conjectured inhomogeneous generalisation (Conjecture
         \ref{inhomdivgall}) have the extra `log factor' in the
         Lebesgue `volume' sum compared to Khintchine's simultaneous
         statement (Theorem \ref{inhomKJ} with $s=n=2$).  A priori, it
         is natural to expect the log factor to appear in one form or
         another when determining the Hausdorff measure $\cH^s$ of
         $W^{\times}_{\bm\g}(2,\psi) $ for $ s \in (1, 2)$.  This, as
         we see from Theorem \ref{t1dani}, is very far from the
         truth. The `log factor' completely disappears.  Thus, genuine
         `fractal' Hausdorff measures are insensitive to the
         multiplicative nature of $W^{\times}_{\bm\g}(2,\psi)$.  }
     \end{rem}

     \medskip

     \begin{rem} {\rm Note that in view of the previous remark, even
         if we had written $\cH^{s}(\I^2)$ instead of $\infty$ in the
         divergence case of Theorem \ref{t1dani} , it is still
         necessary to exclude the case $s=2$.  }
     \end{rem}

     For $n > 2$, the proof given in \cite{BVdani} of
     Theorem~\ref{t1dani} can be adapted to show that for any
     $s\in(n-1,n)$
$$
\cH^{s}\big(W^{\times}_{\bm\g}(n,\psi) \big) \, = \, 0 \qquad\text{if}
\qquad {\sum_{q=1}^\infty} \; q^{n-s}\psi^{s+1-n}(q)\log^{n-2}q \;
<\infty \; .
$$
Thus, for convergence in higher dimensions we lose a log factor from
the Lebesgue volume sum appearing in Gallagher's homogeneous result
and indeed Lemma \ref{inhomconvgall}. This of course is absolutely
consistent with the $n=2$ situation given by Theorem~\ref{t1dani}.
Regarding a divergent statement, the arguments used in proving
Theorem~\ref{t1dani} can be adapted to show that for any $s\in(n-1,n)$
$$
\cH^{s}\big(W^{\times}_{\bm\g}(n,\psi) \big) \, = \, \infty
\qquad\text{if} \qquad {\sum_{q=1}^\infty} \; q^{n-s}\psi^{s+1-n}(q)
\; =\infty \; .
$$
Thus, there is a discrepancy in the above `$s$-volume' sum conditions
for convergence and divergence when $n > 2$. In view of this, it
remains an interesting open problem to determine the necessary and
sufficient condition for
$\cH^{s}\big(W^{\times}_{\bm\g}(n,\psi) \big)$ to be zero or infinite
in higher dimensions.

\vspace*{3ex}

\subsubsection{The multiplicative theory for
  manifolds \label{multmanifolds}}

Let $\cM$ be a non-degenerate sub-manifolds of $\R^n$. In a nutshell,
as in the simultaneous case, the overarching problem is to develop a
Lebesgue and Hausdorff theory for
$ \cM \cap W^{\times}_{\bm\g}(n,\psi) $.  Given that our current
knowledge for the independent theory (i.e. when $\cM = \R^n$) is
pretty poor, we should not expect too much in terms of the dependent
(manifold) theory.  We start with describing coherent aspects of the
Lebesgue theory.  The following is the multiplicative analogue of the
statement that $\cM$ is inhomogeneously extremal.  Given $\tau >0$ and
a fixed point $\bm\g \in \R^n$, we write
$ W^{\times}_{\bm\g}(n, \tau) $ for the set
$ W^{\times}_{\bm\g}(n,\psi) $ with $ \psi(q) = q^{-\tau}$.

\begin{thm}
  Let $\cM$ be a non-degenerate sub-manifold of $\R^n$. Then
  \begin{equation*}
    \left|  \cM\cap W^{\times}_{\bm\g}(n, 1 + \ve) \right|_{\cM}   =0    \qquad  \forall   \  \ve>0   \,.
  \end{equation*}
\end{thm}

\vspace*{1ex}

\noindent In the homogeneous case, the above theorem is due to
Kleinbock $\&$ Margulis \cite{KM98} and implies that non-degenerate
manifolds are \emph{strongly extremal} (by definition).  It is easily
seen that strongly extremal implies extremal. The inhomogeneous
statement is established via the general Inhomogeneous Transference
Principle developed in \cite{Beresnevich-Velani-10:MR2734962}.

\vspace*{2ex}

Beyond strong extremality, we have the following convergent statement
for the Lebesgue measure of $\cM \cap W^{\times}_{\bm\g}(n,\psi) $ in
the case $\cM$ is a planar curve $\cC$ .

\begin{thm}\label{t200} Let $\psi:\N\to \Rp$ be a monotonic function
  and $\bm\g\in \R^n$.  Let $\cC$ be a non-degenerate planar
  curve. Then
  \begin{equation}\label{e:00600}
    \left|  \cC \cap  W^{\times}_{\bm\g}(2,\psi)   \right|_{\cC} \, = \,
    0   \qquad\text{if} \qquad
    {\sum_{q=1}^\infty} \; \psi(q)\log q \; <\infty \; .
  \end{equation}
\end{thm}

\vspace*{1ex}

\noindent The homogeneous case is established in \cite[Theorem
1]{BadLev}.  However, on making use of the upper bound counting
estimate appearing within Theorem 2 of
\cite{Beresnevich-Vaughan-Velani-11:MR2777039}, it is easy to adapt
the homogeneous proof to the inhomogeneous setup. The details are left
as an \emph{exercise}.  Just as in the homogeneous theory, obtaining
the counterpart divergent statement for the Lebesgue measure of
$\cC \cap W^{\times}_{\bm\g}(2,\psi)$ remains a stubborn problem.
However, for genuine fractal Hausdorff measures $\cH^s$ we have a
complete convergence/divergence result \cite[Theorem 2]{BVdani}.

\begin{thm}\label{t2001}
  Let $\psi:\N\to \Rp$ be a monotonic function, $\bm\g\in \R^n$ and
  $s \in (0,1)$.  Let $\cC$ be a $C^{(3)}$-planar curve with non-zero
  curvature everywhere apart from a set of $s$-dimensional Hausdorff
  measure zero.  Then
  \begin{equation*}\label{e:006001}
    \cH^{s}\big(  \cC \cap   W^{\times}_{\bm\g}(2,\psi)  \big)
    \, = \,
    \left\{\begin{array}{ll}
             0  & {\rm if} \;\;\;
                  \textstyle{\sum_{q=1}^\infty} \; q^{1-s}\psi^s(q) \; <\infty  ,\\[2ex]
             \infty & {\rm if} \;\;\;
                      \textstyle{ \sum_{q=1}^\infty } \;  q^{1-s}\psi^s(q) \;
                      =\infty  .
           \end{array}\right.
       \end{equation*}
     \end{thm}

     \noindent

     \noindent It is evident from the proof of the divergence case of
     the above theorem \cite[\S2.1.3]{BVdani}, that imposing the
     condition that $\cC$ is a $C^{(1)}$-planar curve suffices.

     \vspace*{1ex}

     Beyond planar curves, the following lower bound dimension result
     represents the current state of knowledge.

     \begin{thm} \label{allweknow} Let $\cM$ be an arbitrary Lipschitz
       manifold in $\R^n$ and $\bm\g\in \R^n$. Then, for any
       $\tau \ge 1$
       \begin{equation}\label{e:006002} \dim \big( \cM \cap
         W^{\times}_{\bm\g}(n, \tau) \big) \ \geq \ \dim \cM -1 +
         \frac{2}{1+\tau} \, .
       \end{equation}
     \end{thm}

     \noindent The homogeneous case is established in \cite[Theorem
     5]{Beresnevich-Velani-07:MR2285737}.  The homogeneous proof
     \cite[\S6.2]{Beresnevich-Velani-07:MR2285737} rapidly reduces to
     the inequality
     $$ \dim \big( \cM \cap W^{\times}_{\bm 0}(n, \tau) \big) \ \geq \
     \dim \cM -1 + \dim W^{\times}_{0}(1, \tau) \, .
$$
But $ W^{\times}_{0}(1, \tau) := W(1, \tau) $ and the desired
statement follows on applying the \Jarnik-Besicovitch Theorem (Theorem
\ref{jb}).  Now, Theorem \ref{inhomKJ} implies that the inhomogeneous
generalisation of the \Jarnik-Besicovitch Theorem is valid; namely
that, for any $\g \in \R$ and $ \tau \ge 1 $
$$
\dim W_{\g}(1, \tau) \ = \ \frac{2}{1+\tau} \, .
$$
Thus, the short argument given in
\cite[\S6.2]{Beresnevich-Velani-07:MR2285737} can be adapted in the
obvious manner to establish Theorem \ref{allweknow}.

 %
%
%

\subsubsection{Cassels' problem}

A straightforward consequence of Theorem \ref{inhomKJ} with $s=2$
(inhomogeneous Khintchine), is that for any
$\bm\g=(\g_1, \g_2) \in \R^2$, the set
\begin{equation} \label{inhommultdefus} W^{\times}_{\bm\g}: = \{ \vx
  \in {\rm I}^2 \colon \liminf_{q\to \infty} q \, \| q x_1 - \gamma_1
  \| \, \| q x_2 - \gamma_2 \| \, = \, 0 \}
\end{equation}
is of full Lebesgue measure; i.e. for any $\bm\g \in \R^2$, we have
that
$$
m_2(W^{\times}_{\bm\g}) \, = \, 1 \, .
$$
Of course, one can actually deduce the stronger `fiber' statement that
for any $x \in \I $ and $\bm\g=(\g_1, \g_2) \in \R^2$, the set
$$
\{ y \in {\rm I} \colon \liminf_{q\to \infty} q \, \| q x - \gamma_1
\| \, \| q y - \gamma_2 \| \, = \, 0 \}
$$
is of full Lebesgue measure.  In a beautiful paper
\cite{Shapira-11:MR2753608}, Shapira establishes the following
statement which solves a problem of Cassels dating back to the
fifties.

\begin{thm}[U. Shapira]\label{usthm}
  \begin{equation*}
    m_2  \Big(  \bigcap_{\bm\g \in \R^2} W^{\times}_{\bm\g} \Big)   = 1   \, .
  \end{equation*}
\end{thm}

\noindent Thus, almost any pair of real numbers $(x_1,x_2) \in \R^2 $
satisfies
\begin{equation} \label{us34} \forall \ (\g_1, \g_2) \in \R^2 \qquad
  \liminf_{q\to \infty} q \, \| q x_1 - \gamma_1 \| \, \| q x_2 -
  \gamma_2 \| \, = \, 0 \, .
\end{equation}
In fact, Cassels asked for the existence of just one pair $(x_1,x_2)$
satisfying \eqref{us34}. Furthermore, Shapira showed that if
$1, x_1, x_2$ form a basis for a totally real cubic number field, then
$(x_1,x_2)$ satisfies \eqref{us34}.  On the other hand, if
$1, x_1, x_2$ are linearly dependent over $\Q$, then $(x_1,x_2)$
cannot satisfy \eqref{us34}.

\vspace{1ex}

Most recently, Gorodnik $\&$ Vishe \cite{Gorodnik-Vishe} have
strengthened Shapira's result in the following manner: {\em almost any
  pair of real numbers $(x_1,x_2) \in \R^2 $ satisfies
  \begin{equation*} \label{gv34} \forall \ (\g_1, \g_2) \in \R^2
    \qquad \liminf_{q\to \infty} q \, \log_5 \! q \| q x_1 - \gamma_1
    \| \, \| q x_2 - \gamma_2 \| \, = \, 0 \, ,
  \end{equation*}}
where $\log_5$ is the fifth iterate of $\log$.
This `rate' result makes a  contribution towards the following open problem.

\vspace{1ex}

\begin{conjecture} Almost any pair of real numbers
  $(x_1,x_2) \in \R^2 $ satisfies
  \begin{equation} \label{task2} \forall \ (\g_1, \g_2) \in \R^2
    \qquad \liminf_{q\to \infty} q \, \log q \, \| q x_1 - \gamma_1 \|
    \, \| q x_2 - \gamma_2 \| \, < \infty \, .
  \end{equation}
\end{conjecture}

\medskip

\begin{rem} {\rm It is relatively straightforward to show (\emph{exercise}) that for any $\tau > 2$
$$
\left\{ \vx \in  {\rm I}^2 \colon  \forall  \ (\g_1,  \g_2) \in \R^2    \quad    \liminf_{q\to \infty}  q \, \log^{\tau} \! \! q \,   \| q x_1  - \gamma_1 \| \, \| q x_2 - \gamma_2 \|   \, =   \, 0  \right\} = \varnothing \, .
$$
}
\end{rem}

We end this section by mentioning Cassels' problem within the context
of Diophantine approximation on manifolds. By exploiting the work of
Shah \cite{shah}, it is shown in \cite{grv} that for any
non-degenerate planar curve $\cC$
\begin{equation*}
  \left| \  \cC  \ \cap \  \textstyle{\bigcap_{\bm\g \in \R^2}} W^{\times}_{\bm\g} \ \right|_{\cC}   = 1   \, .
\end{equation*}

\section{The  badly approximable theory}

\newcommand{\diam}{\operatorname{diam}}
\newcommand{\cK}{\mathcal{K}}

We have had various discussions regarding badly approximable points in earlier sections,  in particular within  \S\ref{sec:beating-dirichlet} and  \S\ref{BadinR^n}. We mentioned that  the badly approximable set $\Bad $ and its higher dimensional generalisation  $\bad(i_1, \ldots,i_n)$  are small in the sense that they are of zero Lebesgue measure but are nevertheless large in the sense that they have  full Hausdorff dimension. In this section we outline the basic techniques used in establishing  the dimension results.  For transparency and simplicity, we shall concentrate on the one-dimensional case. We begin with the classical nearly 100 years old result due to Jarn\'ik.

\subsection{$\Bad$ is of full dimension}\label{Jar}

The key purpose of this section is to introduce a basic Cantor set construction and show how it can be  utilised to show that  $\Bad$ is of maximal dimension  -- a result first established by \Jarnik{} in \cite{Ja28}.  Towards the end we shall mention the additional ideas required in higher dimensions.

\begin{thm}[Jarn\'ik, 1928]\label{t7.1}
The Hausdorff dimension of $\Bad$ is\/ one; that is
$$ \dim \Bad = 1 \, . $$
\end{thm}

The proof utilises the following simple \textbf{Cantor set construction.}
Let $R,M\in\N$ and $M\le R-1$. Let $E_0=[0,1]$. Partition the interval $E_0$ into $R$ equal close subinterval and remove any $M$ of them. This gives $E_1$ - the union of $(R-M)$ closed intervals $\{I_{1,j}\}_{1\le j\le R-M}$ of length $|I_{1,j}|=R^{-1}$. Then repeat the procedure: partition each interval $I_{1,j}$ within $E_1$ into $R$ equal close subinterval and remove any $M$ intervals of the partitioning of each $I_{1,j}$. This procedure gives rise to  $E_2$ - the union of $(R-M)^2$ closed intervals $\{I_{2,j}\}_{1\le j\le (R-M)^2}$ of length $|I_{2,j}|=R^{-2}$. The process goes on recurrently/inductively as follows: for $ n \ge 1$, given that $E_{n-1}$ is constructed and represents the union of $(R-M)^{n-1}$ closed intervals $\{I_{n-1,j}\}_{1\le j\le (R-M)^{n-1}}$ of length $|I_{n-1,j}|=R^{-(n-1)}$, to construct $E_{n}$ we
\begin{itemize}
  \item[(i)]
  partition each interval $I_{n-1,j}$ within $E_{n-1}$ into $R$ equal closed subintervals, and \\[0.5ex]
  \item[(ii)]
  remove any $M$ of the $R$ intervals of the above partitioning of each $I_{n-1,j}$.
\end{itemize}
 Observe that $E_n$ will be the union of exactly $(R-M)^n$ closed intervals $\{I_{n,j}\}_{1\le j\le (R-M)^n}$ of length $|I_{n,j}|=R^{-n}$.
 The corresponding Cantor set is defined to be
$$
\cK:=\bigcap_{n=0}^\infty E_n\,.
$$

\medskip

\begin{rem}
Of course the Cantor set constructed above is not unique and depends on the specific choices of $M$ intervals being removed in each case. Indeed, there are continuum many possibilities for the resulting set $\cK$. For example, if $R=3$, $M=1$ and we always remove the middle interval of the partitioning, the set $\cK$ is the famous middle third Cantor set as described in  Example~\ref{eg382} of \S\ref{DB}.
\end{rem}

\medskip

Trivially, the Cantor set  $\cK$ is non-empty since it is the intersection of a nested sequence of closed intervals within $[0,1]$. Indeed, if $ 0 \leq  M \le R-2$ then we have that $\cK$ is uncountable.   The following result relates  the Hausdorff dimension of $\cK$  to the parameters $R$ and $M$ associated with $\cK$.

\medskip

\begin{lem}\label{lem9+1}
Let $\cK$ be the Cantor set constructed above.  Then
\begin{equation}\label{vb+1}
\dim \cK =\frac{\log(R-M)}{\log R}.
\end{equation}
\end{lem}

\begin{proof}
Let $\{I_{n,j}\}_{1\le j\le (R-M)^n}$ be the collection of intervals within $E_n$  associated  with the  construction of $\cK$. Recall that this is a collection of $(R-M)^n$ closed intervals, each of length $R^{-n}$. Naturally, $\{I_{n,j}\}_{1\le j\le (R-M)^n}$ is a cover of $\cK$. Furthermore, for every $\rho>0$ there is a sufficiently large $n$ such that $\{I_{n,j}\}_{1\le j\le (R-M)^n}$ is a $\rho$-cover of $\cK$ -- simply make sure that $R^{-n}<\rho$. Observe that
$$
\sum_j\diam(I_{n,j})^s=(R-M)^{n}R^{-ns}=1\,  \qquad {\rm where  \ } \quad s:=\frac{\log(R-M)}{\log R}.
$$
Hence, by definition, $\cH^s_\rho(\cK)\le1$ for all sufficiently small $\rho>0$. Consequently,
$\cH^s(\cK)\le1$ and  it follows that
 $$\dim\cK  \, \le  \,  s \, . $$

\medskip

\noindent For the lower  bound,  let $0<\rho<1$ and $\{B_i\}$ be an arbitrary $\rho$-cover of $\cK$. We  show that
$$\sum_i\diam(B_i)^s\ge \kappa,$$ where $s$ is  as above and the constant $\kappa > 0$  is independent of the cover. Without loss of generality,  we will assume that each $B_i$ is an  open interval.
Since $\cK $ is the intersection of closed subsets of $[0,1]$, it is bounded and closed and hence compact. Therefore, $\{B_i\}$ contains a finite subcover. Thus,   without  loss of generality,  we can assume that $\{B_i\}$ is a finite $\rho$-cover of $\cK$. For each $B_i$, let $k\in\Z$ be the unique integer such that
$$
R^{-(k+1)}\le\diam(B_i)<R^{-k}\,.
$$
Then
$B_i$ intersects at most two intervals of $E_k$ as the intervals in $E_k$ are $R^{-k}$ in length.
If $j\ge k$, then $B_i$ intersects at most
\begin{equation}\label{vbmn+}
2(R-M)^{j-k}=2(R-M)^jR^{-sk}\le 2(R-M)^jR^s\diam(B_i)^s
\end{equation}
intervals within $E_j$. These are the intervals that are contained in the (at most) two intervals of $E_k$ that intersect $B_i$.  Now choose $j$ large enough so that
$$
R^{-(j+1)}  \, \le  \,  \diam(B_i)\qquad\forall\ B_i\,.
$$
This is possible since  the cover $\{B_i\}$ is finite.
Since $\{B_i\}$ is a cover of $\cK$, it must intersect every interval of $E_j$. There are $(R-M)^j$ intervals within $E_j$. Hence, by \eqref{vbmn+} it follows  that
$$
(R-M)^j\le \sum_i2(R-M)^jR^s\diam(B_i)^s\, .
$$
The upshot of this is that for any $\rho$-cover $\{B_i\}$ of $\cK$,   we have that
$$
\sum_i\diam(B_i)^s\ge \tfrac12R^{-s}=\frac1{2(R-M)}\,.
$$
Hence, by definition,  we have   that $\cH^s_\rho(\cK)\ge \tfrac1{2(R-M)}$ for all sufficiently small $\rho>0$. Therefore, $\cH^s(\cK)\ge \tfrac1{2(R-M)} > 0 $ and  it follows that
$$
\dim\cK\ge s=\frac{\log(R-M)}{\log R}
$$
as required.
\end{proof}

\bigskip

Armed with Lemma \ref{lem9+1}, it is relatively straight forward to prove \Jarnik{}'s full dimension result.

\medskip

\begin{proof}[Proof of Theorem~\ref{t7.1}]
  Let $R\ge4$ be an integer. For $n\in\Z$, $n\ge0$ let
\begin{equation}\label{Q_n}
  Q_n=\{p/q\in\Q:R^{\frac{n-3}2}\le q<R^{\frac {n-2}2}\}\subset\Q\,,
\end{equation}
where $p/q$ is a reduced fraction of integers.
  Observe that $Q_0=Q_1=Q_2=\varnothing$, that the sets $Q_n$ are disjoint and that
   \begin{equation}\label{Q}
     \Q=\bigcup_{n=3}^\infty Q_n\,.
\end{equation}
Furthermore,  note that
\begin{equation}\label{vbvb+}
  \left|\frac pq-\frac{p'}{q'}\right|\ge \frac1{q'q}> R^{-n+2}\qquad\text{for different $p/q$ and $p'/q'$ in $Q_n$}.
\end{equation}
Fix $0<\delta\le \tfrac12$. Then for $p/q\in Q_n$,  define the \emph{dangerous interval} $ \Delta(p/q)$ as follows:
\begin{equation}\label{Dang}
\Delta(p/q):=\Big\{x\in[0,1]:\left|x-\frac pq\right|<\delta R^{-n}\Big\}\,.
\end{equation}
The goal is to construct  a Cantor set
$
\cK=\bigcap_{n=0}^\infty E_n
$ such that for every $n\in\N$
  \begin{equation}\label{vbc}
    E_n\cap\Delta(p/q)=\varnothing\qquad\text{for all }~p/q\in Q_n\,.
  \end{equation}
To this end, let $E_0=[0,1]$ and suppose that  $E_{n-1}  $  has already been constructed. Let $I$ be any of the intervals $I_{n-1,j}$ within $E_{n-1}$. Then $|I|=R^{-n+1}$.
By \eqref{vbvb+} and \eqref{Dang}, there is at most one dangerous interval $\Delta(p_I/q_I)$ with $p_I/q_I\in Q_n$ that intersects $I$.  Partition $I$ into $R$ closed subintervals of length $R^{-n}=R^{-1}|I|$. Note that since $\delta\le\tfrac12$, the dangerous interval $\Delta(p_I/q_I)$, if it exists, can intersect at most $2$ intervals of the partitioning of $I$. Hence, by removing $M=2$ intervals of the partitioning of each $I$ within $E_{n-1}$ we construct $E_n$ while ensuring that \eqref{vbc} is satisfied. By Lemma~\ref{lem9+1}, it follows that for any $R\ge4$
$$
\dim\cK\ge \frac{\log(R-2)}{\log R}\,.
$$

\noindent Now take any $x\in\cK$ and any $p/q\in\Q$. Then $p/q\in Q_n$ for some $n\in\N$ and since $\cK\subset E_n$ we have that $x\in E_n$. Then, by \eqref{vbc}, we have that $x\not\in\Delta(p/q)$, which implies that
\begin{equation}\label{vbt}
\left|x-\frac pq\right|\ge \delta R^{-n}\ge \delta R^{-3}q^{-2}\,.
\end{equation}
Since $p/q\in\Q$ is arbitrary and $R$ and $\delta$ are fixed, we have that $x\in\Bad$. That is,
$\cK\subset\Bad$ and thus it follows that
$$\dim\Bad  \, \ge  \,    \dim \cK   \, \ge \,  \frac{\log(R-2)}{\log R}  \, . $$
This is true for any $R \ge 4$ and so on letting $R\to\infty$, it follows that $\dim\Bad  \ge 1$.   The complementary upper bound statement $\dim\Bad  \le 1$ is trivial since $ \Bad \subset \R$.
\end{proof}

\bigskip
\bigskip

\begin{rem}
The crucial property underpinning the  proof of Theorem~\ref{t7.1} is the separation property \eqref{vbvb+} of rationals. Indeed, without appealing to Lemma~\ref{lem9+1}, the above proof  based on \eqref{vbvb+} alone shows that $\Bad$ is uncountable.  The construction of the  Cantor set $\cK$ as well as the proof of Theorem~\ref{t7.1} can be generalised to higher dimensions in order to show that
 $$
 \dim \bad(i_1, \ldots,i_n)  = n \, .
 $$
 Regarding the higher dimensional generalisation of the proof of  Theorem~\ref{t7.1},    the  appropriate analogue of \eqref{vbvb+} is the following elegant
Simplex Lemma --  see  for example \cite[Lemma~4]{Kristensen-Thorn-Velani-06:MR2231044}.
\end{rem}

\medskip

\begin{lem}[Simplex Lemma]
  Let $m \geq 1$ be an integer and $Q > 1$ be a real number.  Let $E
  \subseteq \mathbb{R}^m$ be a convex set of $m$-dimensional Lebesgue
  measure
  \begin{equation*}
    \vert E \vert \ \leq \ (m! \, )^{-1} Q^{-(m+1)} \ .
  \end{equation*}
  Suppose that $E$ contains $m+1$ rational points $(p_i^{(1)}/q_i,
  \ldots, p_i^{(m)}/q_i)$ with $1 \leq q_i < Q$, where $0 \leq i \leq
  m$. Then these rational points lie in some hyperplane of $\R^m$.
\end{lem}

\subsection{Schmidt's games\index{Schmidt's games} \label{sgrt}}

In  his pioneering work \cite{Schmidt-1966},  Wolfgang M. Schmidt introduced the notion of $(\alpha,\beta)$-games which now bear his name.  These games  are an extremely powerful tool for investigating badly approximable  sets.  The simplified account which we are about to present  is sufficient to bring out the main features of the games.
 \medskip

 Suppose that $0<\alpha<1$ and $0 <  \beta < 1$.   Consider the following game involving the two arch rivals $\cA $yesha  and $\cB $hupen --  often simply referred to as players \textbf{A} and \textbf{B}.  First, \textbf{B}  chooses a  closed ball $\cB_0 \subset \RR^m$. Next,  \textbf{A} chooses a closed ball  $\cA_0$ contained in $\cB_0$ of diameter  $ \alpha \, \rho( \cB_0 ) $ where $\rho(\ .\ )$ denotes the diameter of the ball under consideration.  Then,  \textbf{B}  chooses at will a closed ball $\cB_1$ contained in $\cA_0$ of diameter  $ \beta \, \rho( \cA_0 )$.  Alternating in this manner between the two players, generates  a nested sequence of closed balls  in $\RR^m$:
\begin{equation}\label{x33+}
\cB_0\supset \cA_0\supset \cB_1 \supset \cA_1 \supset\ldots \supset \cB_n\supset \cA_n\supset \ldots
\end{equation}
with diameters
$$
\rho(\cB_n) \, = \,   (\alpha \, \beta )^{n}  \,   \rho(\cB_0)  \text{\quad and \quad} \rho(\cA_n) \, =  \,   \alpha \,   \rho(\cB_n)      \, .
$$
A subset  $X$ of $\RR^m$ is said to be {\em $(\alpha,\beta)$-winning} if  $\cA$  can play in such a way that the unique point of the intersection
$$
\bigcap_{n=0}^{\infty}   \cB_n   \,  = \, \bigcap_{n=0}^{\infty}   \cA_n
$$
lies in $X$, regardless of how  $\cB$ plays.  The set $X$ is called {\em
$\alpha$-winning} if  it is  $(\alpha,\beta)$-winning for all
$\beta\in (0,1)$. Finally, $X$ is simply called {\em winning} if it
is $\alpha$-winning for some~$\alpha$.  Informally, player $\cB$ tries
to stay away from the `target' set $X$ whilst player $\cA$ tries to land
on $X$. As shown by Schmidt in \cite{Schmidt-1966}, the following are the key consequences of winning.
{\em \begin{itemize}
\item If $X \subset \RR^m$ is a winning set, then   $\dim X=m$.
\item The intersection of countably many  $\alpha$-winning sets is $\alpha$-winning.
\item If $X\subset\R^m$ is winning and $f$ is a $C^1$ diffeomorphism of $\R^m$ into itself, then $f(X)$ is winning.
\end{itemize}  }

\noindent Schmidt \cite{Schmidt-1966}   proved the following fundamental result for the symmetric case of the higher dimensional analogue of $\Bad$ which, given the above properties, has implications well beyond simply full dimension.

\medskip

\begin{thm}[Schmidt, 1966] \label{sch66}
For any $m\in\N$, the set $\Bad(\frac1m, \ldots,\frac1m)$ is winning.
\end{thm}

\begin{proof}
To illustrate the main ideas involved in proving the theorem we shall  restrict our attention to when  $m=1$.  In this case, we  are able to establish the desired winning statement by naturally  modifying the proof of  Theorem~\ref{t7.1}. Without loss of generality,  we can restrict $\Bad:= \Bad(1)$ to the unit interval $[0,1]$. Let $0<\alpha<\tfrac12$ and $0<\beta<1$. Let $R=(\alpha\beta)^{-1}$ and define $Q_n$ by \eqref{Q_n}.
Again $Q_0=Q_1=Q_2=\varnothing$; the sets $Q_n$ are disjoint; \eqref{Q} and \eqref{vbvb+} are both true.
Furthermore, for $p/q\in Q_n$ the corresponding dangerous interval $\Delta(p/q)$ is defined by
\eqref{Dang}, where $0<\delta<1$ is to be specified  below and will be dependent on $\alpha$ and the first move made by $\cB $hupen .

\noindent Our goal is to show that $\cA $yesha   has a strategy to ensure that sequence \eqref{x33+} satisfies
  \begin{equation}\label{vbc+}
    \textbf{A}_n\cap\Delta(p/q)=\varnothing\qquad\text{for all }~p/q\in Q_n\,.
  \end{equation}
Then the single point $x$ corresponding to  the intersection over  all the closed and nested intervals $\textbf{A}_n$ would satisfy
\eqref{vbt}
for all $p/q\in\Q$   meaning that $x$ is badly approximable.   By definition,  this would implying that $\Bad$ is $\alpha$-winning as desired.

Let $\cB_0 \subset [0,1]$ be any closed interval. Now we set
$$
\delta:=\rho(\cB_0)(\tfrac12-\alpha).
$$
Suppose that
$$
\cB_0\supset \cA_0\supset \cB_1 \supset \cA_1 \supset\ldots \supset \cB_{n-1}\supset \cA_{n-1}
$$
are already chosen and satisfy the required properties; namely   \eqref{vbc+}. Suppose that $\cB_n\subset\cA_{n-1}$ is any closed interval of length
$$
\rho(\cB_n)=\beta\rho(\cA_{n-1})=(\alpha \, \beta )^{n}  \,   \rho(\cB_0) = R^{-n} \rho(\cB_0).
$$
Next, \textbf{A} has to choose a closed interval $\cA_n$ contained in $\cB_n$ of diameter
$$
\rho(\cA_n)=\alpha \, \rho( \cB_n )=\alpha R^{-n} \rho(\cB_0)
$$
and satisfying \eqref{vbc+}. If \eqref{vbc+} is satisfied with $\cA_n$ replaced by $\cB_n$, then choosing $\cA_n$ obviously represents no problem. Otherwise, using \eqref{vbvb+} one readily verifies that  there is exactly one point $p_n/q_n\in Q_n$ such that $\Delta(p_n/q_n)$ intersects $\cB_n$. In this case $\cB_n\setminus\Delta(p_n/q_n)$ is either
the union of two closed intervals, the larger one being of length
$$
\ge \tfrac12\Big(\rho(\cB_n)-\rho(\Delta(p_n/q_n))\Big)=
\tfrac12R^{-n}\Big(\rho(\cB_0)-2\delta \Big)= \alpha R^{-n} \rho(\cB_0)=\alpha\rho(\cB_n)
$$
or a single closed interval of even greater length.
Hence, it is possible to choose a closed interval $\cA_n\subset\cB_n\setminus\Delta(p_n/q_n)$ of length
$\rho(\cA_n)=\alpha\rho(\cB_n)$. By construction, \eqref{vbc+} is satisfied, thus proving the existence of a winning strategy for $\cA$.
\end{proof}

\bigskip

\begin{rem} For various reasons, over the last decade or so there has been an explosion of interest in Schmidt's games. This  has given rise to several ingenious  generalisations of the original game leading to stronger notions of winning, such as modified winning, absolute winning, hyperplane winning and potential winning. For details see \cite{FSU,Kleinbock-Weiss} and references within.
\end{rem}

\medskip

 The framework of Schmidt games and thus the notion of winning  is defined in terms of balls. Thus,  it  is naturally applicable when considering the symmetric case ($i_1= \ldots=i_n= 1/n$) of the badly approximable sets $\bad(i_1 \ldots,i_n) $. Recall, that in the symmetric case, points in $\Bad(\frac1n, \ldots,\frac1n)$   avoid  squares (which are essentially  balls) centred around rational points were as in the general case the points  avoiding rectangles (far from being balls).     We now turn our attention to the general case. Naturally, it would be desirable to be able to show that the general set $\bad(i_1 \ldots,i_n) $ is  winning.

\subsection{Properties of  general $\bad(i_1 \ldots,i_n) $ sets  beyond full dimension}  \label{bey}

Despite the fact that the sets $\Bad(i_1,\dots,i_n)$ have long been know to be uncountable and indeed of full dimension, see \cite{dav, Kleinbock-Weiss, Kristensen-Thorn-Velani-06:MR2231044, Pollington-Velani-02:MR1911218}, the following conjecture of Schmidt dating back to 1982 remained unresolved until reasonably recently.

\begin{thschmidt}
$$ \Bad(\tfrac13,\tfrac23)\cap\Bad(\tfrac23,\tfrac13)\neq\varnothing  \, . $$
\end{thschmidt}

\noindent As is already highlighted in Remark~\ref{4.2}, if false then it would imply that Littlewood's Conjecture is true.

 Schmidt's Conjecture was proved  in  \cite{Badziahin-Pollington-Velani-Schmidt} by establishing the following stronger statement regarding the intersection of $\Bad(i,j)$ sets  with vertical lines $ {\rm L}_{\alpha} := \{(\alpha,y):y\in\R\}\subset\R^2 $.  To some extent it represents the badly approximable analogue of the `fiber' results that appeared  in  \S\ref{KonFibers}.
 \begin{thm}\label{BMbv}
Let  $(i_k,j_k)$ be a countable sequence of  non-negative   reals  such that  $i_k+j_k=1$ and let $ i:= \sup\{i_k: k \in \N \}$.  Suppose that
\begin{equation}\label{eee02}
\liminf_{k\to\infty}\min\{i_k,j_k\}>0   \, .
\end{equation}
Then, for any $ \alpha \in \R$  such that $\liminf\limits_{q \to \infty } q ^{1/i} \| q \alpha \| > 0 $,  we have that
\begin{equation}\label{dbsv}
\textstyle\dim\bigcap_k\Bad(i_k,j_k)\cap {\rm L}_{\alpha}=1.
\end{equation}
\end{thm}

 \medskip

\begin{rem}
The Diophantine condition imposed on $\alpha $ associated with the vertical line ${\rm L}_{\alpha}$ is easily seen to be necessary -- see \cite[\S1.3]{Badziahin-Pollington-Velani-Schmidt}.  Note that the condition is automatically satisfied if $ \alpha \in \Bad$.  On the other hand, condition \eqref{eee02} is present for technical reason and can be removed -- see Theorem \ref{anFAB} and discussion below.   At the point, simply observe that it is automatically satisfied for any finite collection of pairs $(i_k,j_k)$ and thus Theorem~\ref{BMbv} implies Schmidt's Conjecture.  Indeed, together with a standard `slicing' result from fractal geometry one obtains the following full dimension  statement -- see  \cite[\S1.2]{Badziahin-Pollington-Velani-Schmidt} for details.

\end{rem}

\begin{cor}\label{BMbvcor}
Let  $(i_k,j_k)$ be a countable sequence of  non-negative   reals  such that  $i_k+j_k=1$ and  satisfying condition \eqref{eee02}.
Then,
\begin{equation}\label{dbsvsv}
\textstyle\dim\bigcap_k\Bad(i_k,j_k)=2.
\end{equation}
\end{cor}

\medskip

\noindent At the heart of establishing   Theorem \ref{BMbv}  is the `raw' construction of the  generalised Cantor sets framework   formulated in \cite{Badziahin-Velani-MAD}.  For the purposes of these notes, we opt to  follow the framework of \emph{Cantor rich sets}\index{Cantor rich sets} introduced in  \cite{BadM} which is a variation of the   aforementioned  generalised Cantor sets.

\bigskip

Let $R\ge3$ be an integer. Given a collection $\cI$ of compact intervals in $\R$, let $\tfrac1R\cI$ denote the collection of intervals obtained by splitting each interval in $\cI$ into $R$ equal closed subintervals with disjoint interiors. Given a compact interval $I_0\subset\R$, the sequence $(\cI_q)_{q\ge0}$ such that
$$
\cI_0=\{I_0\} \qqand\cI_q\subset\tfrac1R\cI_{q-1}\quad\text{for }q\ge1
$$
is called \emph{an $R$-sequence in $I_0$}. It defines the corresponding \emph{generalised Cantor set}:
\begin{equation}\label{gCs}
\cKI\ :=\ \bigcap_{q\ge0}\hspace*{1ex} \bigcup_{I_q\in\cI_q}I_q.
\end{equation}
Given $q\in\N$ and any interval $J$, let
$$
\wcI_q\  := \ \big(\tfrac1R\cI_{q-1}\big)\setminus \cI_q
\qqand
\wcI_q\sqcap J\  := \ \{I_q\in\wcI_q:I_q\subset J\}\,.
$$
Furthermore,  define
\begin{equation}\label{defl}
   d_q(\cI_q) \ :=\ \min_{\{\wcI_{q,p}\}}\ \ \sum_{p=0}^{q-1} \ \left(\frac4R\right)^{q-p} \max_{I_p\in\cI_p}\#\big(\wcI_{q,p}\sqcap I_p\big)\ ,
\end{equation}
where the minimum is taken over all partitions $\{\wcI_{q,p}\}_{p=0}^{q-1}$ of $\wcI_q$; that is $\wcI_q=\bigcup_{p=0}^{q-1}\wcI_{q,p}$.

The following dimension statement was established in
\cite[Theorem~4]{Badziahin-Velani-MAD}, see also \cite[Theorem~5]{BadM}.

\medskip

\begin{lem}\label{qwert}
Let $\cKI$ be the Cantor set given by \eqref{gCs}. Suppose that
\begin{equation}\label{dp}
d_q(\cI_q) \le 1
\end{equation}
for all $q\in\N$. Then
$$
\dim\cKI\ge1-\frac{\log 2}{\log R}\,.
$$
\end{lem}

\medskip

\noindent Although the lemma  can be viewed as a generalisation of Lemma~\ref{lem9+1}, we stress that its  proof is substantially  more involved and requires new ideas. At the heart of the proof  is  the `extraction'   of a `local' Cantor type subset $\cK$ of  $\cKI$.  By a local Cantor set we mean a set arising from a construction as described  in \S\ref{Jar}.    The  parameter  $M$ associated with the  extracted local Cantor  set $\cK$ is essentially$ \tfrac12R$.

It is self evident from Lemma \ref{qwert}, that if a given set $X\subset\R$ contains a  generalised Cantor set given by \eqref{gCs} with arbitrarily large $R$, then $\dim X=1$. The following  definition of Cantor rich \cite{BadM},  imposes a stricter requirement than \eqref{dp} in order to ensure  that the  countable intersection of generalised Cantor sets is of full dimension.
To some extent, building upon the raw construction of \cite[\S7.1]{Badziahin-Pollington-Velani-Schmidt},    the full dimension aspect for countable intersections had previously been investigated in \cite[\S7]{Badziahin-Velani-MAD}.


\medskip

\begin{df}\label{crs}
Let $M>1$, $X\subset \R$ and $I_0$ be a compact interval. The set $X$ is said to be \emph{$M$-Cantor rich in $I_0$} if for any $\ve>0$ and any integer $R\ge M$ there exists an $R$-sequence $(\cI_q)_{q\ge0}$ in $I_0$ such that $\cKI\subset X$ and $$\sup_{q\in\N}d_q(\cI_q)\le \ve  \, .  $$ The set $X$ is said to be \emph{Cantor rich in $I_0$} if it is \emph{$M$-Cantor rich in $I_0$} for some $M$, and it is said to be \emph{Cantor rich} if it is \emph{Cantor rich in $I_0$} for some compact interval $I_0$.
\end{df}

\medskip

\noindent  The following summarises  the key properties of Cantor rich sets.
{\em \begin{itemize}
\item[(i)] Any Cantor rich set $X$ in $\R$ satisfies \ $\dim X=1$.
\item[(ii)]
For any given compact interval $I_0$ and any given fixed $M\in\N$,
any countable intersection of $M$-Cantor rich sets in $I_0$ is $M$-Cantor rich in $I_0$.
\end{itemize}}

The framework of Cantor-rich sets  was utilised in the same paper \cite{BadM} to  establish  the following result concerning badly approximable points  on manifolds.

\begin{thm}\label{BM}
For any non-degenerate analytic sub-manifold $\cM\subset\R^n$ and any sequence $(i_{1,k},\dots,i_{n,k})$ of non-negative reals such that
$i_{1,k}+\dots+i_{n,k}=1$ and
\begin{equation}\label{vba}
\inf\{i_{j,k}>0:1\le j\le n,~k\in\N\}>0   \, ,
\end{equation}
one has that
\begin{equation}\label{vb}
\dim \textstyle{\bigcap_{k}} \Bad(i_{1,k},\dots,i_{n,k})   \, \cap \, \cM =\dim\cM\,.
\end{equation}
\end{thm}

The condition of analyticity  from Theorem \ref{BM} can be omitted  in the case the sub-manifold  $\cM\subset\R^n$ is a curve.   Indeed, establishing  the theorem for curves is very much the crux since any manifold can be `fibred' into an appropriate collection of curves  -- see \cite[\S2.1]{BadM} for details.  In the case $n=2$, so that $\cM $ is a non-degenerate planar curve, the theorem  was  previously  established  in \cite{Badziahin-Velani-Dav} and provides a solution to an explicit  problem  of
Davenport dating back to the swinging sixties concerning the existence of badly approximable pairs on the parabola.  Furthermore, in \cite{Badziahin-Velani-Dav} partial results  for  lines  (degenerate curves)  with slopes satisfying  certain Diophantine constraints are also obtained. Although not optimal, they  naturally extend Theorem \ref{BMbv}  beyond vertical lines.  As already mentioned, Theorem~\ref{BM} as stated for general $n$ was established in \cite{BadM} and it settles the natural generalisations of Schmidt's  Conjecture and  Davenport's problem in arbitrary dimensions.

\bigskip

\begin{rem}
Building upon the one-dimensional, generalised Cantor sets framework formulated in \cite{Badziahin-Velani-MAD}, an abstract `metric space' framework of higher dimensional generalised Cantor sets, branded as `Cantor winning sets'\index{Cantor winning sets}, has recently been  introduced in \cite{BadHar}. Projecting this framework onto the specific one-dimensional construction of Cantor rich sets given above,  the definition of Cantor-winning sets reads as follows.
Let $\ve_0>0$, $X\subset \R$ and $I_0$ be a compact interval. Then the set $X$ is \emph{$\ve_0$-Cantor-winning in $I_0$} if for any positive $\ve<\ve_0$ there exists a positive integer $R_\ve$ such that for any integer $R\ge R_\ve$ there exists an $R$-sequence $(\cI_q)_{q\ge0}$ in $I_0$ such that $\cKI\subset X$ and
$$
\max_{I_p\in\cI_p}\#\big(\wcI_{q,p}\sqcap I_p\big)\le R^{(q-p)(1-\ve)}\,.
$$
The latter key condition implies that $d_q(\cI_q)$ is no more than $8R^{-\ve}$
provided that  $8R^{-\ve}<1$.
Most recently, David Simmons has shown that the notion of Cantor winning as defined in \cite{BadHar} is equivalent to the notion of  potential winning as defined in \cite{FSU}.
\end{rem}

\medskip

The use of Cantor rich sets in establishing  statements  such as Theorems \ref{BMbv} $\&$  \ref{BM},  comes at a cost of having to impose, seemingly for technical reasons,  conditions such as \eqref{eee02} and \eqref{vba}. Although delivering some additional benefits, unfortunately the framework of Cantor winning sets described above  does not seem to resolve this issue.  However, if for example,  we could show that $\Bad(i_1,\dots,i_n)$ is (Schmidt) winning,  then we would be able to intersect countably many such sets without imposing any technical conditions.  When $n=2$,  this has been  successfully  accomplished by  Jinpeng An in his elegant paper \cite{An-12-2}.

\begin{thm}[J. An]\label{anFAB}  For any pair of non-negative  reals $(i,j)$  such that  $i+j=1$, the two-dimensional set $\Bad(i,j)$ is winning.
\end{thm}

\noindent A simple consequence of this is that we can remove condition \eqref{eee02} from the statement of Corollary \ref{BMbvcor}.   Prior to \cite{An-12-2},  it is important to  note that  An in \cite{An-12} had shown that $\Bad(i,j) \cap {\rm L}_{\alpha} $ is winning, where ${\rm L}_{\alpha}$ is a vertical line  as in Theorem  \ref{BMbv}.   Of course, this implies that Theorem  \ref{BMbv} is true without imposing condition \eqref{eee02}.   On combining the ideas and techniques introduced in  the papers \cite{An-12,Badziahin-Velani-Dav,BadM}, it is shown  in \cite{An-Beresnevich-Velani} that  $\Bad(i,j) \cap \cC$ is winning,  where $\cC$ is a non-degenerate planar curve. This implies that we can  remove condition \eqref{vba} from  the $n=2$ statement of Theorem \ref{BM}.
In higher dimensions ($n>2$),  removing condition \eqref{vba} remains very much a key open problem.  The recent work of Guan and Yu \cite{GY} makes a contribution toward this problem.  Building upon the work of An \cite{An-12-2}, they show that the set $\Bad(i_1,\dots,i_n)$ is winning whenever $i_1=\dots=i_{n-1}\ge i_n$.

\bigskip

So far we have discussed the homogeneous theory of badly approximable sets. We now turn our attention to the inhomogeneous  theory.

\subsection{Inhomogeneous badly approximable points }

Given $\theta \in \R$ the natural inhomogeneous generalisation of the one-dimensional set $\bad$ is the set
$$
\bad(\theta)  := \{ x \in \R :  \exists \ c(x)>0  \ { \rm \ so \  that  \ }
 \|qx - \theta \| \ > \ c(x) \ q^{-1}  \quad \forall \  q \in \N  \}  \,.
$$
Within these  notes we shall prove the following  inhomogeneous strengthening of Theorem \ref{t7.1}.

\medskip

\begin{thm} \label{ttt}
For any $\theta\in\R$, we have that
  $$
  \dim\Bad(\theta)=1\,.
  $$
\end{thm}

The basic philosophy behind the proof is simple and exploits the already discussed homogeneous `intervals construction'; namely
$$
\mbox{(homogeneous construction)} \quad   +  \quad   (\bm\theta - \bm\theta = \bm0)    \quad  \Longrightarrow    \quad \mbox{(inhomogeneous statement).}
$$

\medskip

\begin{rem}
Recall that we have already made use of this type of philosophy in establishing the inhomogeneous extremality conjecture stated in \S\ref{ITPtheory}, where the proof very much relies on the fact that we already know that any non-degenerate manifold is (homogeneously) extremal.
\end{rem}

\begin{proof}[Proof of Theorem~\ref{ttt}]
Let $R\ge4$ be an integer and $\delta=\tfrac12$.
For $n\in\Z$, $n\ge0$, define the sets $Q_n$ by \eqref{Q_n} and additionally
define the following sets of `shifted' rational points
\begin{equation}\label{Q_n+x2}
  Q_n(\theta)=\{(p+\theta)/q\in\R:p,q\in\Z,~R^{\frac{n-5}2}\le q<R^{\frac {n-4}2}\}\,.
\end{equation}
Clearly, $Q_0(\theta)=\dots=Q_4(\theta)=\varnothing$ and the union
$Q(\theta):=\bigcup_{n=5}^\infty Q_n(\theta)$
contains all the possible points $(p+\theta)/q$ with $p,q\in\Z$, $q>0$.

Next,
for $p/q\in Q_n$ define the dangerous interval
$\Delta(p/q)$ by \eqref{Dang} and additionally define the inhomogeneous family of dangerous intervals given by
\begin{equation}\label{Dang+}
\Delta((p+\theta)/q):=\Big\{x\in[0,1]:\left|x-\frac {p+\theta}q\right|<\delta R^{-n}\Big\}\,,
\end{equation}
where $(p+\theta)/q\in Q(\theta)$.
With reference to the Cantor construction of \S\ref{Jar}, our goal is to construct
a Cantor set  $
\cK=\bigcap_{n=0}^\infty E_n
$ such that for every $n\in\N$
  \begin{equation}\label{vbc+x}
    E_n\cap\Delta(p/q)=\varnothing\qquad\text{for all }~p/q\in Q_n\qquad
  \end{equation}
  and simultaneously
  \begin{equation}\label{vbc+x2}
    E_n\cap\Delta((p+\theta)/q)=\varnothing\qquad\text{for all }~(p+\theta)/q\in Q_n(\theta)\,.
  \end{equation}
To this end, let $E_0=[0,1]$ and suppose that $E_{n-1}$ has been constructed as required. Let $I$ be any interval within $E_{n-1}$. Then $|I|=R^{-n+1}$. When constructing $E_n$, $I$ is partitioned into $R$ subintervals. We need to decide how many of these subintervals have to be removed in order to satisfy \eqref{vbc+x} and \eqref{vbc+x2}. As was argued in the proof of Theorem~\ref{t7.1}, removing $2$ intervals of the partitioning of $I$ ensures that \eqref{vbc+x} is satisfied.
We claim that the same applies to \eqref{vbc+x2}, that is removing $2$ intervals of the partitioning of $I$ ensures \eqref{vbc+x2}. Indeed, since the length of $\Delta((p+\theta)/q)$ is no more that $R^{-n}$, to verify this claim it suffices to show that there is only one point $(p+\theta)/q\in Q_n(\theta)$ such that
$$
\Delta((p+\theta)/q)\cap I\neq\varnothing.
$$
This condition implies that
\begin{equation}\label{zzz}
|qx-p-\theta|<R^{\frac {n-4}2}(\delta R^{-n}+R^{-n+1})\qquad\text{for any }x\in I\,.
\end{equation}
For a contradiction, suppose there are two distinct points $(p_1+\theta)/q_1$ and $(p_2+\theta)/q_2$ in $Q_n(\theta)$ satisfying \eqref{zzz}. Then, by \eqref{zzz} and the triangle inequality, we get that
\begin{equation}\label{zzz+}
|(q_1-q_2)x-(p_1-p_2)|<2R^{\frac {n-4}2}(\delta R^{-n}+R^{-n+1})\qquad\text{for any }x\in I\,.
\end{equation}
Clearly $q_1\neq q_2$ as otherwise we would have that $|p_1-p_2|<2R^{\frac {n-4}2}(\delta R^{-n}+R^{-n+1})<1$, implying that $p_1=p_2$ and contradicting to the fact that $(p_1+\theta)/q_1$ and $(p_2+\theta)/q_2$ are distinct. In the above we have used that $n\ge5$. Also without loss of generality we assume that $q_1>q_2$.
Then define $d=\gcd(q_1-q_2,p_1-p_2)$, $q=(q_1-q_2)/d$, $p=(p_1-p_2)/d$ and let $m$ be the unique integer such that
$$
p/q\in Q_m\,.
$$
Thus, $R^{\frac{m-3}2}\le q<R^{\frac{m-2}2}$.
Since $q<q_1<R^{\frac {n-4}2}$ we have that $m\le n-2$. Then, by \eqref{zzz+},
\begin{equation}\label{zzz++}
\left|x-\frac pq\right|<R^{-\frac {m-3}2}2R^{\frac {n-4}2}(\delta R^{-n}+R^{-n+1})\le \delta R^{-m}
\end{equation}
for any $x\in I$ provided that $R\ge36$ (recall that $\delta=\tfrac12$). It means that $\Delta(p/q)\cap I\neq\varnothing$. But this is impossible since \eqref{vbc+x} is valid with $n$ replaced by $m$ and $I\subset E_{n-1}\subset E_m$. This proves our above claim. The upshot is that by removing $M=4$ intervals  of the partitioning of each $I$ within $E_{n-1}$  we construct $E_n$ while ensuring that the  desired  conditions \eqref{vbc+x} and \eqref{vbc+x2} are satisfied.   The finale  of the proof makes use of Lemma \ref{lem9+1} and is almost identical to that of the proof of Theorem~\ref{t7.1}.  We leave the details to the  reader.
\end{proof}

\bigskip

\begin{rem}
Note that in the above proof of Theorem~\ref{ttt}, we actually show that
$$
\dim\bad\cap\bad(\theta)=1\,.
$$
It seems that proving this stronger  statement  is   simpler  than any potential `direct' proof of the implied fact that  $\dim\bad(\theta)=1 $.
\end{rem}

\medskip

\begin{rem}
In the same way that the proof of Theorem \ref{t7.1} can be modified to show that $\Bad$ is winning (see the proof of Theorem \ref{sch66} for the details),  the proof of Theorem~\ref{ttt} can be adapted to show that
$
\bad(\theta)
$
is winning.
\end{rem}

\medskip

  In higher dimensions,  the natural generalisation of the one-dimensional set $\bad(\theta)$  is the set  $\bad(i_1, \ldots,i_n; \bm\theta)$  defined in the following manner. For any $\bm\theta = (\theta_1, \ldots, \theta_n)  \in \R^n$ and  $n$-tuple of real numbers $i_1, ...,i_n \geq 0 $ such
that $i_1+\dots+i_n = 1 $, we let $\bad(i_1, \ldots,i_n; \bm\theta)$ to be the set of
points $(x_1, ...,x_n) \in \R^n $ for which  there exists a
positive constant $ c(x_1, ...,x_n)$ such that   $$ \max \{ \;
||qx_1  - \theta_1||^{1/i_1} \; , ..., \ ||qx_n   -   \theta_n||^{1/i_n} \,  \} \ > \
c(x_1, ...,x_n) \ q^{-1} \ \ \ \forall \ \ \ q \in \N . $$
The ideas used in  the proof of Theorem~\ref{ttt} can be naturally generalised to show that
$$
\dim \bad(i_1 \ldots,i_n;\bm\theta)  = n  \, .
$$
In the case $n=2$, the details of the proof are explicitly given in \cite[\S3]{BVdani}.    Indeed, as mentioned in  \cite[Remark 3.4]{BVdani},  in the symmetric case $ i_1=\ldots=i_n=1/n$,  we actually have that $\Bad(\frac1n, \ldots,\frac1n;\bm\theta
)$ is winning; i.e. the inhomogeneous strengthening  of Theorem \ref{sch66}.

\bigskip

\begin{rem}
The basic philosophy exploited in proving Theorem~\ref{ttt} has been successfully  incorporated within the context of Schmidt games to establish the inhomogeneous generalisation of the homogeneous winning statements discussed at the end of \S\ref{bey}. In particular, let $ \bm\theta \in \R^2$ and $(i,j)$ be  a pair  of non-negative  real numbers such that  $i+j=1$.  Then,  it is shown  in \cite{An-Beresnevich-Velani} that (i) the set $ \bad(i,j;\bm\theta)  $  is winning and (ii) for any non-degenerate planar curve  $\cC$, the set $ \bad(i,j;\bm\theta) \cap \cC $  is winning.   Also, in \cite{An-Beresnevich-Velani} the following almost optimal winning result for the intersection of $\Bad(i,j)$ sets with arbitrary lines (degenerate curves) is obtained.  It substantially extends and generalises the previous `line' result  obtained  in \cite{Badziahin-Velani-Dav}.
\end{rem}

\begin{thm}\label{jpanvb}
Let  $(i,j)$ be a pair of non-negative real numbers such that  $i+j=1$ and
given $a,b\in \RR$ with $a \neq 0$, let ${\rm L}_{a,b} $ denote the  line defined
by the equation $y = a x +b $.   Suppose there exists
$\epsilon>0$ such that
\begin{equation} \label{diocondbetter}
\liminf_{q \to \infty}   q^{\frac{1}{\sigma}-\epsilon}  \max \{\|q a
\|, \|q b
\|  \}
> 0 \, \qquad{\rm where } \ \  \sigma := \min\{i,j\} \, .
\end{equation}
Then, for any $\bm\theta \in\RR^2$ we have that   $ \Bad_{\bm\theta} (i,j) \cap  {\rm L}_{a,b} $ is  winning. Moreover, if $ a \in \QQ$ the statement is true with $ \epsilon=0$ in  (\ref{diocondbetter}).
\end{thm}

The condition \eqref{diocondbetter} is optimal up to the $\epsilon$ -- see \cite[Remark~4]{An-Beresnevich-Velani}. It is indeed, both necessary and sufficient in the case $a\in\Q$. Note that the argument presented in \cite[Remark~4]{An-Beresnevich-Velani} showing the necessity of \eqref{diocondbetter} with $\epsilon=0$ only makes use of the assumption that $\Bad(i,j) \cap  {\rm L}_{a,b} \not=\varnothing$. It is plausible to suggest that this latter assumption  is a necessary and sufficient condition for the conclusion of Theorem~\ref{jpanvb} to hold.

\bigskip

\begin{conjecture}
Let  $(i,j)$ be a pair of non-negative real numbers such that  $i+j=1$ and
given $a,b\in \RR$ with $a \neq 0$, let ${\rm L}_{a,b} $ denote the  line defined
by the equation $y = a x +b $.  Then
$$
\Bad(i,j) \cap  {\rm L}_{a,b} \not=\varnothing
$$
if and only if
$$
\forall\ \bm\theta \in\RR^2\qquad \Bad_{\bm\theta} (i,j) \cap  {\rm L}_{a,b} \quad\text{is  winning.}
$$
\end{conjecture}

\medskip

\noindent Observe that the conjecture is true  in the case  $ a \in \QQ$ and when the line $ {\rm L}_{a,b}$ is
horizontal or vertical in the homogenous case.

\vspace*{8ex}

\noindent{\bf Acknowledgements.}    SV would like to thank the organisers of the 2014 Durham Easter School ``Dynamics and Analytic Number Theory'' for  giving him the opportunity to give a mini-course  --  it was a stimulating and enjoyable experience.  Subsequently, the subject matter of that mini-course formed the foundations for a MAGIC graduate lecture course on metric number theory  given jointly by VB and SV at the University of York in Spring 2015.    We would like to thank the participants  of these  courses for providing valuable feedback on both the lectures and the accompanying notes.   In particular, we thank Demi Allen and  Henna Koivusalo for their detailed comments (well beyond the call of duty) on earlier drafts of this end product.  For certain their input has improved the clarity and the accuracy of the exposition.  Of course, any remaining typos and mathematical errors are absolutely their fault!

\vspace{0ex}

\

{\small

\noindent Victor V. Beresnevich:\\
 Department of Mathematics,
 University of York,\\
 Heslington, York, YO10 5DD,
 England\\
 E-mail: {\tt victor.beresnevich@york.ac.uk}

\bigskip

\noindent Felipe A. Ram\'{\i}rez:\\
 Department of Mathematics and Computer Science,
 Wesleyan University,\\
265 Church Street
Middletown, CT 06459\\
 E-mail: {\tt framirez@wesleyan.edu}

\bigskip

\noindent Sanju L. Velani:\\
 Department of Mathematics,
 University of York,\\
 Heslington, York, YO10 5DD,
 England\\
 E-mail: {\tt sanju.velani@york.ac.uk}

}

\end{document}